\documentclass[12pt,reqno]{article}
\usepackage[usenames]{color}
\usepackage{amssymb}
\usepackage{graphicx}
\usepackage{amscd}
\usepackage[colorlinks=true,
linkcolor=webgreen,
filecolor=webbrown,
citecolor=webgreen,
pdfstartpage=1]{hyperref}

\definecolor{webgreen}{rgb}{0,.5,0}
\definecolor{webbrown}{rgb}{.6,0,0}

\usepackage{color}
\usepackage{float}

\usepackage{graphics,amsmath,amssymb}
\usepackage{amsthm}
\usepackage{amsfonts}
\usepackage{latexsym}

\newcommand{\Aut}{{\rm Aut}}
\newcommand{\Prism}{{\rm Prism}}
\newcommand{\GA}{{\rm GA}}
\newcommand{\Br}{{\rm Br}}
\newcommand{\MBr}{{\rm MBr}}
\newcommand{\Pet}{{\rm Pet}}
\newcommand{\Dod}{{\rm Dod}}
\newcommand{\Syl}{{\rm Syl}}
\newcommand{\AW}{{\rm AW}}
\newcommand{\Cox}{{\rm Cox}}
\newcommand{\Ico}{{\rm Ico}}

\begin{document}

\theoremstyle{plain}
\newtheorem{theorem}{Theorem}
\newtheorem{corollary}[theorem]{Corollary}
\newtheorem{lemma}[theorem]{Lemma}
\newtheorem{proposition}[theorem]{Proposition}
\newtheorem{obs}[theorem]{Observation}

\theoremstyle{definition}
\newtheorem{definition}[theorem]{Definition}
\newtheorem{example}[theorem]{Example}
\newtheorem{conjecture}[theorem]{Conjecture}

\theoremstyle{definition}
\newtheorem{remark}[theorem]{Remark}
\begin{center}
\vskip 1cm

{\LARGE\bf  Tight factorizations of girth-$g$-regular graphs}
\vskip 1cm
\large
Italo J. Dejter\\
University of Puerto Rico\\
Rio Piedras, PR 00936-8377\\
\href{mailto:italo.dejter@gmail.com}{\tt italo.dejter@gmail.com} \\
\end{center}

\begin{abstract}\noindent Girth-regular graphs with equal girth, regular degree and chromatic index are studied for the determination of 1-factorizations with each 1-factor intersecting every girth cycle. Applications to hamiltonian decomposability and to 3-dimensional geometry
are given.Applications are suggested for priority assignment and optimization problems.
 
\end{abstract}

\section{Introduction}\label{intro}

Let $3\le \kappa\in\mathbb{Z}$. Given a finite graph $\Gamma$ with girth $g(\Gamma)=\kappa$ we inquire whether assigning $\kappa$ colors to the edges of $\Gamma$ {\it properly} (i.e., no two adjacent edges of equal color in $\Gamma$) can be performed so that each $\kappa$-cycle of $\Gamma$ has a bijection from the edges of $\Gamma$ to the $\kappa$ colors.
Such an inquiry is applicable to managerial situations in which a number of agents must participate in committees around round tables, with $\kappa$ stools about each table. A roster of $\kappa$ tasks is handed to each agent at each round table, and such an agent is assigned each of the $\kappa$ tasks but at pairwise different tables. This can be interpreted as a 2-way (vertex incidence versus girth-cycle membership) problem of sorting the $\kappa$ colors according to some prioritization hierarchy. 
Thus, an assignment problem is developed with a range of potential applications in geometry (Section~\ref{3d}), optimization and decision making; see for example \cite{s1,s2,s3,s4,s5,s6}. We pass to formalize our ideas.

Let $\Gamma$ be a finite connected $\kappa$-regular simple graph with {\it chromatic index} $\chi'(\Gamma)=\kappa$.
Let $g=g(\Gamma)$ be the girth of $\Gamma$. We say that $\Gamma$ is a $g${\it -tight graph} if $g=g(\Gamma)=\kappa=\chi'(\Gamma)$.
In each $g$-tight graph $\Gamma$ it makes sense to look for a proper edge-coloring via $\kappa$ colors, each girth cycle colored via a bijection between the cycle edges and the colors they are assigned, precisely $\kappa$ colors. We will say that such a coloring is an {\it edge-girth coloring} of $\Gamma$ and, in such a case, that $\Gamma$ is {\it edge-girth chromatic}, or {\it egc} for short.

We focus on $g$-tight $\kappa$-regular graphs that are {\it girth-regular}, a concept whose definition in \cite{PV} we adapt as follows. Let $\Gamma$ be a graph.
Let  $\{e_1, \ldots , e_\kappa\}$ be the set of edges incident in $\Gamma$ to a vertex $v$. Let $(e_i)$ be the number of $\kappa$-cycles containing an edge $e_i$ for $1\le i\le \kappa$. Assume $(e_1)\ge(e_2)\ge\ldots\ge(e_\kappa)$. Let the {\it signature} of $v$ be the $\kappa$-tuple
$((e_1),(e_2),\ldots,(e_\kappa))$. The graph $\Gamma$ is said to be {\it girth-regular} if all its vertices have a common signature. In such a case, the signature of any vertex of $\Gamma$ is said to be the {\it signature} of $\Gamma$.

In this work, girth-regular graphs that are $g$-tight will be said to be {\it girth-$g$-regular graphs} as well as $((e_1),(e_2),\ldots,(e_g))$-{\it graphs}, or $(e_1)(e_2)\cdots (e_g)$-{\it graphs}, if no confusion arises, where $g=\kappa$.
In this notation, a prefix $$a_1^{(1)}a_1^{(2)}\cdots a_1^{(m_1)}a_2^{(1)}a_2^{(2)}\cdots a_2^{(m_2)}\cdots a_t^{(1)}a_t^{(2)}\cdots a_t^{(m_t)}$$ with $a_i^{(j)}=a_{i'}^{(j')}$ iff $i=i'$ may be abbreviated as $a_1^{m_1}a_2^{m_2}\cdots a_t^{m_t}$ where superscripts equal to 1 may be omitted (e.g., $3221$ abbreviates to $32^21$). So,
the prefixes $(e_1)(e_2)\cdots (e_g)$ will include and further be denoted as follows:

\begin{figure}[htp]
\includegraphics[scale=0.87]{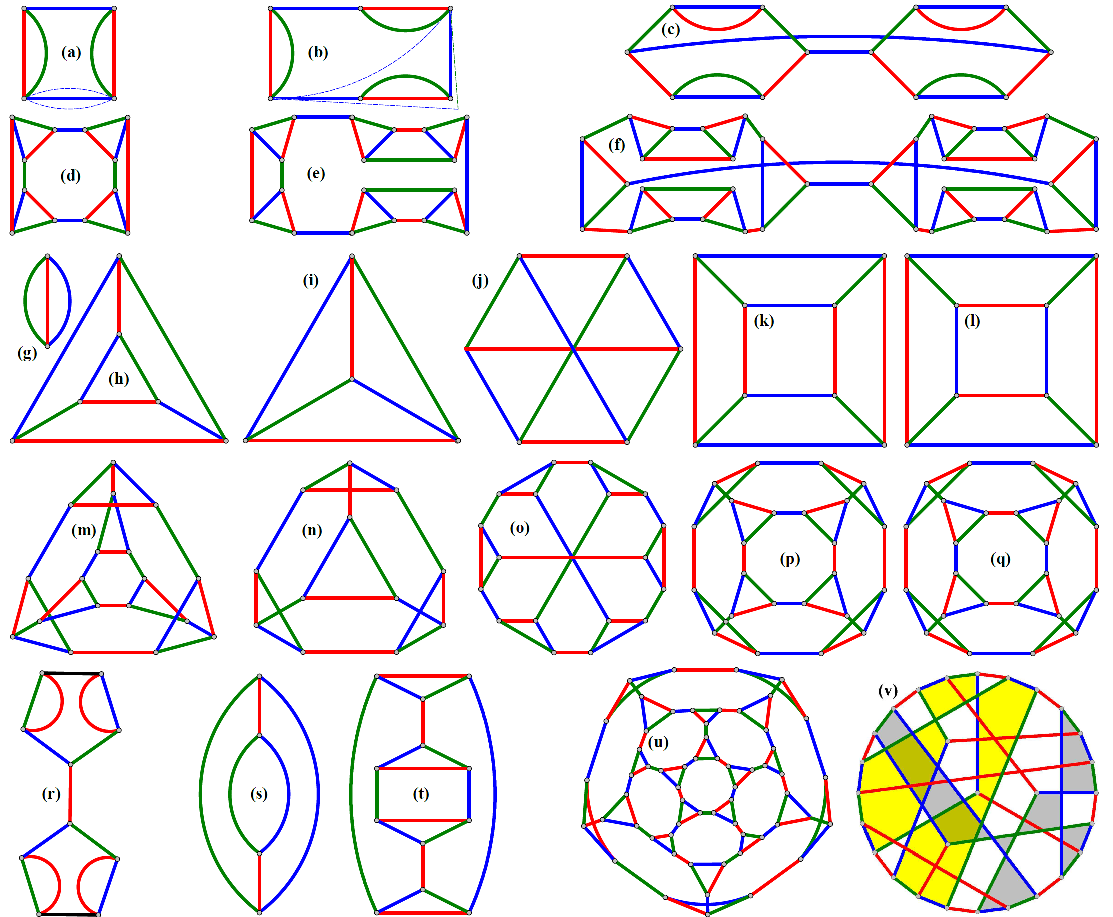}
\caption{Producing egc $(e_1)(e_2)(e_3)$-graphs $\Gamma$ that are triangle-replaced graphs from graphs $\Gamma'$. In particular, pairs $(\Gamma)(\Gamma')$ represented here are: (a)(d), (b)(e), (c)(f), (g)(h), (h)(m), (i)(n), (j)(o), (k)(p), (l)(q) and (s)(t). In additon, item (r) is an example of a generalized snark other than the Petersen graph. Items (u) and (v) are related to the dodecahedral graph and to the Coxeter graph, respectively.}
\label{fig1}
\end{figure}

\begin{enumerate}
\item[\bf{(i)}] $222=2^3$ and $110=1^20$ in Section~\ref{s1}, (Theorems~\ref{2^3} and~\ref{stat});
\item[\bf{(ii)}] $3333=3^4$ $3322=3^22^2$ and $2222=2^4$ in Section~\ref{s2}, (Theorem~\ref{re}, via Lemma~\ref{SC});
\item[\bf{(iii)}] $4443=4^33$ $3221=32^21$ and $3111=31^3$ in Section~\ref{s3}, (Theorem~\ref{prismas});
\item[\bf{(iv)}] $1111=1^4$ in Section~\ref{1111}, (Theorems~\ref{needed} and~\ref{barrels}, via Lemma~\ref{SC}, or a variation of it);
\item[\bf{(v)}] $44400=4^30^2$ $2^30^2$ $8^5$ and $(12)^5$ in Section~\ref{12345}, (Theorems~\ref{arman}, \ref{4^30^2}, \ref{2^30^2} and \ref{10^4}).
\end{enumerate}

\noindent Extending this context, girth-regular graphs $\Gamma$ of regular degree $g$ and girth larger than $g$ will be said to be $0^g$-graphs, or {\it improper} $(e_1)(e_2)\cdots(e_g)$-graphs.

Edge-girth colorings of $(e_1)(e_2)\ldots (e_g)$-graphs are equivalent to 1-factorizations \cite{WDW} such that the cardinality of the intersection of each 1-factor with each girth cycle is 1. These factorizations are said to be {\it tight}, and the resulting colored girth cycles, are said to be {\it tightly colored}.
Note a $\Gamma$ with a tight factorization is egc. Unions of pairs of 1-factors of such graphs are treated in Section~\ref{hd} for their hamiltonian decomposability, (Corollary~\ref{hamil}).
Applications to  M\"obius-strip compounds and hollow-triangle polylinks are found in Section~\ref{3d}.

\section{Egc girth-3-regular graphs}\label{s1}

\begin{theorem}\label{2^3}\cite{PV} There is only one $(e_1)(e_2)(e_3)$-graph $\Gamma$ with $(e_1)(e_2)(e_3)=222=2^3$ namely $\Gamma=K_4$. Moreover, $\Gamma=K_4$ is egc. All other proper $(e_1)(e_2)(e_3)$-graphs are $1^20$-graphs, but not necessarily egc.
\end{theorem}

\begin{proof} For the first sentence in the statement, we refer to item (1) of Theorem 5.1 \cite{PV}. To see that $\Gamma=K_4$ is egc, we refer to Fig.~\ref{fig1}(i), below.
\end{proof}

\begin{remark}\label{multi} In order to determine which $1^20$-graphs are egc, let $\Gamma'=(V',E',\phi')$ be a finite undirected loopless cubic multigraph. Let $e\in E'$ with $\phi'(e)=\{u,v\}$ and $u,v\in V'$. Then, $e$ determines two {\it arcs} (that is, ordered pairs of end-vertices of $e$) denoted $(e;u,v)$ and $(e;v,u)$ (if the girth $g(\Gamma')$ of $\Gamma'$ is larger than 2, then $\Gamma'$ is a simple graph, a particular case of multigraph).
The following definition is an adaptation of a case of the definition of generalized truncation in \cite{Eiben}. Let $A'$ denote the set of arcs of $\Gamma'$. A {\it vertex-neighborhood labeling} of $\Gamma'$ is a function $\rho:
A'\rightarrow\{1,2,3\}$ such that for each $u\in V'$ the restriction of $\rho$ to the set $A'(u)=\{(e;u,v)\in A':e\in E'; \phi'(e)=\{u,v\}; v\in V'\}$
of arcs leaving $u$ is a bijection. For our purposes, we require $\rho(e;u,v)=\rho(e;v,u)$ $\forall e\in E'$ with $\phi'(e)=\{u,v\}$ so that each $e\in E'$ is assigned a well-defined color from the color set $\{1,2,3\}$. This yields a 1-factorization of $\Gamma'$ with three 1-factors that we can call $E'_1,E'_2,E'_3$ for respective color 1, 2, 3, with $E'$ being the disjoint union $E'_1\cup E'_2\cup E'_3$. For the sake of examples in Fig.~\ref{fig1}, to be presented below, let colors 1, 2 and 3 be taken as red, blue and green, respectively.\end{remark}

Let $K_3$ be the triangle graph with vertex set $\{v_1, v_2, v_3\}$. The {\it triangle-replaced} graph $\nabla(\Gamma')$ of $\Gamma'$ with respect to 4$\rho$ has vertex set 
$$\{(e_i;u,v_i) : u \in V'; 1\le i\le 3\}$$ and edge set
$$\{(e_i;u,v_i)(e_j;u, v_j)| v_iv_j \in E(K_3),u\in V'\}\cup\{u,v_{\rho(e;u,w)})(w, v_{\rho(e;w,u)})|e\in E';\phi(e)=\{u,w\}\}.$$
Note that $\nabla(\Gamma')$ is a $1^20$-graph.
We will refer to the edges of the form ($e_i;u,v_i)(e_j;u,v_j)$
as {\it $\nabla$-edges} or {\it triangle edges}, and to the edges $(e_i;u,v_i)(e_j;w,v_j)$ $u\neq w$
as {\it $\Gamma'$-edges} or {\it non-triangle edges}. Observe that a $\Gamma'$-edge is incident only to $\nabla$-edges and that each vertex
of $\nabla(\Gamma')$ is incident to precisely one $\Gamma'$-edge. This yields the following.

\begin{obs} \cite{Eiben} Let $\Gamma'$ be a finite undirected cubic multigraph of girth $g$. Then, for any vertex-neighborhood labeling $\rho$ of $\Gamma'$ the shortest cycle in the triangle-replaced graph $\nabla(\Gamma')$ containing a $\Gamma'$-edge is of length at least $2g$.
\end{obs}

We say that $\Gamma'$ is  a {\it generalized snark} if its chromatic index $\chi'(\Gamma')$ is larger than 3.
Two examples of generalized snark are: {\bf(i)} the Petersen graph and {\bf(ii)} the multigraph obtained by joining two $(2k+1)$-cycles ($k\ge 1$) via an extra-edge (a bridge between the two $(2k+1)$-cycles) and adding $k$ parallel edges to each of the two $(2k+1)$-cycles so that the resulting multigraph is cubic, see Fig.~\ref{fig1}(r) for $k=5$. The triangle-replaced graph $\Gamma'_1=\nabla(\Gamma')$ of a generalized snark $\Gamma'$ will also be said to be a generalized snark. This denomination will also be  used for the triangle-replaced graphs $\Gamma'_{i+1}=\nabla(\Gamma'_i)$ of $\Gamma'_i$ for $i=1,2,\ldots$ etc.
In addition, we will say that $\Gamma'$ is {\it snarkless} if it is not a generalized snark. Clearly, $K_4$ is snarkless.

Vertex-neighborhood labelings $\rho$ for the examples of $\Gamma'$ below, represented in Fig.~\ref{fig1}, have the elements 1, 2 and 3 of $\rho(A')$ interpreted respectively as edge colors red, blue and green. Now, the smallest snarkless multigraphs $\Gamma'\ne K_4$ are:

\begin{enumerate}
\item[{\bf(A)}] the cubic multigraph $\Gamma'_A$ of two vertices and three edges in Fig.~\ref{fig1}(g), with $\nabla(\Gamma'_A)$ being the triangular prism $\Prism(K_3)=K_2\square K_3$ in Fig.~\ref{fig1}(h), where $V(K_2)=\{0,1\}$ and $\square$ stands for the graph cartesian product \cite{Imrich};
\item[{\bf(B)}] the cubic multigraph $\Gamma'_B$ of four vertices resulting as the edge-disjoint union of a 4-cycle and a 2-factor $2K_2$  in Fig.~\ref{fig1}(a), with $\nabla(\Gamma'_B)$ in Fig.~\ref{fig1}(d).\end{enumerate}

Given a snarkless $\Gamma'$ a
new snarkless multigraph $\Gamma"$ is obtained from $\Gamma'$ by replacing any edge $e$ with end-vertices say $u,v$ by the submultigraph resulting as the union of a path $P_4=(u,u',v',v)$ and an extra edge with end-vertices $u',v'$. For example, $\Gamma'$ in Fig.~\ref{fig1}(g) as $\Gamma"$ in Fig.~\ref{fig1}(s) and $\nabla(\Gamma")$ in Fig.~\ref{fig1}(t). Using this replacement of an edge $e$ by the said submultigraph, one can transform the submultigraph $\Gamma'_B$ with the enclosed blue edge $e$ in item (B), above, into a $\Gamma"_B$ as in Fig.~\ref{fig1}(b), with $\nabla(\Gamma"_B)$ in Fig.~\ref{fig1}(e); or with the four red and green edges into a $\Gamma"_B$ as in Fig.~\ref{fig1}(c), with $\nabla(\Gamma"_B)$ in Fig.~\ref{fig1}(f).

The triangle-replaced graphs $\nabla(\Gamma')$ of snarkless $(e_1)(e_2)(e_3)$-graphs $\Gamma'$ either proper or improper, with $(e_1)(e_2)(e_3)\in\{2^3,1^20,0^3\}$ yield  egc $1^20$-graphs, illustrated from:
\begin{enumerate}
\item the graph $\Gamma'$ in Fig.~\ref{fig1}(h), namely $\nabla(\Gamma'_A)$ \hspace*{1.7cm}onto the $1^20$-graph $\Gamma$  in Fig.~\ref{fig1}(m),
\item the graph $\Gamma'$ in Fig.~\ref{fig1}(i), namely $K_4$ \hspace*{2.6cm}onto the $1^20$-graph $\Gamma$  in Fig.~\ref{fig1}(n),
\item the graph $\Gamma'$ in Fig.~\ref{fig1}(j), namely $K_{3,3}$ \hspace*{2.4cm}onto the $1^20$-graph $\Gamma$  in Fig.~\ref{fig1}(o),
\item the graph $\Gamma'$ in Fig.~\ref{fig1}(k), namely the 3-cube graph $Q_3$ onto the  $1^20$-graph $\Gamma$ in Fig.~\ref{fig1}(p). 
\end{enumerate}
This raises the observation that non-equivalent 1-factorizations $F$ and $F'$ of an $(e_1)(e_2)(e_3)$-graph $\Gamma'$ 
like in Fig.~\ref{fig1}(k) and Fig.~\ref{fig1}(l), respectively, for $\Gamma'=Q_3$ 
result in non-equivalent 1-factorizations $\nabla(F)$ and $\nabla(F')$ of $\Gamma=\nabla(\Gamma')$ represented in this case on $\Gamma=\nabla(\Gamma')=\nabla(Q_3)$ in Fig.~\ref{fig1}(p) and Fig~\ref{fig1}(q), respectively. This leads to the final assertion in Theorem~\ref{stat}, below.

Fig.~\ref{fig1}(u) is the egc $1^20$-graph given by $\nabla(\Gamma')$ for the dodecahedral graph $\Dod=\Gamma'$ in which the union of any two edge-disjoint 1-factors of a 1-factorization of $\Gamma'$ yields a Hamilton cycle. 

In contrast, the Coxeter graph $\Cox=\Gamma'$ in Fig.~\ref{fig1}(v) is non-hamiltonian, but the union of any two of its (edge-disjoint) 1-factors is the disjoint union of two 14-cycles, whose apparent interiors are shaded yellow and light gray in the figure. Thus, $\Gamma=\nabla(\Gamma')$ is an egc $1^20$-graph.

\begin{figure}[htp]
\includegraphics[scale=0.88]{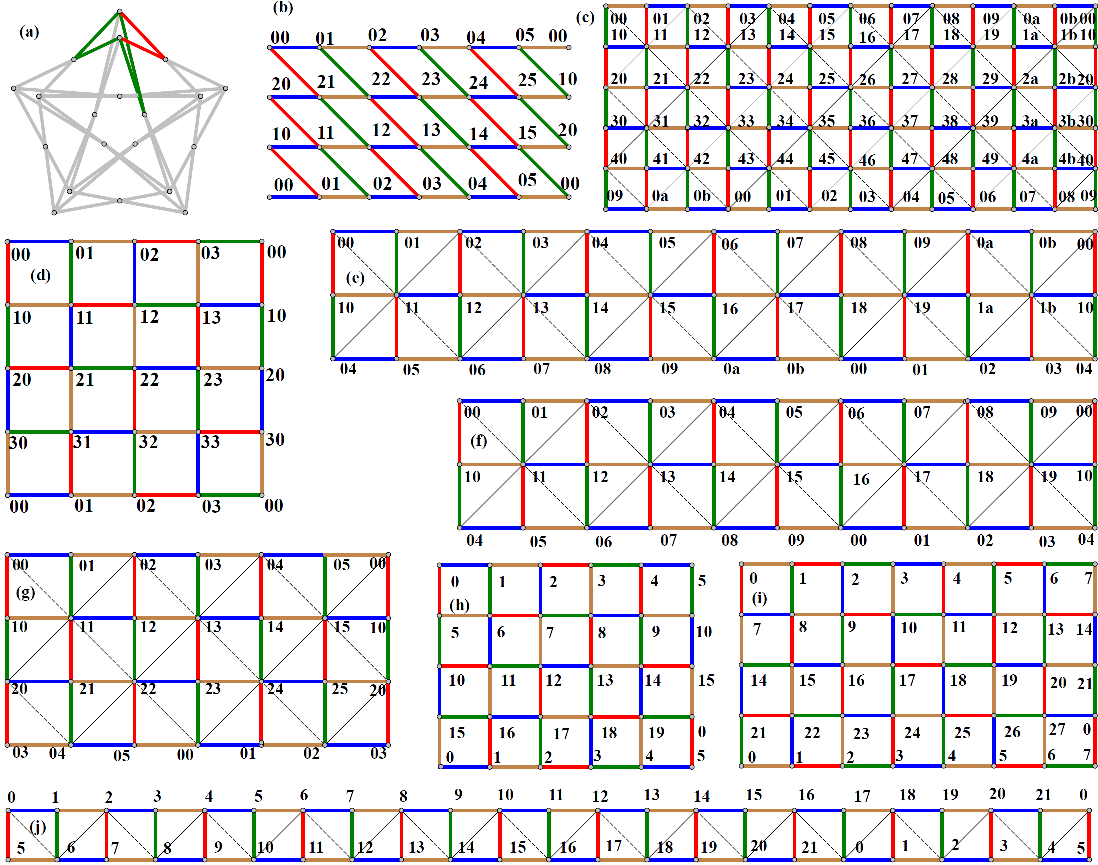}
\caption{Cutouts of $(e_1)(e_2)(e_3)(e_4)$-graphs: (a) is the Folkman graph $\mathbb{D}5$; (b) is for a graph embeddable into the Klein bottle; (c) exemplifies Theorem~\ref{re} {\it 2(e)}. (d) is for the 4-cube in Theorem~\ref{re} {\it 1(c)}; (e) and (f) exemplify Theorem~\ref{re} {\it 3(b)}; (g) exemplifies Theorem~\ref{re} {\it 3(c)}; (h) and (i) exemplify Theorem~\ref{re} {\it 2(b)};
(j) exemplifies Theorem~\ref{re} {\it 3(a)};}
\label{f2}
\end{figure}

\begin{theorem}\label{stat}
A $1^20$-graph $\Gamma$ is egc if and only if $\Gamma$ is the triangle-replaced graph of a snarkless $\Gamma'$.
Moreover, non-equivalent 1-factorizations of  such  $\Gamma'$ result in corresponding non-equivalent 1-factorizations of $\Gamma$.
\end{theorem}

\begin{proof}
There are two types of edges in a $1^20$-graph $\Gamma$ namely the {\it triangle edges} (those belonging to some triangle of $\Gamma$) and the remaining {\it non-triangle edges}. Each vertex $v$ of $\Gamma$ is incident to a unique non-triangle edge  $e_v$ and is nonadjacent to a unique edge $\bar{e}_v$ (opposite to $v$) in the sole triangle $T_v$ of $\Gamma$ to which $v$ belongs. In any 1-factorization $F=(F_1,\ldots,F_g)$ of $\Gamma$ both $e_v$ and $\bar{e}_v$ belong to the same factor $F_i$ ($i=1,\ldots,g$). Moreover, each edge $e=\{u,v\}$ of $\Gamma$ (where $e=e_u=e_v$) belongs solely to corresponding triangles $T_u$ and $T_v$ with opposite edges $\bar{e}_u$ and $\bar{e}_v$. Clearly, $\{e=e_u=e_v,\bar{e}_u,\bar{e}_v\}\subseteq F_i$ with equality given precisely when $\Gamma$ is the triangular prism in Fig.~\ref{fig1}(h).

We will define an inverse operator $\nabla^{-1}$ of $\nabla$ that applies to each egc $1^20$-graph $\Gamma$. Given one such $\Gamma$ contracting simultaneously all the triangles $T$ of $\Gamma$ consists in removing the edges of those $T$ and then identifying the vertices $v_1^T,v_2^T,v_3^T$ of each $T$ into a corresponding single vertex $v_T$ where $v_i^T$ for $i\in\{1,2,3\}$ has its unique incident non-triangle edge of $\Gamma$ with color $i$. This is done
so that whenever two triangles $T$ and $T'$ have respective vertices $v_i^T$ and $v_j^{T'}$ adjacent in $\Gamma$ ($i,j\in\{1,2,3\}$), then $i=j$ and the edge $v_i^Tv_i^{T'}$ of $\Gamma$ is removed and replaced by a new edge $v_Tv_{T'}$. The result of these simultaneous triangle contractions is a multigraph $\Gamma'=(V',E',\phi')$ with each $v_T\in V'$ incident to three edges of $E'$ one per each color in $\{1,2,3\}$. The ensuing edge coloring in $\Gamma'$ corresponds to a vertex-neighborhood labeling $\rho:A'\rightarrow\{1,2,3\}$ of $\Gamma'$ from which it follows that $\Gamma$ is the triangle-replaced graph of $\Gamma'$ with respect to $\rho$ that is: $\nabla^{-1}(\Gamma)=\Gamma'$.
This establishes an identification of $\Gamma$ and $\nabla(\Gamma')$ so that
the triangle-edges of $\Gamma$ are the $\nabla$-edges of $\nabla(\Gamma')$ and the non-triangle edges $\Gamma$ are the $\Gamma'$-edges of $\nabla(\Gamma')$. This implies the main assertion of the statement of the theorem.
\end{proof}

\section{Egc girth-4-regular graphs}\label{s2}

\begin{remark}\label{rectangle} In this section and in Section~\ref{s3}, we consider $(e_1)(e_2)(e_3)(e_4)$-graphs $\Gamma$ with
 $(e_1)(e_2)(e_3)(e_4)\ne1111$.
Many such graphs are toroidal and obtained from the square tessellation denoted by its Schl\"{a}fli symbol $\{4,4\}$. Let $T$ be
the group of translations of the plane that preserve such tessellation $\{4,4\}$. Then, $T$ is isomorphic
to $\mathbb{Z}\times\mathbb{Z}$ and acts transitively on the vertices of $\{4,4\}$. If $U$ is a subgroup of
finite index in $T$ then ${\mathcal M}=\{4,4\}/U$ is a
finite map of type $\{4,4\}$ on the torus, and every
such map arises this way (\cite{WP}, Section 6). A symmetry $\alpha$ of $\{4,4\}$ acts as a symmetry of $\mathcal M$ if and only if $\alpha$ normalizes $U$.
Every such $\mathcal M$ has symmetry group $\Aut({\mathcal M})$ transitive
on vertices, horizontal edges and vertical edges. Moreover, for each edge $e$ of $\mathcal M$ there is
a symmetry that reverses $e$. The tessellation $\{4,4\}$ may be considered as a lattice, so it has a {\it fundamental region} \cite{CS,Pal}. Such region will be called a {\it cutout} $\Phi$ and be given by a rectangle $r$ squares wide and $t$ squares high, with the left and right edges identified by parallel translation in order to get a toroidal embedding of $\Gamma$ and the bottom edges identified with the top edges after a shift of $s$ squares to the right, as in Fig. 6 of \cite{WP}.  A toroidal graph with such a cutout will be denoted: 
\begin{eqnarray}\label{eqn1}\{4,4\}_{r,t}^s\end{eqnarray} 
While the aim of \cite{PW,ARS,WP} is the study of edge-transitive graphs, we find 1-factorizations of $g$-tight graphs in graphs $\{4,4\}_{r,t}^s$ of a more ample nature.
Our notation for the vertices of those cutouts will be $(i,j)$ or $ij$ if no confusion arises, where $0\le i<r$ and $0\le j<t$ as in the examples of tight factorizations in 
Fig.~\ref{f2}(d), Fig.~\ref{f2}(e), Fig.~\ref{f2}(f), Fig.~\ref{f2}(g), Fig.~\ref{f2}(h), Fig.~\ref{f2}(i) and Fig.~\ref{f2}(j). In particular for $t>1$ as in Fig.~\ref{f2}(d), Fig.~\ref{f2}(e), Fig.~\ref{f2}(f) and Fig.~\ref{f2}(g),
the notation arises from the fact that $\{4,4\}$ can be considered as the undirected Cayley graph of the direct-sum group $\mathbb{Z}\oplus\mathbb{Z}$ with generator set formed by $(1,0)$ for horizontal left-to-right arcs and $(0,1)$ for vertical up-to-down arcs.
In case $t=1$ (Fig.~\ref{f2}(c), Fig.~\ref{f2}(h), Fig.~\ref{f2}(i) and Fig.~\ref{f2}(j)), we simplify notation by writing $i$ instead of $(i,0)$ or $i0$.
In all of Fig.~\ref{f2}, edge colors are encoded by numbers as follows: 1 for red, 2 for blue, 3 for green and 4 for hazel. (Thin and dashed diagonals of squares are to used in the proof of Theorem~\ref{re}).\end{remark}

\begin{remark}\label{reverse}
If a fundamental region $\Phi$ of $\{4,4\}$ as in Remark~\ref{rectangle} is identified in reverse on a pair $\mathcal P$ of opposite sides, and directly on the other pair, we get a Klein bottle $\mathcal K$ \cite{umkb}. There are egc-graphs that are skeletons of a $\{4,4\}$-tessellation of $\mathcal K$ for example in Fig.~\ref{f2}(b), whose embedding into $\mathcal K$ has corresponding cutout that can be obtained from the one of $\{4,4\}_{r,t}^s$ above (with $(r,t,s)=(6,3,0)$) first by replacing the vertical edges by corresponding square-face diagonals, while keeping the horizontal edges (so the new faces are lozenge rhombi), and second by identifying the horizontal top and bottom borders of the original cutouts, as well as the left and right borders, these with reverse orientations, with the resulting $\mathcal K$-embedding that we will be denoted: 
\begin{eqnarray}\label{eqn2}\lfloor 4,4\rceil_{r,t}^s\end{eqnarray} Then, the example of Fig.~\ref{f2}(b) is in $\lfloor 4,4\rceil_{6,3}^0$.
 If $\Phi$ is identified in reverse on both pairs of opposite sides, a projective-planar graph
is obtained. This can be ruled out because $\{4,4\}$-tessellations only exist on surfaces with Euler characteristic 0.
\end{remark}

\begin{remark}\label{4partes} Let $0<n\in\mathbb{Z}$. The {\it $n$-cube graph} $Q_n$ has as vertices the $n$-tuples with entries in $\mathbb{Z}_2$ and edges only between vertices at unit Hamming distance.
 In Subsection~\ref{Q4}, we consider the 4-cube graph $Q_4=\{4,4\}_{4,4}^0$. Other $(e_1)(e_2)(e_3)(e_4)$-graphs with $(e_1)(e_2)(e_3)(e_4)\ne 1111$ and that are not prisms of $(e_1)(e_2)(e_3)$-graphs are:
 \begin{enumerate}
 \item[{\bf(i)}] the bipartite complement of the Heawood graph, in Subsection~\ref{Hea};
 \item[{\bf(ii)}] the {\it subdivided double} $\mathbb{D}\Gamma$ \cite{PW,WP} of a $4$-regular graph $\Gamma$ is the bipartite graph
with vertex set $(V(\Gamma)\times\mathbb{Z}^2)\cup E(\Gamma)$ and an edge between vertices $(v, i)\in V(\Gamma)\times\mathbb{Z}^2$ and
$e\in E(\Gamma)$ whenever $v$ is incident to $e$ in $\Gamma$;
\cite[Lemma 4.2]{PW} asserts that if $\Gamma$ is 4-regular and arc-transitive, then $\mathbb{D}\Gamma$ is 4-regular and semisymmetric;
for example, the Folkman graph (Fig.~\ref{f2}(a)) is the subdivided double $\mathbb{D}K_5$ of the complete graph $K_5$;
\item[{\bf(iii)}] the {\it circulant graphs}, i.e. the Cayley graphs $C_n(i,j)$ of the cyclic group $\mathbb{Z}_n$ ($n>6$) with generating sets $\{\pm i,\pm j\}$ where $1\le i<j<\frac{n}{2}$ and $\gcd(n,i,j)=1$; most of these are $2^4$-graphs (assuming $n$ even, otherwise chromatic index is not 4), with additional cycles appearing whenever a congruence $\lambda i\pm(4-\lambda)j\equiv 0 \pmod{n}$ holds, (e.g., $C_{14}(2,3)$ is a $2^4$-graph, $C_{12}(2,3)$ is a $3^22^2$-graph, $C_{10}(1,3)$ is a $6^4$-graph and $C_8(1,3)\equiv K_{4,4}$ is a $9^4$-graph); however, such graphs $C_n(i,j)$ can always be seen as toroidal graphs;
\item[{\bf(iv)}] the {\it wreath graphs} $W(n,2)=C_n[\overline{K_2}]$ ($n>4$), i.e. lexicographic products of an $n$-cycle and the complement $\overline{K_2}$ of $K_2$; these are $5^4$-graphs; ($W(4,2)\equiv K_{4,4}$ is a $9^4$-graph).\end{enumerate}
\end{remark}

\begin{remark}\label{obstruct}
If an $(e_1)(e_2)(e_3)(e_4)$-graph $\Gamma$ as in Remark~\ref{4partes} contains a subgraph $\Gamma'$ guaranteeing that $\Gamma$ is not an egc-graph, then $\Gamma'$ is said to be an {\it egc-obstruction}. A subgraph $\Gamma'\equiv K_{2,3}$ is an egc-obstruction for a girth-4-regular graph $\Gamma$ since each of the proper edge-colorings of $\Gamma$ contains a quadrangle of $\Gamma'$ with only two colors. In Fig.~\ref{f2}(a), one such graph $\Gamma$ namely the Folkman graph $\Gamma=\mathbb{D}K_5$ is presented with a subgraph $\Gamma'\equiv K_{2,3}$ formed by a green quadrangle and a red 2-path. $C_{10}(1,3)$ and $W(6,2)$ also have obstruction isomorphic to $K_{2,3}$.\end{remark}

The following lemma is a tool for Theorems~\ref{SC} and~\ref{barrels}, and a variation of it, for Theorem~\ref{needed}.

\begin{lemma}\label{SC}
A sufficient condition for an $(e_1)(e_2)(e_3)(e_4)$-graph $\Gamma$ with $(e_i)<3$ ($i=1,2,3,4$) to be egc is existence of 2-factorization $\{F_1,F_2\}$ of $\Gamma$ such that each 2-factor $F_i$ ($i=1,2$):
\begin{enumerate}\item is the disjoint union of even-length cycles; and \item has an even-length cycle $D(C)$ as in item 1, for each 4-cycle $C$ of $\Gamma$ sharing with $C$ exactly two consecutive edges.\end{enumerate}
\end{lemma}

\begin{proof} A 1-factorization of $F_1$ via colors 1 and 2 and a 1-factorization of $F_2$ via colors 3 and 4 exist and form a tight 1-factorization of $\Gamma$.
\end{proof}

\begin{theorem}\label{re} The following $3^4$-, $3^22^2$- and $2^4$-graphs exist, and are egc or not, as indicated, where notations (\ref{eqn1}) and (\ref{eqn2}) are usedt:
\begin{enumerate}
\item $3^4$-graphs comprising the:
\begin{enumerate} \item bipartite complement of the Heawood graph, which is not egc, (Subsection~\ref{Hea});
\item  4-regular subdivided doubles $\mathbb{D}\Gamma$ of $4$-regular graphs $\Gamma$ which are not egc;
\item 4-cube graph $Q_4=\{4,4\}_{4,4}^0$ which is egc in two different, orthogonally related ways, (Subsections~\ref{Q4}-\ref{s5}, Remark~\ref{not}; an initial example is in Fig.~\ref{f2}(d));
\end{enumerate}
\item $2^23^2$-graphs (assuming $0<t\le r$ and $0\le s<r$), comprising:
\begin{enumerate}
\item $\{4,4\}_{2\ell,4}^0$: egc $\Leftrightarrow \ell\in(3,\infty)\cap 2\mathbb{Z}$; (concatenating copies of $\{4,4\}_{4,4}^0$ in Fig.~\ref{f2}(d));
\item $\{4,4\}_{4s,1}^s$: egc\! $\Leftrightarrow s\in(4,\infty)\cap\mathbb{Z}\setminus 2\mathbb{Z}$; (Fig.~\ref{f2}(h--i), for $r_t^s=20_1^5,28_1^7$);
\end{enumerate}
\item $2^4$-graphs (assuming $0<t\le r$ and $0\le s<r$) comprising:
\begin{enumerate}
\item $\{4,4\}_{r,1}^s$: egc $\Leftrightarrow r\in 2\mathbb{Z}\setminus 4\mathbb{Z}$ $s\in\mathbb{Z}\setminus(2\mathbb{Z}\cup 1)$ and $r\ne 3s+1$; (Fig.~\ref{f2}(j), $r_t^s=22_1^5$);
\item $\{4,4\}_{r,2}^s$: egc $\Leftrightarrow r\in [10,\infty)\cap 2\mathbb{Z}$ and $s\in[4,r-4]\cap 2\mathbb{Z}$;
(Fig.~\ref{f2}(e--f), $\!r_t^s=\!12_2^4,10_2^4$);
\item $\{4,4\}_{r,3}^s$: egc $\Leftrightarrow r\in[6,\infty)\cap 2\mathbb{Z}$ and $s=[3,r-3]\!\setminus\!2\mathbb{Z}$; (Fig.~\ref{f2}(g), $r_t^s=6_3^3$);
\item $\{4,4\}_{r,4}^s$: egc $\Leftrightarrow 4\le r\in 2\mathbb{Z}$ and $0<s\in 2\mathbb{Z}$; (color pattern as in Fig.~\ref{f2}(e--g),(j);
\item $\{4,4\}_{r,t}^s$ $t\in[4,\infty)$: egc $\Leftrightarrow r\in 2\mathbb{Z}$ and $t+s\in 2\mathbb{Z}$; (color pattern as in Fig.~\ref{f2}(c));
\item $\lfloor 4,4\rceil_{r,t}^0$: egc $\Leftrightarrow r\in[6,\infty)\cap 2\mathbb{Z}$ and $t\in[3,\infty)\cap\mathbb{Z}\setminus 2\mathbb{Z}$; (Fig.~\ref{f2}(b), $r_t^s=6_3^0,8_3^0$).
\end{enumerate}
\end{enumerate}\end{theorem}

A cycle of a graph $\Gamma$ as in Remarks~\ref{rectangle}-\ref{reverse} is said to be {\it 1-zigzagging} if it is formed by alternate horizontal and non-horizontal (i.e., all vertical or all $45^\circ$-tilted) edges. A 2-factor
of $\Gamma$ is said to be {\it 1-zigzagging} if its composing cycles are 1-zigzagging. A 2-factorization of $\Gamma$ is said to be {\it 1-zigzagging} if its composing 2-factors are 1-zigzagging.

\begin{proof} We pass to analyze the different items composing the statement of Theorem~\ref{re}.

{\bf Items 1 and 2:}
Item 1($a$) is proved in Subsection~\ref{Hea}.
The graphs of item 1($b$) have egc-obstructions (see Remark~\ref{obstruct}) formed by three edge-disjoint paths of length 2 between two nonadjacent vertices, e.g., $\mathbb{D}K_5$ in Fig.~\ref{f2}(a), with egc-obstruction formed by four green edges and two red edges.
Items 1($c$) and 2($b$)
are proved in Subsection~\ref{s5}, Remark~\ref{not}; (see also Fig.~\ref{f2}(d)).
Item 2($a$) is proved by concatenating copies of $\{4,4\}_{4,4}^0$ as in Fig.~\ref{f2}(d).

{\bf Item 3:}
Lemma~\ref{SC} applies to each $\Gamma$ as in Remarks~\ref{rectangle}-\ref{reverse} via the 1-zigzagging 2-factorization $(12)(34)$ which contains its 2-factors having exactly two consecutive edges in common with each 4-cycle, as shown in Fig.~\ref{f2}(e,f,g,j,b,c).
In items 3($a$--$f$), the non-egc cases indicated via ``$\Leftrightarrow$" include those not satisfying the sufficient condition of Lemma~\ref{SC}, because such condition becomes also necessary for each $\Gamma$ arising from a toroidal or Klein-bottle cutout as in Remarks~\ref{rectangle}--~\ref{reverse}.
Moreover, all the 1-zigzagging cycles in $\Gamma$ have even length and share two consecutive edges with each 4-cycle
 precisely where indicated via ``$\Leftrightarrow$" in items 3($a$--$f$).
 Furthermore, in Fig.~\ref{f2}(e,f,g,j), the thin diagonals separate those pairs of consecutive edges (for the 2-factorization $(12)(34)$), while the dashed ones do the same for the 2-factorization $(14)(23)$.
In addition,
note the exclusion in item 3($a$) of the cases $\{4,4\}_{6,1}^1$ and those for which $r=3s+1$ and in item 3($c$) the case $\{4,4\}_{6,3}^1$.
In item 3($b$), note the lower bound for $r$ due to $\{4,4\}_{8,2}^4$ being a $3^22^2$-graph but not egc.
For item 3($e$), the case $s=0$ is covered in item 2($a$).

For the cases of Klein-bottle graphs in item 3($f$), there are two different color patterns, the first one,  exemplified in Fig.~\ref{f2}(b), valid for $6\le r\in 2\mathbb{Z}$ and the second one further restricted to having $r\in 4\mathbb{Z}$ with more than two colors on each horizontal line.
Note the exclusion of the cases $\lfloor 4,4\rceil_{4,t}^0$ ($3<t\in\mathbb{Z}\setminus 2\mathbb{Z}$), for they are not girth-regular.
\end{proof}

The egc-cases of Theorem~\ref{re} item 3 are exemplified respectively in Fig.~\ref{f2}(e,f,g,j,b,c), characterized by having cycles with blue-hazel horizontal edges and cycles with red-green non-horizontal edges. However, transposing the two colors in 1-zigzagging cycles of 2-factors in 2-factorizations $(12)(34)$ $(13)(24)$ or $(14)(23)$  yield tight factorizations with horizontal cycles colored with more than 2 colors.

\subsection{The 4-cube as a twice-egc girth-4-regular graph}\label{Q4}

Consider the three mutually orthogonal Latin squares of order 4, or MOLS(4) \cite{CD} contained as the second, third and fourth rows in the following compound matrix:

\begin{align}\label{(1)}\begin{array}{c|cccc}
&1&2&4&7\\\hline
0&111&222&333&444\\
3&243&134&421&312\\
5&324&413&142&231\\
6&432&341&214&123
\end{array}\end{align}

\noindent where for us row and column headings will stand for the following 4-tuples:
\begin{align}\label{(2)}\begin{array}{cccccccc}
0\!=\!0000, \!&\!1\!=\!1000, \!&\!2\!=\!0100, \!&\!3\!=\!1100,\!&\! 4\!=\!0010,\!&\!5\!=\!1010,\!&\!6\!=\!0110,\!&\!7\!=\!1110,\\
0'\!=\!0001,\!&\!1'\!=\!1001,\!&\!2'\!=\!0101,\!&\!3'\!=\!1101,\!&\!4'\!=\!0011,\!&\!5'\!=\!1011,\!&\!6'\!=\!0111,\!&\!7'\!=\!1111.
\end{array}\end{align}

\begin{figure}[htp]
\hspace*{1.3cm}
\includegraphics[scale=0.4]{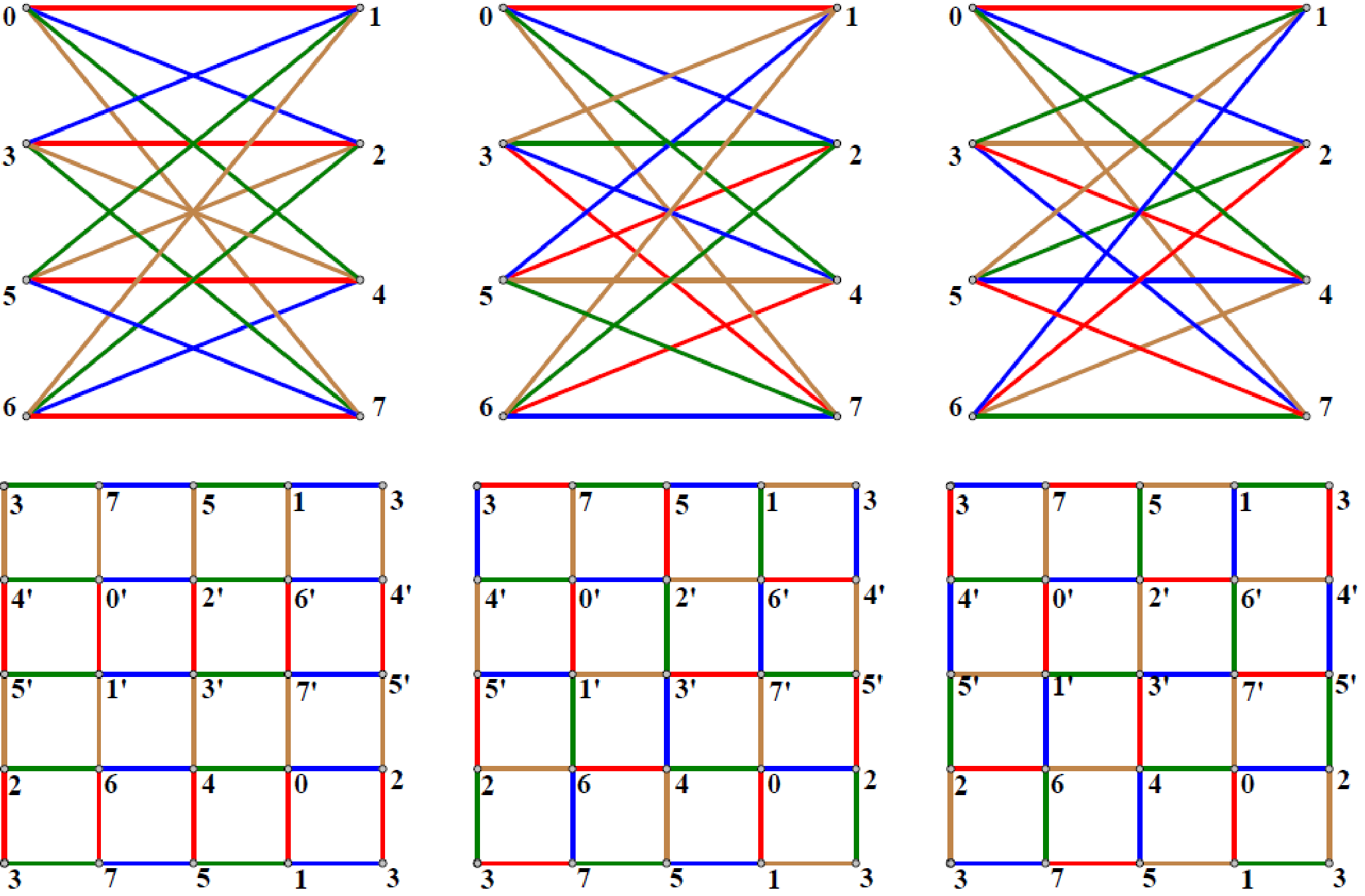}
\caption{On top of the figure there are represented three 1-factorizations of $K_{4,4}$. Below them, corresponding toroidal cutouts of $Q_4$ are drawn, with 0-, 1- and 2-color 4-cycles.}
\label{f3}
\end{figure}

Based on display~(\ref{(1)}), the top of Fig.~\ref{f3} contains three copies of $K_{4,4}$ properly colored in a mutually-orthogonal way, where colors are numbered as established in Remark~\ref{rectangle}.
Letting $\phi:Q_4\rightarrow K_{4,4}$ be the canonical projection map of $Q_4$ seen as a double covering of $K_{4,4}$ obtained by identifying the pairs of antipodal vertices of $Q_4=\{4,4\}_{4,4}^0$ these vertices denoted as in display~(\ref{(2)}), note that
in the bottom of Fig.~\ref{f3} corresponding copies of the colored inverse images $\phi^{-1}(K_{4,4})$ of the three mentioned copies of $K_{4,4}$
 are depicted.
The leftmost copy of $Q_4$ in Fig.~\ref{f3} has color $i$ attributed precisely to those edges parallel to the $i^{th}$ coordinate direction, for $i=1,2,3,4$. This constitutes a 1-factorization $F_0=\{F_0^1,F_0^2,F_0^3,F_0^4\}$ of $Q_4$. On the other hand, the center and rightmost copies of $Q_4$ in the figure determine
 1-factorizations $F_1$ and $F_2$ of $Q_4$ for which each girth cycle of $Q_4$ intersects every composing 1-factor $F_i^j$ of $F_i$ where $i=1,2$ and $j=$ 1 (red), 2 (blue), 3 (green), 4 (hazel).

For each edge $e$ of $Q_4$ we say that $e$ has $i$-{\it color} $j\in\{1,2,3,4\}$ if $e\in F_i^j$. Then, each 4-cycle of $Q_4$ has opposite edges with a common 0-color in $\{1,2,3,4\}$ with a total of two (nonadjacent) 0-colors in $\{1,2,3,4\}$ per 4-cycle, say 0-colors $\ell_1,\ell_2\in\{1,2,3,4\}$ with $\ell_1\ne\ell_2$ so one such 4-cycle can be expressed as $(\ell_1,\ell_2,\ell_1,\ell_2)$. On the other hand, the 4-cycles of $Q_4$ use all four $i$-colors 1,2,3,4, once each, for $i=1,2$.

There are twenty-four 4-cycles in $Q_4$ six of each of the 0-color 4-cycles expressed in the first two columns of the following array, with two complementary 0-color 4-cycles per row. On the other hand, the third and fourth columns here contain respectively the 1-color and 2-color 4-cycles corresponding to the 0-color 4-cycles in the first two columns:

\begin{align}\label{3rdand4th}\begin{array}{||c|c||c||c||}\hline
(1212)&(3434)&(1234)&(1243)\\
(1313)&(2424)&(1324)&(1342)\\
(1414)&(2323)&(1423)&(1432)\\\hline
\end{array}\end{align}

\subsection{Toroidal representation tables}\label{s5}

\begin{remark}\label{not} The triple array in Table~\ref{I} presents $F_i$ ($i=0,1,2$)  in schematic representations of $Q_4=\{4,4\}_{4,4}^0$ where $\circ$ stands for a vertex  of $Q_4$ and $\square$ stands for an $i$-color 4-cycle. This table guarantees Theorem~\ref{re} item 3($c$) via the last two columns of color quadruples in display~(\ref{3rdand4th}), because the four colors are employed on the edges of each 4-cycle:

\begin{table}[htp]
$$\begin{array}{||ccccccccc||ccccccccc||ccccccccc||}\hline\hline
\circ&3\!&\circ&2\!&\circ&3\!&\circ&2\!&\circ&  \circ&1\!&\circ&3\!&\circ&2\!&\circ&4\!&\circ& \circ&2\!&\circ&1\!&\circ&4\!&\circ&3\!&\circ\\
 4\!&\square\!&4\!&\square\!&4\!&\square\!&4\!&\square\!&4\!& 2\!&\square\!&4\!&\square\!&1\!&\square\!&3\!&\square\!&2\!& 1\!&\square\!&4\!&\square\!&3\!&\square\!&2\!&\square\!&1\\
\circ&3\!&\circ&2\!&\circ&3\!&\circ&2\!&\circ& \circ&3\!&\circ&2\!&\circ&4\!&\circ&1\!&\circ& \circ&3\!&\circ&2\!&\circ&1\!&\circ&4\!&\circ\\
 1\!&\square\!&1\!&\square\!&1\!&\square\!&1\!&\square\!&1\!& 4\!&\square\!&1\!&\square\!&3\!&\square\!&2\!&\square\!&4\!& 2\!&\square\!&1\!&\square\!&4\!&\square\!&3\!&\square\!&2\\
 \circ&3\!&\circ&2\!&\circ&3\!&\circ&2\!&\circ& \circ&2\!&\circ&4\!&\circ&1\!&\circ&3\!&\circ& \circ&4\!&\circ&3\!&\circ&2\!&\circ&1\!&\circ\\
 4\!&\square\!&1\!&\square\!&4\!&\square\!&1\!&\square\!&4\!& 1\!&\square\!&3\!&\square\!&2\!&\square\!&4\!&\square\!&1\!& 3\!&\square\!&2\!&\square\!&1\!&\square\!&4\!&\square\!&3\\
\circ \!&3\!&\circ&2\!&\circ&3\!&\circ&2\!&\circ&  \circ&4\!&\circ&1\!&\circ&3\!&\circ&2\!&\circ& \circ&1\!&\circ&4\!&\circ&3\!&\circ&2\!&\circ\\
 1\!&\square\!&1\!&\square\!&1\!&\square\!&1\!&\square\!&1\!& 3\!&\square\!&2\!&\square\!&4\!&\square\!&1\!&\square\!&3\!& 4\!&\square\!&3\!&\square\!&2\!&\square\!&1\!&\square\!&4\\
\circ \!&3\!&\circ&2\!&\circ&3\!&\circ&2\!&\circ&  \circ&1\!&\circ&3\!&\circ&2\!&\circ&4\!&\circ& \circ&2\!&\circ&1\!&\circ&4\!&\circ&3\!&\circ\\\hline\hline
\end{array}$$
\caption{Representations of $F_i$ ($i=0,1,2$) corresponding to the three representations of Fig.~\ref{f3}, where the toroidal cutouts are represented with vertices given with symbols $\circ$ and 4-cycles given with symbols $\square$. The edge colors 1, 2, 3, 4 replace the lines between the vertices $\circ$.}
\label{I}\end{table}

\begin{enumerate}\item[{\bf(a)}] either as horizontal or vertical color quadruples (as in display~\ref{3rdand4th}) alternated with the symbols $\circ$ that represent the vertices of $Q_4$;  \item[{\bf(b)}] or as quadruples around the symbols $\square$ representing the other 4-cycles.
\end{enumerate}

\begin{table}[htp]
$$\begin{array}{||ccccccccccc||ccccccccccccccc||}\hline\hline
_{00}\!\!&2\!\!&_{01}\!\!&4\!\!&_{02}\!\!&1\!\!&_{03}\!\!&3\!\!&_{04}\!\!&2\!\!&_{05}&_{00}\!\!&2\!\!&_{01}\!\!&1\!\!&_{02}\!\!&3\!\!&_{03}\!\!&4\!\!&_{04}\!\!&2\!\!&_{05}\!\!&1\!\!&_{06}\!\!&3\!\!&_{07}\\
1\!\!&\square\!\!&3\!\!&\square\!\!&2\!\!&\square\!\!&4\!\!&\square\!\!&1\!\!&\square\!\!&3&1\!\!&\square\!\!&3\!\!&\square\!\!&4\!\!&\square\!\!&2\!\!&\square\!\!&1\!\!&\square\!\!&3\!\!&\square\!\!&4\!\!&\square\!\!&2\\
_{05}\!\!&4\!\!&_{06}\!\!&1\!\!&_{07}\!\!&3\!\!&_{08}\!\!&2\!\!&_{09}\!\!&4\!\!&_{10}&_{07}\!\!&4\!\!&_{08}\!\!&2\!\!&_{09}\!\!&1\!\!&_{10}\!\!&3\!\!&_{11}\!\!&4\!\!&_{12}\!\!&2\!\!&_{13}\!\!&1\!\!&_{14}\\
3\!\!&\square\!\!&2\!\!&\square\!\!&4\!\!&\square\!\!&1\!\!&\square\!\!&3\!\!&\square\!\!&2&2\!\!&\square\!\!&1\!\!&\square\!\!&3\!\!&\square\!\!&4\!\!&\square\!\!&2\!\!&\square\!\!&1\!\!&\square\!\!&3\!\!&\square\!\!&4\\
_{10}\!\!&1\!\!&_{11}\!\!&3\!\!&_{12}\!\!&2\!\!&_{13}\!\!&4\!\!&_{14}\!\!&1\!\!&_{15}&_{14}\!\!&3\!\!&_{15}\!\!&4\!\!&_{16}\!\!&2\!\!&_{17}\!\!&1\!\!&_{18}\!\!&3\!\!&_{19}\!\! &4\!\!&_{20}\!\!&2\!\!&_{21}\\
2\!\!&\square\!\!&4\!\!&\square\!\!&1\!\!&\square\!\!&3\!\!&\square\!\!&2\!\!&\square\!\!&4&4\!\!&\square\!\!&2\!\!&\square\!\!&1\!\!&\square\!\!&3\!\!&\square\!\!&4\!\!&\square\!\!&2\!\!&\square\!\!&1\!\!&\square\!\!&3\\
_{15}\!\!&3\!\!&_{16}\!\!&2\!\!&_{17}\!\!&4\!\!&_{18}\!\!&1\!\!&_{19}\!\!&3\!\!&_{00}&_{21}\!\!&1\!\!&_{22}\!\!&3\!\!&_{23}\!\!&4\!\!&_{24}\!\!&2\!\!&_{25}\!\!&1\!\!&_{26}\!\!&3\!\!&_{27}\!\!&4\!\!&_{00}\\
4\!\!&\square\!\!&1\!\!&\square\!\!&2\!\!&\square\!\!&3\!\!&\square\!\!&4\!\!&\square\!\!&1&3\!\!&\square\!\!&4\!\!&\square\!\!&2\!\!&\square\!\!&1\!\!&\square\!\!&3\!\!&\square\!\!&4\!\!&\square\!\!&2\!\!&\square\!\!&1\\
_{00}\!\!&2\!\!&_{01}\!\!&4\!\!&_{02}\!\!&1\!\!&_{03}\!\!&3\!\!&_{04}\!\!&2\!\!&_{05}&_{00}\!\!&2\!\!&_{01}\!\!&1\!\!&_{02}\!\!&3\!\!&_{03}\!\!&4\!\!&_{04}\!\!&2\!\!&_{05}\!\!&1\!\!&_{06}\!\!&3\!\!&_{07}\\\hline\hline
\end{array}$$
\caption{In this table, instead of $\circ$ standing for each vertex, we set the vertex notation of $\{4,4\}_{20,1}^5$ in Fig.~\ref{f2}(h) and $\{4,4\}_{28,1}^5$ in Fig.~\ref{f2}(i). Here the four colors are indicated as in Fig.~\ref{f2}(d--j) and Subsection~\ref{Q4}. To distinguish these two cases, note that the 4-cycles of $\{4,4\}_{20,1}^5$ (resp. $\{4,4\}_{28,1}^5$) have the 2-factors by color pairs $\{1,2\}$ and $\{3,4\}$ descending in zigzag from right to left (resp. left to right), by alternate vector displacements $(-1,0)$ (resp. $(1,0)$) for colors 1 and 3, and $(0,-1)$ for colors 2 and 4.}
\label{II}
\end{table}

\noindent By associating the oriented quadruple (1,3,2,4) (resp. (3,4,1,2)) of successive edge colors on the left-to-right and the downward (resp. the right-to-left and the downward) straight paths in $F_1$ (resp. $F_2$), situations that we indicate by ``$\searrow(1,3,2,4)$" (resp. ``$\swarrow(3,4,1,2)$"), a complete invariant for $F_1$ (resp. $F_2$) is obtained that we denote by combining between square brackets the just presented notations: $$[a_{\swarrow 12,34}, b_{\searrow 13,24}, \searrow(1,3,2,4)],\mbox{  (resp. }[a_{\searrow 13,24}, b_{\swarrow 12,34}, \swarrow(3,4,1,2)]).$$ This invariant distinguishes $F_1$ and $F_2$ from each other and is generalized for the toroidal graphs in Theorem~\ref{re}, as we will see below in this subsection.

\noindent Table~\ref{I} is also presented to establish similar patterns, like in Table~\ref{II}, allowing in a likewise manner to guarantee Theorem~\ref{re} item 2($b$). We say that
\begin{enumerate}
\item $F_1$ has $a_{\swarrow 12,34}$-{\it zigzags} if any {\it 1-zigzagging} path obtained by walking left, down, left, down and so on, alternates either colors 1 and 2, or colors 3 and 4;
\item $F_1$ has $b_{\searrow 13,24}$-{\it zigzags} if any {\it 2-zigzagging} path obtained by walking right, right, down, down and so on, alternates either colors 1 and 3, or colors 2 and 4;
\item $F_2$ has $a_{\searrow 13,24}$-{\it zigzags} if any {\it 2-zigzagging} path obtained by walking left, left, down, down and so on, either alternates colors 1 and 2, or colors 3 and 4;
\item $F_2$ has $b_{\swarrow 12,34}$-{\it zigzags} if  any {\it 1-zigzagging} path obtained by walking right, down, right, down and so on, either alternates colors 1 and 3, or colors 2 and 4. \end{enumerate}
\end{remark}

A representation as in Table~\ref{I} may be used for the graphs in items 1($b$) and 2 of Theorem~\ref{re}. For example, the cases $(r,t,s)=(20,1,5)$ and $(r,t,s)=(28,1,5)$ in Fig.~\ref{f2}(h--i) are representable as in Table~\ref{II}, where, instead of $\circ$ standing for each vertex, we set the vertex notation of Fig.~\ref{f2}(h--i). Here the four colors are indicated as in Fig.~\ref{f2}(d--j) and Subsection~\ref{Q4}. To distinguish these two cases in Table~\ref{II}, note that the 4-cycles of $\{4,4\}_{20,1}^5$ (resp. $\{4,4\}_{28,1}^5$) have the 2-factors by color pairs $\{1,2\}$ and $\{3,4\}$ descending in zigzag from right to left (resp. left to right), by alternate vector displacements $(-1,0)$ (resp. $(1,0)$) for colors 1 and 3, and $(0,-1)$ for colors 2 and 4. Generalizing and using the invariant notation of Remark~\ref{not}, we can say that the egc graphs in item 2(b) of Theorem~\ref{re} are as follows:
 \begin{enumerate}\item
 $\{4,4\}_{6x+2,1}^5$ for $x>2$ has invariant $[a_{\swarrow 12,34}, b_{\searrow 13,24}, \searrow(2,4,1,3)]$;
 \item
 $\{4,4\}_{6x+ 4,1}^5$ for $x>2$ has invariant $[a_{\swarrow 13,24}, b_{\searrow 12,34}, \swarrow(2,4,3,1)]$.
 \end{enumerate}

\noindent As in Remark~\ref{not}, we have here a distinguished oriented slanted arrow triple: either $[_\swarrow,_\searrow,\searrow]$ or $[_\searrow,_\swarrow,\swarrow]$. The graphs in item 2 of Theorem~\ref{re} admit both invariants.

\subsection{The bipartite complement of the Heawood graph}\label{Hea}

\begin{figure}[htp]
\includegraphics[scale=0.5]{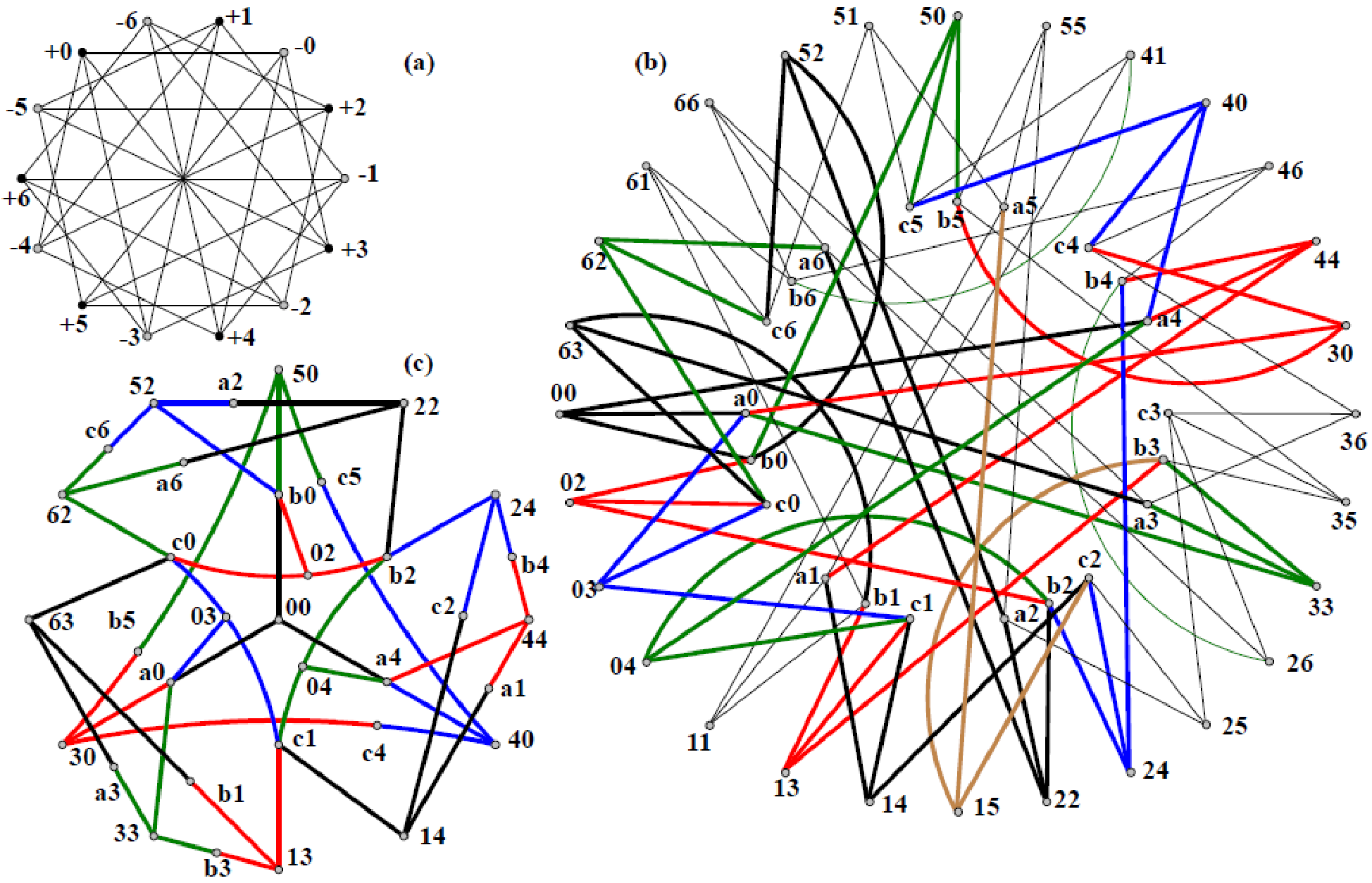}
\caption{The biipartite complement $H$ of the Heawood-graph, with vertex set $V(H)\!=\!\{ij; i\in\{+,-\}$ $j\in\mathbb{Z}_7\}$ is depicted on the upper left, (a), of the figure; its edges $\{+j,-j\}$ $\{+j,-(j+2)\}$ $\{+j,-(j+3)\}$ and $\{+j,-(j+4)\}$ for $j\in\mathbb{Z}_7$ are denoted $jj$ $j(j+2)$ $j(j+3)$ and $j(j+4)$ respectively, where addition is taken$\pmod{7}$. Its associated graph $\GA(H)$ is depicted on the right, (b), of the figure. This forces the coloring of the subgraph of $\GA(H)$ in the lower left, (c), of the figure.}
\label{f4}
\end{figure}

The bipartite complement $H$ of the Heawood-graph, with vertex set $V(H)\!=\!\{ij; i\in\{+,-\}$ $j\in\mathbb{Z}_7\}$ is depicted on the upper left of Fig.~\ref{f4}; its edges $$\{+j,-j\}, \{+j,-(j+2)\}, \{+j,-(j+3)\}\mbox{ and }\{+j,-(j+4)\},$$ for $j\in\mathbb{Z}_7$ will be denoted $$jj, j(j+2), j(j+3)\mbox{ and }j(j+4),\mbox{ respectively,}$$ where addition is taken$\pmod{7}$. This yields the twenty-eight edges of $H$ as arcs from $+$ to $-$ vertices. They form twenty-one 4-cycles $a_i,b_i,c_i$ ($i\in\mathbb{Z}_7$) expressed, by omitting the signs $\pm$ as in Table~\ref{III}.

\begin{table}[htp]
$$\begin{array}{|c|c|c|}\hline
a_0=(00,30,33,03),&b_0=(00,50,52,02),&c_0=(02,62,63,03),\\
a_1=(11,41,44,14),&b_1=(11,61,63,13),&c_1=(13,03,04,14),\\
a_2=(22,52,55,25),&b_2=(22,02,04,24),&c_2=(24,14,15,25),\\
a_3=(33,63,66,36),&b_3=(33,13,15,35),&c_3=(35,25,26,36),\\
a_4=(44,04,00,40),&b_4=(44,24,26,46),&c_4=(46,36,30,40),\\
a_5=(55,15,11,51),&b_5=(55,35,30,50),&c_5=(50,40,41,51),\\
a_6=(66,26,22,62),&b_6=(66,46,41,61),&c_6=(61,51,52,62).\\\hline
\end{array}$$
\caption{The twenty-one 4-cycles $a_i,b_i,c_i$ of $H$ ($i\in\mathbb{Z}_7$).}
\label{III}
\end{table}

\begin{figure}[htp]
\hspace*{3mm}
\includegraphics[scale=01.1]{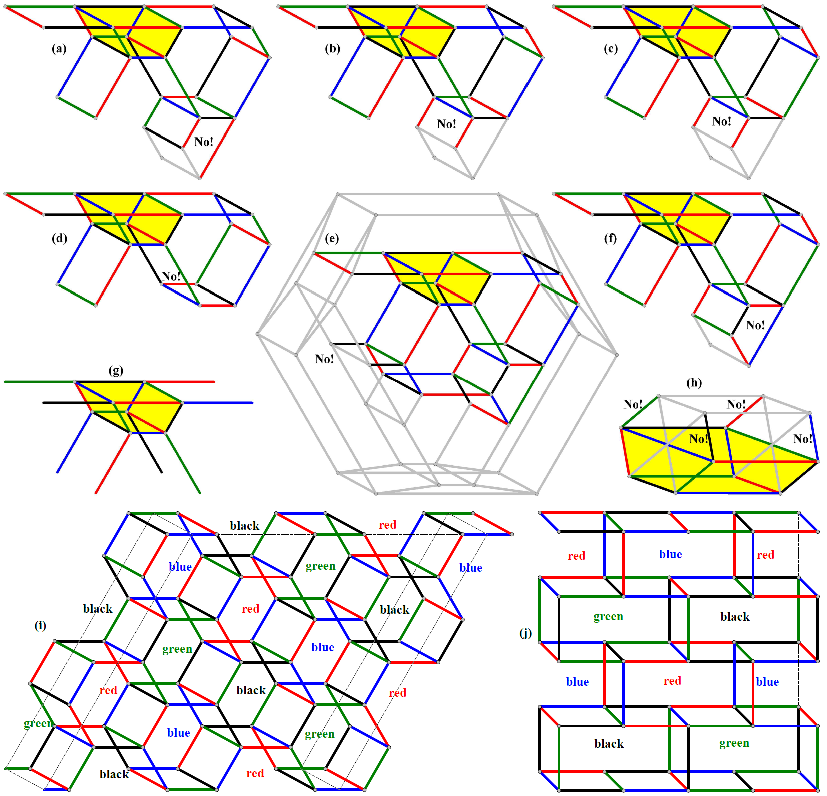}
\caption{Cases of prisms of: truncated octahedron $32^21$ in (a)--(g), $K_{3,3}$ $4^33$ in (h), $ST_4$  $31^3$ in (i) and a 16-vertex graph $31^3$ in (j), where the first two cases are shown not to be egc by arguments presented in the text, and the last two cases are explicitly shown to be egc.}
\label{f5}
\end{figure}

This way,
$$jj=a_j\cap a_{j+4}\cap b_j, j(j+2)=b_j\cap b_{j+2}\cap c_j, j(j+3)=a_j\cap c_j\cap c_{j+1}\mbox{ and }j(j+4)=a_{j+4}\cap b_{j+2}\cap c_{j+1},$$ 
$\forall j\in\mathbb{Z}_7.$
We show there is no proper edge coloring of $H$ that is tight on every 4-cycle. To prove this, we recur to the bipartite graph $\GA(H)$ whose parts $V_1$ and $V_2$ are respectively the
twenty-eight edges and twenty-one 4-cycles of $H$ with adjacency between an edge $ij$ of $H$ and a 4-cycle $C$ of $H$ whenever $C$ passes through $ij$.

$\GA(H)$ is represented in
Fig.~\ref{f4}(b) with $a_i$ written as $ai$ ($i=1,2,3$). A tight factorization of $H$ would be equivalent to a 4-coloring of $\GA(H)$ that is monochromatic on each vertex of $V_1$ but covering the four colors at the edges incident to  each vertex of $V_2$. We begin by coloring the edges incident to vertices $00,02,03,04$ respectively with colors black, red, blue and green. This forces the coloring of the subgraph of $\GA(H)$ in the lower left of Fig.~\ref{f4}. By transferring this coloring to the representation of $\GA(H)$ on the right of Fig.~\ref{f4}, as shown, it is verified that vertex 15 on the bottom of the representation does not admit properly any of the four used colors.

\section{Prisms of types 4443, 3221 and 3221}\label{s3}

\noindent Given a  graph $\Gamma'$ the prism graph $\Prism(\Gamma')$ of $\Gamma'$ is the graph cartesian product $K_2\square\Gamma'$. The cases of $(e_1)(e_2)(e_3)(e_4)$-graphs with $(e_1)(e_2)(e_3)(e_4)$ $\ne 1^4$ apart from those treated in Section~\ref{s2}, are the prisms $\Gamma=Prism(\Gamma')$ of $(e_1)(e_2)(e_3)$-graphs $\Gamma'$.
It is easy to see that there is no egc graph $\Gamma$ if $g(\Gamma')$ is odd.

\begin{figure}[htp]
\includegraphics[scale=0.665]{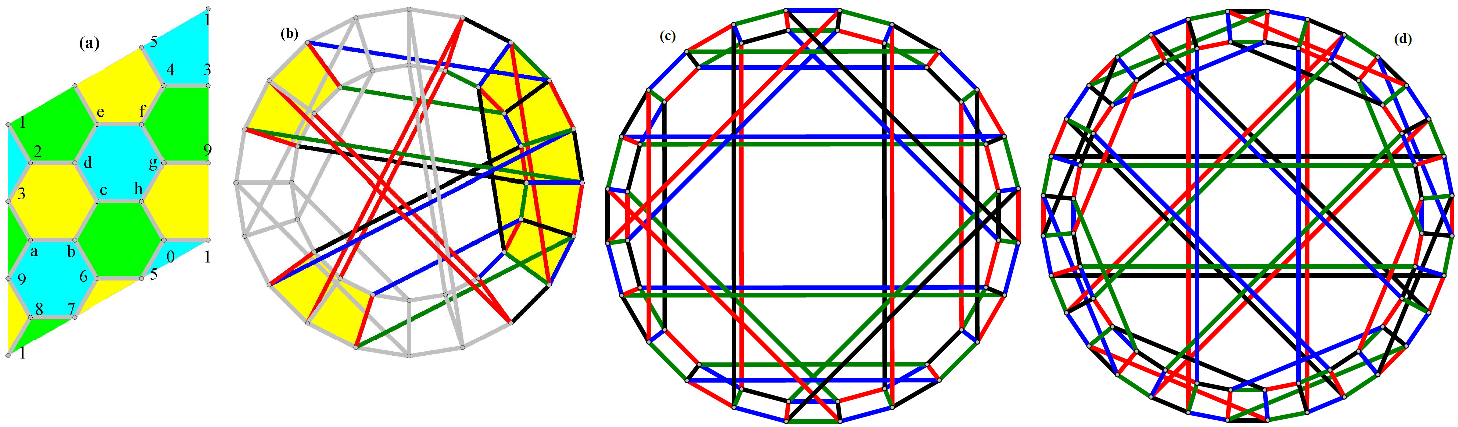}
\caption{A Pappus-graph cutout and the Desargues, Nauru and Dyck graph prisms.}
\label{f6}
\end{figure}

\begin{conjecture}\label{2egcs}
Graphs $\Gamma$ with signatures $32^21$ and $4^33$ are not egc.\end{conjecture}

\begin{example} Conjecture~\ref{2egcs} is sustained by the exhaustive partial colorings of the prisms of the 24-vertex truncated octahedral graph \cite[pp.~79--86]{Crom} in Fig.~\ref{f5}(a--g) and of the 6-vertex Thomsen graph $K_{3,3}$ \cite{Thomsen} in Fig.~\ref{f5}(h), which are respectively a $32^21$-graph and a $4^33$-graph, with the incidental obstructions indicated by a notification "No!" in each case.
Such exhaustive partial colorings can be found similarly for example in the 120-vertex truncated-icosidodecahedral graph \cite[pp.~97--99]{Crom}.\end{example}

\begin{figure}[htp]
\hspace*{40mm}
\includegraphics[scale=0.35]{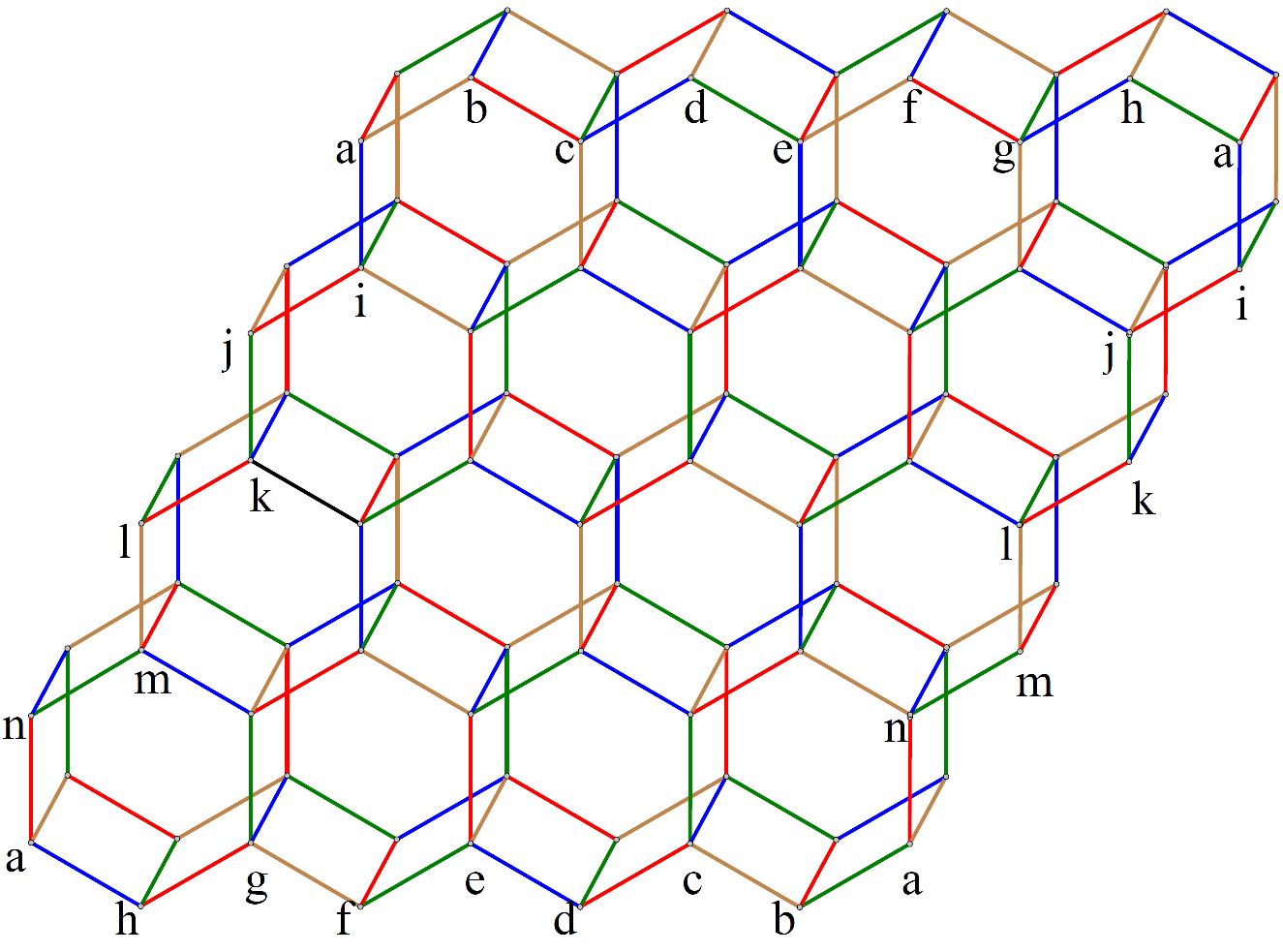}
\caption{An edge-girth coloring of the prism of $\{6,3\}_{|4,4|}$ on the Klein bottle. Note that the
colorings of the hexagonal prisms in which one color appears six times only occur in the first
row, and no edge-girth coloring of this graph is possible if only such colorings are used.}
\label{ff5}
\end{figure}

\begin{remark}\label{31^3} A $31^3$-graph $\Gamma=K_2\square\Gamma'$ where $\Gamma'$ is a toroidal quotient graph of the hexagonal tessellation \cite{CS,Pal} (i.e. the tiling of the plane with Schl\"afli symbol $\{6,3\}$), may be an egc $31^3$-graph. This is exemplified in Fig.~\ref{f5}(i--j),
namely for the prisms of the 24-vertex star graph $ST_4$ (with twelve girth 6-cycles) \cite{DS} and a 16-vertex graph (with just eight girth 6-cycles).
But $\Gamma'$ cannot be the Pappus graph. Let us see why.  A cutout in this case, as in Fig.~\ref{f6}(a) (with octodecimal vertex notation and proper face coloring) contains nine hexagonal tiles.
In order to use the specified coloring to guarantee the existence of an egc-graph, the ``period" employed when moving from any particular vertex $v_R$ of a cutout $R$ in any of the three directions perpendicular to the edges of the tessellation -- i.e., the number of tiles met until a similar vertex $v_{R'}$ in a cutout $R'$ adjacent to $R$ is reached -- must be even, but 9 is odd, leading to a contradiction.\end{remark}

\begin{remark}\label{junction} In the setting of Remark~\ref{31^3}, by considering the junction of three hexagonal prisms $P_1,P_2,P_3$ it is seen that in any such $P_i$  ($i=1,2,3$), two edges not belonging to a hexagon
can only be of the same color if their endpoints are antipodal in the two hexagons of $P_i$. Up to automorphism and permutation of colors, this allows for two distinct colorings of the hexagonal prisms:

\begin{enumerate}\item the one as in both instances of Fig.~\ref{f5}(i--j) (with one color appearing on six edges and the remaining four colors appearing each on four edges), and
\item one with two colors appearing each on five edges and the other two colors appearing each on four edges. \end{enumerate}

\noindent See Fig.~\ref{ff5}, explained in Example~\ref{laconj}, below. Computational evidence has been obtained that gives support to the following conjectures.\end{remark}

\begin{conjecture}\label{att} The condition of even periods in Remark~\ref{31^3} is sufficient for
the case of prisms of hexagonal tessellations of the torus.\end{conjecture}

\begin{conjecture}\label{att2}
The prisms of the hexagonal tessellations of the Klein bottle
\cite{umkb} are egc if the cutout
contains $m\times n$ tiles, with $m$ and $n$ even, and having both types of colorings of the hexagonal
prism as in Remark~\ref{junction}.\end{conjecture}

\begin{example}\label{laconj} Conjecture~\ref{att2} is sustained by the prism cutout in Fig.~\ref{ff5},
showing an edge-girth coloring of the prism of $\{6,3\}_{|4;4|}$ \cite{umkb} on the Klein bottle. Note that the
colorings of the hexagonal prisms in which one color appears six times only occur in the first
row, and no edge-girth coloring of this graph is possible if only such colorings are used.
\end{example}

\begin{example}\label{umkb} In Fig.~\ref{f6}(b), the prism of the Desargues graph on twenty vertices is shown non-egc via obstructions by pairs of forced ``long" parallel red edges.
However, in Fig.~\ref{f6}(c--d) the Nauru and Dyck graphs on twenty-four and thirty-two vertices, respectively, are shown to have their prisms as egc graphs by means of corresponding tight factorizations.\end{example}

\begin{example} For the case $g(\Gamma')=8$ let us consider $\Gamma'$ to be the Tutte 8-cage on 30 vertices. Fig.~\ref{f7}(a--c) shows why its $31^3$-graph prism is not egc, with five 8-cycle prisms $K_2\square C_8$ in $\Gamma$ (presented cyclically mod 5), each of whose vertices should have its four incident edges colored differently.
In fact, Fig.~\ref{f7}(a--c) presents exhaustively without loss of generality partial edge-colorings in $\Gamma$ with copies of $K_2\square C_8$ edge-colored accordingly
and notification ``No!" if an obstruction to edge-coloring continuation appears.\end{example}

\begin{table}[htp]
$$\begin{array}{l}
\!\hspace*{1.4mm}
\bar{3}\hspace*{4.61mm}\bar{3}\hspace*{4.61mm}
4\hspace*{4.61mm}1\hspace*{4.61mm}2\hspace*{4.61mm}3\hspace*{4.61mm}
\bar{4}\hspace*{4.61mm}\bar{3} \hspace*{4.61mm}
4\hspace*{4.61mm}2\hspace*{4.61mm}1\hspace*{4.61mm}3\hspace*{4.61mm}
\bar{4}\hspace*{4.61mm}\bar{4}\hspace*{4.61mm}
3\hspace*{4.61mm}2\hspace*{4.61mm}1\hspace*{4.61mm}4\hspace*{4.61mm}
\bar{3}\hspace*{4.61mm}\bar{4}\hspace*{4.61mm}
3\hspace*{4.61mm}1\hspace*{4.61mm}2\hspace*{4.61mm}4\\

\!\,\,\circ\, 4\hspace*{0.335mm}\circ\hspace*{0.335mm}1\hspace*{0.335mm}\circ\hspace*{0.335mm}2\hspace*{0.335mm}\circ\hspace*{0.335mm}3\hspace*{0.335mm}\circ\hspace*{0.335mm}4\hspace*{0.335mm}\circ\hspace*{0.335mm}1\hspace*{0.335mm}\circ\hspace*{0.335mm}2\hspace*{0.335mm}\circ\hspace*{0.335mm}1\hspace*{0.335mm}\circ\hspace*{0.335mm}3\hspace*{0.335mm}\circ\hspace*{0.335mm}4\hspace*{0.335mm}\circ\hspace*{0.335mm}2\hspace*{0.335mm}\circ\hspace*{0.335mm}1\hspace*{0.335mm}\circ\hspace*{0.335mm}3\hspace*{0.335mm}\circ\hspace*{0.335mm}2\hspace*{0.335mm}\circ\hspace*{0.335mm}1\hspace*{0.335mm}\circ\hspace*{0.335mm}4\hspace*{0.335mm}\circ\hspace*{0.335mm}3\hspace*{0.335mm}\circ\hspace*{0.335mm}2\hspace*{0.335mm}\circ\hspace*{0.335mm}1\hspace*{0.335mm}\circ\hspace*{0.335mm}2\hspace*{0.335mm}\circ\hspace*{0.335mm}4\hspace*{0.335mm}\circ\hspace*{0.335mm}3\hspace*{0.335mm}\circ\hspace*{0.335mm}1\hspace*{0.335mm}\circ\hspace*{0.335mm}2\\

\!\,1\hspace*{0.7mm}\square\hspace*{0.7mm}2\hspace*{0.7mm}\square\hspace*{0.7mm}3\hspace*{0.7mm}\square\hspace*{0.7mm}4\hspace*{0.7mm}\square\hspace*{0.7mm}1\hspace*{0.7mm}\square\hspace*{0.7mm}2\hspace*{0.7mm}\square\hspace*{0.7mm}3\hspace*{0.7mm}\square\hspace*{0.7mm}4\hspace*{0.7mm}\square\hspace*{0.7mm}2\hspace*{0.7mm}\square\hspace*{0.7mm}1\hspace*{0.7mm}\square\hspace*{0.7mm}3\hspace*{0.7mm}\square\hspace*{0.7mm}4\hspace*{0.7mm}\square\hspace*{0.7mm}2\hspace*{0.7mm}\square\hspace*{0.7mm}1\hspace*{0.7mm}\square\hspace*{0.7mm}4\hspace*{0.7mm}\square\hspace*{0.7mm}3\hspace*{0.7mm}\square\hspace*{0.7mm}2\hspace*{0.7mm}\square\hspace*{0.7mm}1\hspace*{0.7mm}\square\hspace*{0.7mm}4\hspace*{0.7mm}\square\hspace*{0.7mm}3\hspace*{0.7mm}\square\hspace*{0.7mm}1\hspace*{0.7mm}\square\hspace*{0.7mm}2\hspace*{0.7mm}\square\hspace*{0.7mm}4\hspace*{0.7mm}\square\hspace*{0.7mm}3\hspace*{0.7mm}\square\\

\!\,\,\circ\, 3\hspace*{0.335mm}\circ\hspace*{0.335mm}4\hspace*{0.335mm}\circ\hspace*{0.335mm}1\hspace*{0.335mm}\circ\hspace*{0.335mm}2\hspace*{0.335mm}\circ\hspace*{0.335mm}3\hspace*{0.335mm}\circ\hspace*{0.335mm}4\hspace*{0.335mm}\circ\hspace*{0.335mm}1\hspace*{0.335mm}\circ\hspace*{0.335mm}3\hspace*{0.335mm}\circ\hspace*{0.335mm}4\hspace*{0.335mm}\circ\hspace*{0.335mm}2\hspace*{0.335mm}\circ\hspace*{0.335mm}1\hspace*{0.335mm}\circ\hspace*{0.335mm}3\hspace*{0.335mm}\circ\hspace*{0.335mm}4\hspace*{0.335mm}\circ\hspace*{0.335mm}3\hspace*{0.335mm}\circ\hspace*{0.335mm}2\hspace*{0.335mm}\circ\hspace*{0.335mm}1\hspace*{0.335mm}\circ\hspace*{0.335mm}4\hspace*{0.335mm}\circ\hspace*{0.335mm}3\hspace*{0.335mm}\circ\hspace*{0.335mm}2\hspace*{0.335mm}\circ\hspace*{0.335mm}4\hspace*{0.335mm}\circ\hspace*{0.335mm}3\hspace*{0.335mm}\circ\hspace*{0.335mm}1\hspace*{0.335mm}\circ\hspace*{0.335mm}2\hspace*{0.335mm}\circ\hspace*{0.335mm}4\\

\!\hspace*{1.4mm}
\bar{2}\hspace*{4.61mm}\bar{1}\hspace*{4.61mm}
2\hspace*{4.61mm}3\hspace*{4.61mm}4\hspace*{4.61mm}1\hspace*{4.61mm}
\bar{2}\hspace*{4.61mm}\bar{2} \hspace*{4.61mm}
1\hspace*{4.61mm}3\hspace*{4.61mm}4\hspace*{4.61mm}2\hspace*{4.61mm}
\bar{1}\hspace*{4.61mm}\bar{2}\hspace*{4.61mm}
1\hspace*{4.61mm}4\hspace*{4.61mm}3\hspace*{4.61mm}2\hspace*{4.61mm}
\bar{1}\hspace*{4.61mm}\bar{1}\hspace*{4.61mm}
2\hspace*{4.61mm}4\hspace*{4.61mm}3\hspace*{4.61mm}1\\
\end{array}$$
\caption{Representing a coloring of 8-cycle prisms $\Theta_i$ ($i=1,2,3,4$).}
\label{IV}
\end{table}

On the other hand, Table~\ref{IV} uses the notation of Table~\ref{I} in representing a coloring of the union $U$ of almost four (namely $3\frac{3}{4}$) contiguous 8-cycle prisms $\Theta_i$ ($i=1,2,3,4$) and the resulting forced colors for the departing edges away from $U$.
In Table~\ref{IV}, the middle row sequence, call it $\Upsilon$ (obtained by disregarding the symbols ``$\square$", or replacing them by commas) represents the subsequences of colors of the edges $\{(0,u),(1,u)\}$ in the prisms $\Theta_i$
 namely the subsequences $\Upsilon_1=(1,2,3,4)^2$ $\Upsilon_2=(3,4,2,1)^2$ $\Upsilon_3=(2,1,4,3)^2$ and $\Upsilon_4=(4,3,1,2)^2$. Here, the last two terms of each $\Upsilon_i$ coincide (i.e. are shared) with the first two terms of its subsequent $\Upsilon_{i+1}$ where the last 6-term subsequence is completed to $\Upsilon_4$ by adding the first two terms of $\Upsilon_1$ so $\Upsilon_1$ may be considered as the next $\Upsilon_i$ after the last $\Upsilon_4$ (and explaining the fraction $3\frac{3}{4}$ mentioned above). This suggest that $\Upsilon$ can be concatenated with itself a number $\ell$ of times to close a $(24\times\ell)$-cycle of colors for the edges $(\{0,u),(1,u)\}$ of a Hamilton-cycle (of $\Gamma'$) prism $H$ which may be completed to an egc graph $\Gamma$ by means of the following considerations. (An adequately colored graph $\Gamma'$ is obtained from Fig.~\ref{f7}(d) by adding a suitable colored outer cycle, missing in the figure).

On the top and bottom rows of Table~\ref{IV}, the colors 1, 2, 3 and 4 with a bar on top are those of the ``long" edges in the four prisms $\Theta_i$ that close the two 8-cycles in each $\Theta_i$. The remaining (non-barred) colors suggest that the corresponding edges form external 4-cycles that may be joined  with $H$ to form a $\Gamma$ as desired. The ``even longer" edges of these external 4-cycles must be set to form (with two edges of the form $\{(0,u),(1,u)\}$) new 4-cycles and can be selected to form the desired $\Gamma$ by taking the number of concatenated copies of $U$ to be $\ell=4$ so that $|V(\Gamma)|=192$. The two columns in Table~\ref{IV} whose transpose rows are ``$4\circ 3\circ 2$" and ``$4\circ\!1\!\circ\!2$", namely the third leftmost and seventh rightmost $\square$-free columns, integrate one such 4-cycle. The leftmost third, fourth, fifth and sixth columns are paired this way with the rightmost seventh, eighth, ninth and tenth columns, but the last three pairs must be paired with similar columns in the second, third and fourth version of Table~\ref{IV} (for indices $k\in\{2,3,4=\ell\}$ of copies $U_k$ of $U$ if we agree that the leftmost third and rightmost seventh columns above are both for $k=1$ and $U=U_k=U_1$). The same treatment can be set from  the leftmost ninth, tenth, eleventh and twelveth respectively to the rightmost first, second, third and fourth columns, which also correspond in pairs that form again ``long" 4-cycles.

\begin{figure}[htp]
\hspace*{2.6cm}
\includegraphics[scale=0.8]{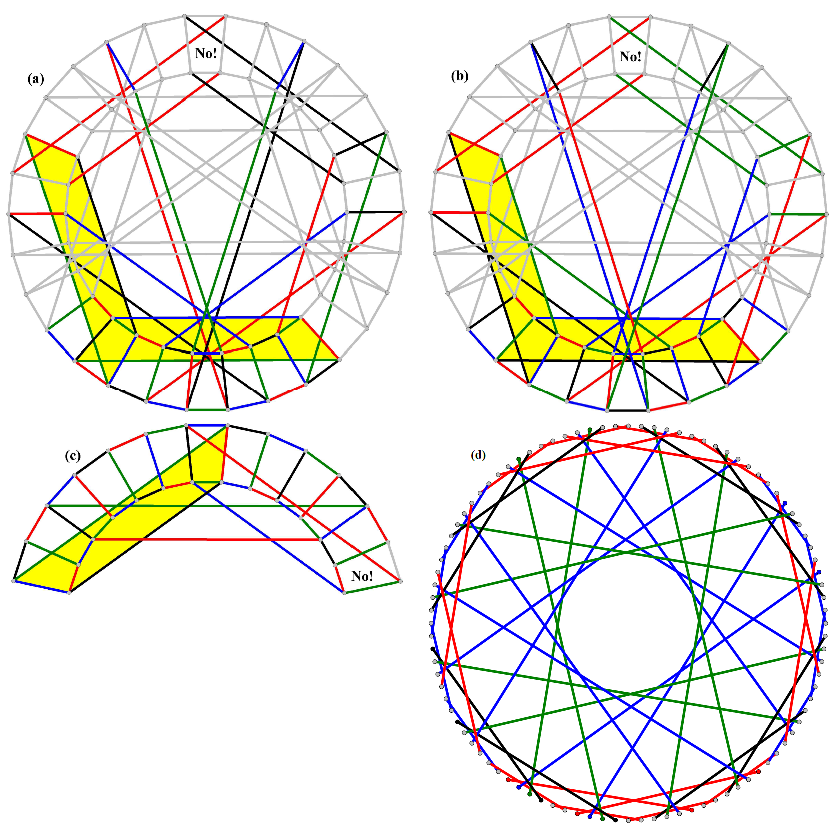}
\caption{Tutte 8-cage 3$1^3$-graph prism in (a), (b) and (c), and 96-vertex cubic $\Gamma'$ in (d), where (a), (b) and (c) show why the $31^3$-graph prism of the Tutte 8-cage is not egc, having five 8-cycle prisms $K_2\square C_8$ in $\Gamma$ (presented cyclically mod 5), each of whose vertices should have its four incident edges colored differently.
In fact, (a), (b) and (c) present exhaustively without loss of generality partial edge-colorings in $\Gamma$ with copies of $K_2\square C_8$ edge-colored accordingly
and notification ``No!" if an obstruction to edge-coloring continuation appears.}
\label{f7}
\end{figure}

\begin{theorem}\label{prismas} For each $4\le k\in\mathbb{Z}$ there is an egc $31^3$-graph $\Gamma$ with $192\times k$ edges as a prism of a hamiltonian cubic graph $\Gamma'$ on $96\times k$ vertices based on $g$ contiguous copies of the edge-colored subgraph in Table~\ref{IV}. However, the Tutte 8-cage is a non-egc $31^3$-graph.
\end{theorem}

\begin{proof}
The argument above the statement can be completed for the case $k=1$. By concatenating the graph from Table~\ref{IV} any multiple of $g$ times, one extends the construction.
\end{proof}

\begin{remark}\label{PSV} The cubic vertex-transitive graphs on less than one hundred vertices with girth 10 and that have egc prisms are in the notation of \cite{census}:

 CubicVT[80,30], CubicVT[96,34], CubicVT[96,49], CubicVT[96,50] and CubicVT[96,62].\end{remark}

\section{Egc 1111-graphs}\label{1111}

A construction \cite{WP} of $1^4$-graphs, also called {\it girth-tight} \cite{PW}, proceeds as follows.
Let $\Gamma$ be 4-regular and let $\mathcal C$ be a partition of $E(\Gamma)$ into cycles. The pair  $(\Gamma,{\mathcal C})$ is a {\it cycle decomposition} of $\Gamma$. Two edges of $\Gamma$ are {\it opposite} at vertex $v$ if both are incident to $v$ and belong to the same element of $\mathcal C$. The {\it partial line graph} $\mathbb{P}(\Gamma,{\mathcal C})$ of $(\Gamma,{\mathcal C})$ is the graph  with the edges of $\Gamma$ as vertices, and any two such vertices adjacent if they share, as edges, a vertex of $\Gamma$ and are not opposite at that vertex.
A cycle $C$ in $\Gamma$ is $\mathcal C${\it -alternating} if no two consecutive edges of $C$ belong to the same element of $\mathcal C$. Lemma 4.10 \cite{PW} says that $\mathbb{P}(\Gamma,{\mathcal C})$ is girth-tight if and only if $(\Gamma,{\mathcal C})$
contains neither $\mathcal C$-alternating cycles nor triangles, except those contained in $\mathcal C$.

\begin{figure}[htp]
\includegraphics[scale=0.67]{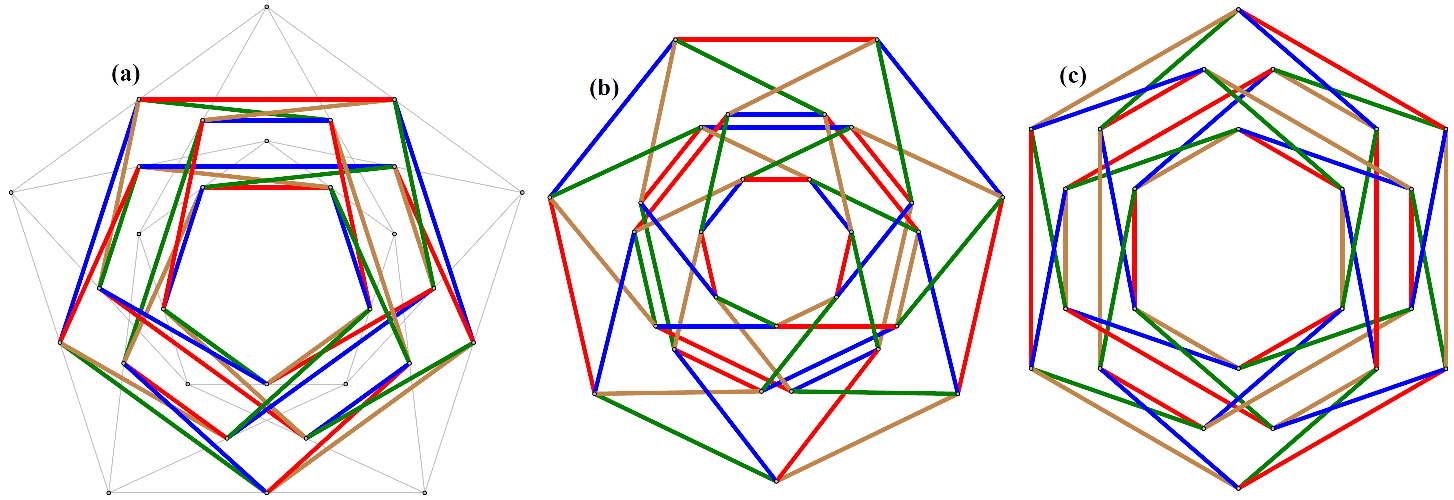}
\caption{Tight factorizations of $\mathbb{P}(W(5,2))$ $\mathbb{P}(W(7,2))$  and $\mathbb{P}(W(6,2))$.}
\label{f8}
\end{figure}

\begin{example}\label{initial}
As initial example of a partial line graph, consider the wreath graph $W(n,2)=C_n[\overline{K_2}]$ ($n>4$), where $C_n$ is a cycle $(v_0,v_1,\ldots,v_{n-1})$. Consider the partition $\mathcal C$ of $W(n,2)$ into the 4-cycles $((v_i,0),(v_{i+1},0),(v_i,1),(v_{i+1},1))$ ($i\in\mathbb{Z}_n$). These form a decomposition $(W(n,2),{\mathcal C})$ which yields the partial line graph $\mathbb{P}(W(n,2),{\mathcal C})$. We prove now that for all values of $n$ $\mathbb{P}(W(n,2),{\mathcal C})$ is egc, as in Fig.~\ref{f8}(a--c), where $\mathbb{P}(W(5,2),{\mathcal C})$ $\mathbb{P}(W(7,2),{\mathcal C})$ and $\mathbb{P}(W(6,2),{\mathcal C})$ are represented, showing tight factorizations via edge colors 1, 2, 3, 4.
\end{example}

\begin{theorem}\label{needed} Let $4<n\in\mathbb{Z}$. Then, $\mathbb{P}(W(n,2),{\mathcal C})$ is egc.
\end{theorem}

\begin{proof}
Each vertex $((v_i,j)(v_{i\pm 1},j'))$ of $\mathbb{P}(W(n,2),{\mathcal C})$ representing the edge between the vertices
$(v_i,j)$ and $(v_{i\pm 1},j')$ of $W(n,2)$ where $i\in\mathbb{Z}_n$ and $j,j'\in\{0,1\}$ will be denoted $(i_j(i\pm 1)_{j'})$. We use modifications of Lemma~\ref{SC} separately for the cases of odd and even $n$. If $n=2k+1$ is odd, then we have a 2-factorization of $\mathbb{P}(W(n,2),{\mathcal C})$ one of whose two 2-factors is composed by  three disjoint even-length cycles
not sharing more than two edges with any 4-cycle, namely one of length 8 and two of length $4k+2$ specifically:

\begin{figure}[htp]
\includegraphics[scale=1.1]{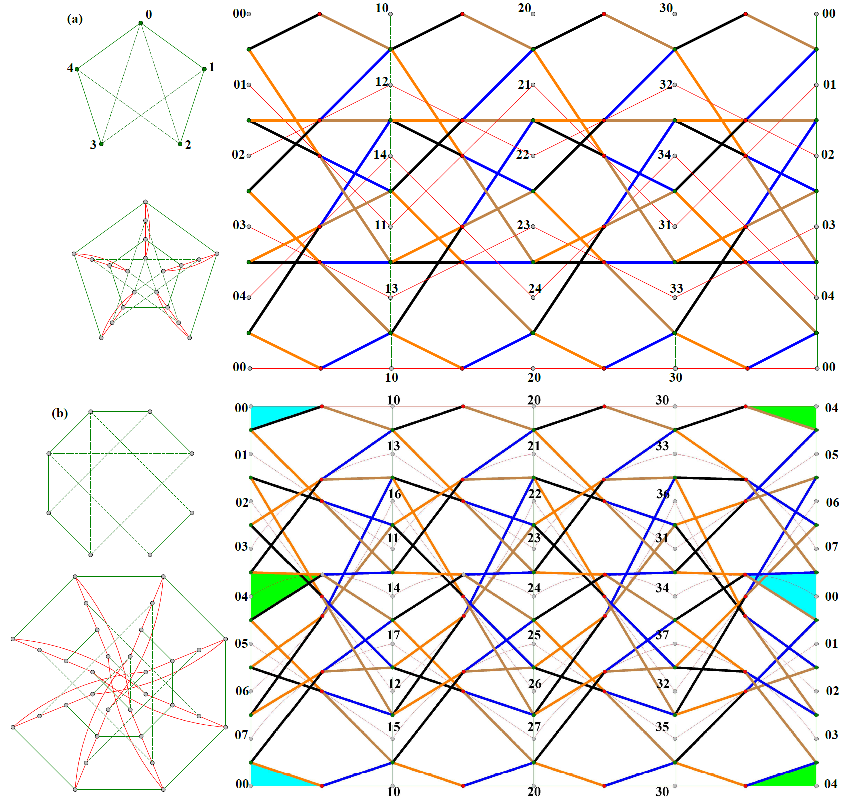}
\caption{Tight factorizations of (a) $\mathbb{P}(\Br(4,5;2))$ and (b) $\mathbb{P}(\MBr(4,8;3))$.}
\label{f9}
\end{figure}

$$\begin{array}{l}
((\!-k_0k_0)(\!-k_0(1\!-k)_1)(\!-k_0k_1)((k\!-1)_0k_1)(-k_1k_1)(\!-k_1(1\!-k)_0)(\!-k_1k_0)((k\!-1)1k_0));\\
((\!-k_0(1\!-k)_0)\cdots(\!-1_00_0)(0_01_0)\cdots((k\!-1)_0k_0)(k_0(k\!-1)_1)\cdots(1_10_1)(0_1-1_1)\cdots(k_1-k_0));\\
((\!-k_1(1\!-k)_1)\cdots(-1_10_1)(0_11_1)\cdots((k\!-1)_1k_1)(k_1(k\!-1)_0)\cdots(1_00_0)(0_0-1_0)\cdots(k_0\!-k_1)),\\
\end{array}$$
which for $k=2$ and $k=3$ can be visualized respectively in Fig.~\ref{f8}(a) and Fig.~\ref{f8}(b) as three alternate-red-blue cycles, one of length 8 and two of length $4k+2$.
The other 2-factor also is formed by even cycles not sharing more than two edges with any 4-cycle, viewable as alternate-green-hazel cycles for $k=2$ in Fig.~\ref{f8}(a) and for $k=3$ in Fig.~\ref{f8}(b).
This 2-factor has reflective $\mathbb{Z}_2$ symmetry on a vertical axis. As for the mentioned modifications of Lemma~\ref{SC}, note that
 the two cycles of length $4k+2$ differ, either for $k$ odd or for $k$ even differ: If $k$ is odd, the two cycles of length $2k+2$ contain opposite edges in the 4-cycles, while if $k$ is even, the two cycles of length $2k+2$ share just one edge with each 4-cycle. Both cases refine into corresponding tight 1-factorizations.
In particular, if $n=2k$ is even, then a 2-factorization of $\mathbb{P}(W(n,2),{\mathcal C})$ is formed by $k$ cycles of length 8 forming a class of cycles$\pmod k$ namely
$$
(\!(i_0(i\!+1)_0\!)(\!(i\!+1)_0(i\!+2)_0\!)(\!(i\!+2)_0(i\!+1)_1\!)(\!(i\!+1)_1i_0\!)(\!(i\!+1)_1i_1\!)(\!(i\!+1)_1(i\!+2)_1)(\!(i\!+2)_1(i\!+1)_0)(\!(i\!+1_0)i_1)\!),
$$
where $i$ is odd, $0<i<n$.
The other 2-factor also is formed by 8-cycles, see Fig.~\ref{f8}(c).
\end{proof}

In order to obtain additional girth-tight graphs with tight factorizations, we recur to a particular case of a cycle decomposition known as {\it linking-ring structure} \cite{WP}, that works for two colors, say red and green. This  structure applies in the following paragraphs only for $n$ even;
(if $n$ is odd, then more than two colors would be needed in order to distinguish adjacent cycles of the decomposition $(W(n,2),{\mathcal C})$).
A {\it linking-ring structure} is
defined in items (i)--(iii) below, as follows. An {\it isomorphism between two cycle decompositions} $(\Gamma_1,{\mathcal C}_1)$ and $(\Gamma_2,{\mathcal C}_2)$ is an isomorphism $\xi:\Gamma_1\rightarrow\Gamma_2$ such that $\xi({\mathcal C}_1)={\mathcal C}_2$. An isomorphism $\xi$ from a cycle decomposition to itself is an {\it automorphism}, written $\xi\in Aut(\Gamma,{\mathcal C})$.
A cycle decomposition
$(\Gamma,{\mathcal C})$ is {\it flexible} if for every vertex $v$ and
each edge $e$ incident to $v$ there is $\xi\in Aut(\Gamma,{\mathcal C})$ such that:
\begin{enumerate}
\item[{\bf(I)}] $\xi$ fixes each vertex of the cycle in
${\mathcal C}$ containing $e$ and
\item[{\bf(II)}] $\xi$ interchanges the two other neighbors of $v$; the edges joining $v$ to
those neighbors are in some other cycle of $\mathcal C$.
\end{enumerate}
A cycle decomposition $(\Gamma,{\mathcal C})$ is {\it bipartite} if $\mathcal C$ can be partitioned into two subsets
$\mathcal G$ (green) and $\mathcal R$ (red) so that each vertex of
$\Gamma$ is in one cycle of $\mathcal G$ and one cycle of $\mathcal R$.

The largest subgroup of $\Aut(\Gamma,{\mathcal C})$
 preserving each of the sets ${{\mathcal C}_1}={\mathcal G}$ ($\mathcal G$ for ``green"), and ${\mathcal C}_2={\mathcal R}$ ($\mathcal R$ for ``red"), is denoted $\Aut^+(\Gamma,{\mathcal C})$.
In a bipartite cycle decomposition, an element of $\Aut(\Gamma,{\mathcal C})$
 either interchanges
$\mathcal G$ and $\mathcal R$ or preserves
each of $\mathcal G$ and $\mathcal R$ set-wise, so it is contained in $\Aut^+(\Gamma,{\mathcal C})$.
This shows that the index of $\Aut^+(\Gamma,{\mathcal C})$ in $\Aut(\Gamma,{\mathcal C})$ is at most 2. If
this index is 2, then we say that $(\Gamma,{\mathcal C})$ is {\it self-dual}; this happens if and only if
there is
$\sigma\in Aut(\Gamma,{\mathcal C})$ such that ${\mathcal G}\sigma={\mathcal R}$ and
${\mathcal R}\sigma={\mathcal G}$.
In \cite{PW}, a cycle decomposition
$(\Gamma,{\mathcal C})$ is said to be a {\it linking-ring (LR) structure} if it is
\begin{enumerate}\item[{\bf(i)}] bipartite, \item[{\bf(ii)}] flexible and \item[{\bf(iii)}]
$\Aut^+(\Gamma,{\mathcal C})$ acts transitively
on $V(\Gamma)$.\end{enumerate}
However, there are tight factorizations of girth-tight graphs $\mathbb{P}(\Gamma,{\mathcal P})$ obtained by relaxing condition (iii) in that definition. So we will say that a cycle decomposition $(\Gamma,{\mathcal P})$ is a {\it relaxed LR structure} if it satisfies just conditions (i) and (ii).

\begin{table}[htp]
$$\begin{array}{|c|c|c|c|}\hline
(0^0_1[a]^0_11[b]0^1_2[c]^3_01[d])&(1^0_2[a]^1_22[b]1^2_4[c]^0_12[d])&(2^0_1[a]^2_31[b]2^1_2[c]^1_21[d])&(3^0_2[a]^3_02[b]3^2_4[c]^2_32[d])\\
(0^1_2[a]^0_12[b]0^2_3[c]^3_02[d])&(1^1_3[a]^1_23[b]1^3_0[c]^0_13[d])&(2^1_2[a]^2_32[b]2^2_3[c]^1_22[d])&(3^1_3[a]^3_03[b]3^3_0[c]^2_33[d])\\
(0^2_3[a]^0_13[b]0^3_4[c]^3_03[d])&(1^2_4[a]^1_24[b]1^4_1[c]^0_14[d])&(2^2_3[a]^2_33[b]2^3_4[c]^1_23[d])&(3^2_4[a]^3_04[b]3^4_1[c]^2_34[d])\\
(0^3_4[a]^0_14[b]0^4_0[c]^3_04[d])&(1^3_0[a]^1_20[b]1^0_2[c]^0_10[d])&(2^3_4[a]^2_34[b]2^4_0[c]^1_24[d])&(3^3_0[a]^3_00[b]3^0_2[c]^2_30[d])\\
(0^4_0[a]^0_10[b]0^0_1[c]^3_00[d])&(1^4_1[a]^1_21[b]1^1_3[c]^0_11[d])&(2^4_0[a]^2_30[b]2^0_1[c]^1_20[d])&(3^0_1[a]^3_01[b]3^1_3[c]^2_31[d])\\\hline
\end{array}$$
\caption{A code representation of the tight factorization in Fig.~\ref{f9}(a).}
\label{V}
\end{table}

\begin{remark}\label{a-b}
With the aim of yielding semisymmetric graphs from LR structures, \cite{WP} defines:
\begin{enumerate}
\item[{\bf(a)}] the {\it barrel} $\Br(k,n;r)$ where $4\le k\equiv 0\pmod{2}$ $n\ge 5$ $r^2\equiv\pm 1\pmod{n}$
$r\not\equiv\pm 1\pmod{n}$ and $0\le r<\frac{n}{2}$ as the graph with vertex set $\mathbb{Z}_k\times\mathbb{Z}_n$ and $(i,j)$ red-adjacent to $(i\pm 1,j)$ and
green-adjacent to $(i,j\pm r^i)$;
\item[{\bf(b)}]
the {\it mutant barrel} $\MBr(k,n;r)$ where $2\le k\equiv n\equiv 0\pmod{2}$
$n\ge 6$ $r^2\equiv\pm 1\pmod{n}$ and $r\not\equiv\pm 1\pmod{n}$ as the graph with vertex set $\mathbb{Z}_k\times\mathbb{Z}_n$ and $(i,j)$ red-adjacent to $(i+1,j)$ for $0\le i<k-1$ $(k-1,j)$ red-adjacent to $(0,j+\frac{n}{2})$ and $(i,j)$ green-adjacent to $(i,j\pm r^i)$.
\end{enumerate}

\begin{figure}[htp]
\hspace*{12.5mm}
\includegraphics[scale=1.24]{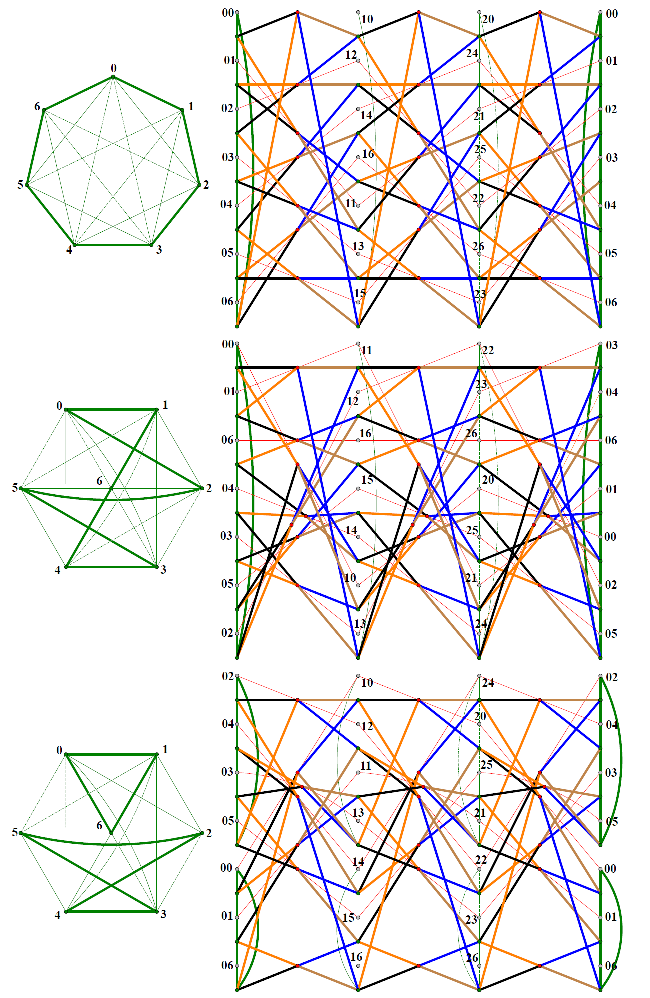}
\caption{Egc  $\mathbb{P}(\Br(3,F))$ for the three 2-factorizations $F$ of $K_7$.}
\label{f10}
\end{figure}

The right side of Fig.~\ref{f9}(a) (resp.~\ref{f9}(b)) represents $\mathbb{P}(\Br(4,5;2))$ (resp. $\mathbb{P}(\MBr(4,8;3))$), where: 
\begin{enumerate}
\item[{\bf(i)}] each vertex $(i,j)$ is denoted $ij$,
\item[{\bf(ii)}] vertices $i0$ appear twice (on top and bottom, to be identified for each $i$), 
\item[{\bf(iii)}] red edges are shown in thin trace, 
\item[{\bf(iv)}]  green edges arising from the cycles $F_1^5=(0,1,2,3,4)$ and $F_2^5=(0,2,4,1,3)$ of $K_5$ (resp. $F_1^8=(0,1,2,3,4,5,6,7)$ and $F_3^8=(0,3,6,1,4,7,2,5)$ of $K_8$) are shown in thin and dashed trace, respectively, and 
\item[{\bf(v)}] the edges of the corresponding partial line graphs are shown in thick trace on the colors orange = $a$, black = $b$, hazel = $c$ and blue = $d$, setting a tight fac\-tor\-iza\-tion.
\end{enumerate}

Vertices of green and red cycles are said to be {\it green} and {\it red}, respectively. To the left of these two graphs in Fig.~\ref{f9}, the corresponding green and red-green subgraphs are shown.\end{remark}

\begin{figure}[htp]
\includegraphics[scale=0.7]{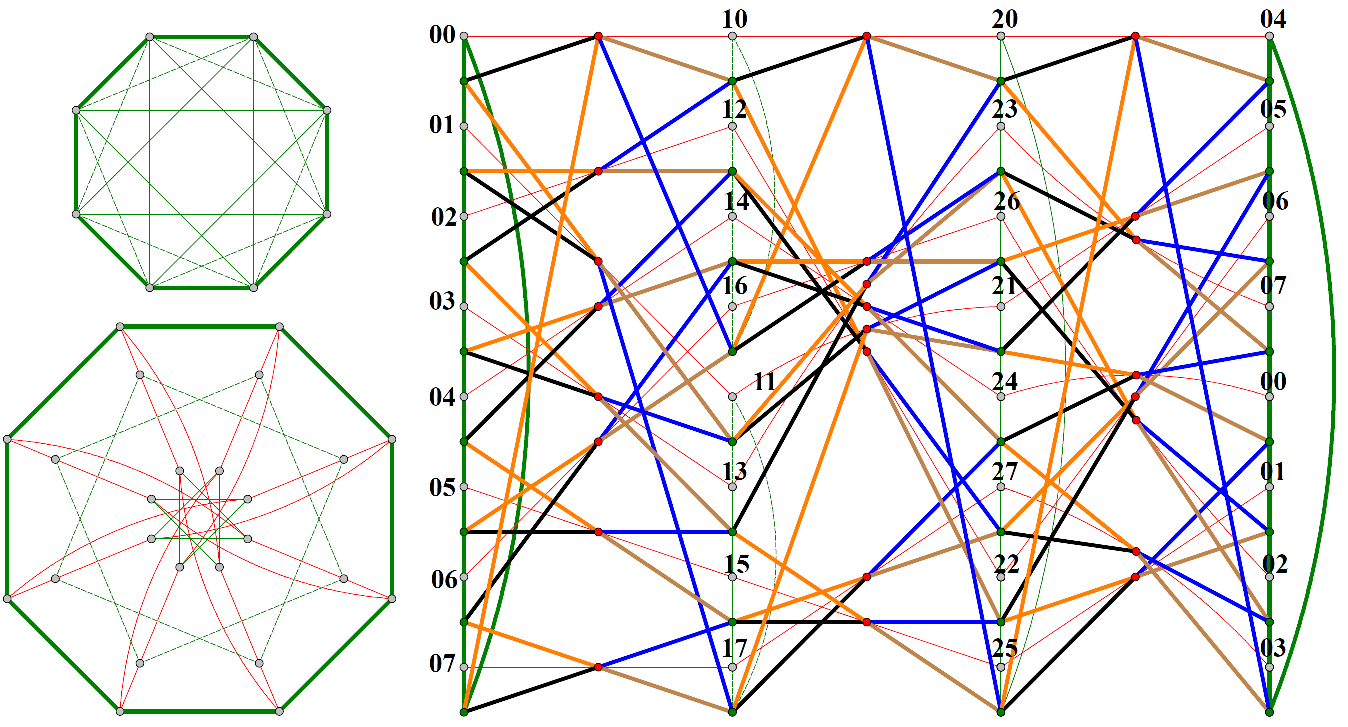}
\caption{Tight factorization of $\mathbb{P}(\MBr(3,\{F_1^8,F_2^8,F_3^8\}))$ based on 2-factors of $K_8$.}
\label{f11}
\end{figure}

\begin{table}[htp]
$$\begin{array}{|c|c|c|}\hline
(0^0_1[a]^0_11[b]0^1_2[c]^2_01[d])&(1^0_2[c]^0_12[d]1^2_4[a]^1_22[b])&(2^0_3[a]^2_03[b]2^3_6[c]^1_23[d])\\
(0^1_2[a]^0_12[b]0^2_3[c]^2_02[d])&(1^1_3[c]^0_13[d]1^3_5[a]^1_23[b])&(2^1_4[a]^2_04[b]2^4_0[c]^1_24[d])\\
(0^2_3[a]^0_13[b]0^3_4[c]^2_03[d])&(1^2_4[c]^0_14[d]1^4_6[a]^1_24[b])&(2^2_5[a]^2_05[b]2^5_1[c]^1_25[d])\\
(0^3_4[a]^0_14[b]0^4_5[c]^2_04[d])&(1^3_5[c]^0_15[d]1^5_0[a]^1_25[b])&(2^3_6[a]^2_06[b]2^6_2[c]^1_26[d])\\
(0^4_5[a]^0_15[b]0^5_6[c]^2_05[d])&(1^4_6[c]^0_16[d]1^6_1[a]^1_26[b])&(2^4_0[a]^2_00[b]2^0_3[c]^1_20[d])\\
(0^5_6[a]^0_16[b]0^6_0[c]^2_06[d])&(1^5_0[c]^0_10[d]1^0_2[a]^1_20[b])&(2^5_1[a]^2_01[b]2^1_4[c]^1_21[d])\\
(0^6_0[a]^0_10[b]0^0_1[c]^2_00[d])&(1^6_1[c]^0_11[d]1^1_3[a]^1_21[b])&(2^6_2[a]^2_02[b]2^2_5[c]^1_22[d])\\\hline
\end{array}$$
$$\begin{array}{|c|c|c|}\hline
(0^0_1[a]^0_11[b]0^1_5[c]^2_05[d])&(1^1_2[c]^0_12[d]1^2_0[a]^1_20[b])&(2^2_3[a]^2_03[b]2^3_1[c]^1_21[d])\\
(0^1_5[a]^0_15[b]0^5_2[c]^2_02[d])&(1^2_0[c]^0_10[d]1^0_3[a]^1_23[b])&(2^3_1[a]^2_01[b]2^1_4[c]^1_24[d])\\
(0^5_2[a]^0_12[b]0^2_4[c]^2_04[d])&(1^0_3[c]^0_13[d]1^3_5[a]^1_25[b])&(2^1_4[a]^2_04[b]2^4_0[c]^1_20[d])\\
(0^2_4[a]^0_14[b]0^4_3[c]^2_03[d])&(1^3_5[c]^0_15[d]1^5_4[a]^1_24[b])&(2^4_0[a]^2_00[b]2^0_5[c]^1_25[d])\\
(0^4_3[a]^0_13[b]0^3_6[c]^2_06[d])&(1^5_4[c]^0_14[d]1^4_6[a]^1_26[b])&(2^0_5[a]^2_05[b]2^5_6[c]^1_26[d])\\
(0^3_6[a]^0_16[b]0^6_0[c]^2_00[d])&(1^4_6[c]^0_16[d]1^6_1[a]^1_21[b])&(2^5_6[a]^2_06[b]2^6_2[c]^1_22[d])\\
(0^6_0[a]^0_10[b]0^0_1[c]^2_01[d])&(1^6_1[c]^0_11[d]1^1_2[a]^1_22[b])&(2^6_2[a]^2_02[b]2^2_3[c]^1_23[d])\\\hline
\end{array}$$
$$\begin{array}{|c|c|c|}\hline
(0^0_1[a]^0_11[b]0^1_6[c]^2_01[d])&(1^0_3[c]^0_13[d]1^3_1[a]^1_21[b])&(2^0_2[a]^2_02[b]2^2_1[c]^1_21[d])\\
(0^1_6[a]^0_16[b]0^6_0[c]^2_06[d])&(1^3_1[c]^0_11[d]1^1_5[a]^1_25[b])&(2^2_1[a]^2_01[b]2^1_4[c]^1_24[d])\\
(0^6_0[a]^0_10[b]0^0_1[c]^2_00[d])&(1^1_5[c]^0_15[d]1^5_0[a]^1_20[b])&(2^1_4[a]^2_04[b]2^4_0[c]^1_20[d])\\
(0^2_3[a]^0_13[b]0^3_4[c]^2_03[d])&(1^5_0[c]^0_10[d]1^0_3[a]^1_23[b])&(2^4_0[a]^2_00[b]2^0_2[c]^1_22[d])\\
(0^3_4[a]^0_14[b]0^4_5[c]^2_04[d])&(1^2_4[c]^0_14[d]1^4_6[a]^1_26[b])&(2^3_5[a]^2_05[b]2^5_6[c]^1_26[d])\\
(0^4_5[a]^0_15[b]0^5_2[c]^2_05[d])&(1^4_6[c]^0_16[d]1^6_2[a]^1_22[b])&(2^5_6[a]^2_06[b]2^6_3[c]^1_23[d])\\
(0^5_2[a]^0_12[b]0^2_3[c]^2_02[d])&(1^6_2[c]^0_12[d]1^2_4[a]^1_24[b])&(2^6_3[a]^2_02[b]2^3_5[c]^1_25[d])\\
\hline
\end{array}$$
\caption{Code representations of $F_1^7$ $F_2^7$ and $F_3^7$.}
\label{VI}
\end{table}

Note that thick edges of colors orange and black form cycles zigzagging between:
\begin{enumerate}
\item[{\bf(A)}] the vertices of each vertical green cycle (excluding the rightmost green cycle) and
\item[{\bf(B)}] their adjacent red vertices to their immediate right.
\end{enumerate}

  Also, note that thick blue and hazel edges form cycles zigzagging between:
 \begin{enumerate}
 \item[{\bf(C)}] the red vertices and
 \item[{\bf(D)}] the vertices of the next vertical green cycle to their right.
 \end{enumerate}

 The girth is realized by 4-cycles with the four colors, with a pair of edges (blue and hazel) to the left of each vertical green cycle and another pair of edges (black and orange) to the corresponding right. This is always attainable, because similar bicolored cycles can always we obtained, generating the desired tight factorizations. For instance, assigning colors $a,b,c,d$ to the edges $(i_j^{j+r^i},\,_i^{i+1}j)$ $(i_j^{j+r^i},\,_i^{i+1}(j+r^i))$ $((i+1)_j^{j+r^i},\,_i^{i+1}(j+r^i))$ and $((i+1)_j^{j+r^i},\,_i^{i+1}j)$ respectively, yields a tight factorization of $\mathbb{P}(\Br(k,n;r))$.

\begin{table}[htp]
$$\begin{array}{|c|c|c|c|}\hline
(0^0_1[a]^0_11[b]0^1_2[c]^3_01[d])&(1^0_2[a]^1_22[b]1^2_4[c]^0_12[d])&(2^0_3[a]^2_33[b]2^3_6[c]^1_23[d])&(3^0_4[a]^3_04[b]3^4_8[c]^2_34[d])\\
(0^1_2[a]^0_12[b]0^2_3[c]^3_02[d])&(1^1_3[a]^1_23[b]1^3_5[c]^0_13[d])&(2^1_4[a]^2_34[b]2^4_7[c]^1_24[d])&(3^1_5[a]^3_05[b]3^5_0[c]^2_35[d])\\
(0^2_3[a]^0_13[b]0^3_4[c]^3_03[d])&(1^2_4[a]^1_24[b]1^4_6[c]^0_14[d])&(2^2_5[a]^2_35[b]2^5_8[c]^1_25[d])&(3^2_6[a]^3_06[b]3^6_1[c]^2_36[d])\\
(0^3_4[a]^0_14[b]0^4_5[c]^3_04[d])&(1^3_5[a]^1_25[b]1^5_7[c]^0_15[d])&(2^3_6[a]^2_36[b]2^6_0[c]^1_26[d])&(3^3_7[a]^3_07[b]3^7_2[c]^2_37[d])\\
(0^4_5[a]^0_15[b]0^5_6[c]^3_05[d])&(1^4_6[a]^1_26[b]1^6_8[c]^0_16[d])&(2^4_7[a]^2_37[b]2^7_1[c]^1_27[d])&(3^0_8[a]^3_08[b]3^8_3[c]^2_38[d])\\
(0^5_6[a]^0_16[b]0^6_7[c]^3_06[d])&(1^1_8[a]^1_27[b]1^7_0[c]^0_17[d])&(2^5_8[a]^2_38[b]2^8_2[c]^1_28[d])&(3^1_0[a]^3_00[b]3^0_4[c]^2_30[d])\\
(0^6_7[a]^0_17[b]0^7_8[c]^3_07[d])&(1^2_0[a]^1_28[b]1^8_1[c]^0_18[d])&(2^6_0[a]^2_30[b]2^0_3[c]^1_20[d])&(3^2_1[a]^3_01[b]3^1_5[c]^2_31[d])\\
(0^7_8[a]^0_18[b]0^8_0[c]^3_08[d])&(1^3_1[a]^1_20[b]1^0_2[c]^0_10[d])&(2^7_1[a]^2_31[b]2^1_4[c]^1_21[d])&(3^3_2[a]^3_02[b]3^2_6[c]^2_32[d])\\
(0^8_0[a]^0_10[b]0^0_1[c]^3_00[d])&(1^4_2[a]^1_21[b]1^1_3[c]^0_11[d])&(2^8_2[a]^2_32[b]2^2_5[c]^1_22[d])&(3^0_1[a]^3_03[b]3^3_7[c]^2_33[d])\\
\hline
\end{array}$$
\caption{Tight factorization of $\mathbb{P}(4,F^9)$.}
\label{VII}
\end{table}

A code representation of the tight factorization in Fig.~\ref{f9}(a) is given in Table~\ref{V}, where each green edge
$\{ij,i(j+r^i)\}$ in $\Br(k,n;r)$ yields a green vertex $i_j^{j+r^i}$ in $\mathbb{P}(\Br(k,n;r))$ each red edge $\{ij,(i+1)j\}$ in $\Br(k,n;r)$ yields a red vertex $_i^{i+1}j$ in $\mathbb{P}(\Br(k,n;r))$ and the color of an edge between a green vertex and a red vertex is indicated between brackets: $[a]$ for orange, $[b]$ for black, $[c]$ for hazel and $[d]$ for blue.

\begin{remark}\label{gener}
Generalizing Remark~\ref{a-b} to get other egc girth-tight graphs, we consider a 2-factorization $F^n=\{F_1^n,F_2^n,\ldots,F_{k-1}^n\}$ of the complete graph $K_n$ for odd $n=2k+1>6$ and use it to define the {\it barrel}  $\Br(k,F^n)$
with
\begin{enumerate}
\item[{\bf(i)}] $\mathbb{Z}_k\times\mathbb{Z}_n$ as vertex set and
 \item[{\bf(ii)}] edges forming precisely red cycles $((0,i),(1,i),\ldots,(k-1,i))$ where $i\in\mathbb{Z}_n$ and green subgraphs $\{j\}\times F_j^n$ where $j\in\mathbb{Z}_k$.
\end{enumerate}

Fig.~\ref{f10} contains representations of  $\mathbb{P}(\Br(3,F^7))$ for three distinct 2-factorizations $F^7$ of $K_7$ with tight factorizations represented as in Fig.~\ref{f9}, with green cycles so that each vertex $(i,0)=i0$ appears just once (not twice, as in Fig.~\ref{f9}(a--b)). In the three cases, $F_1^7$ $F_2^7$ and $F_3^7$  green edges are traced thick, thin and dashed, respectively. To the left of these representations, the corresponding green subgraphs are shown. Code representations of these three tight factorizations can be found in Table~\ref{VI}, following the conventions of Table~\ref{V}. (A different 1-factorization of $K_7$ that may be used with the same purpose is for example $\{(0,1,2,3,4,5,6),(0,3,5,1,6,2,4),(0,2,5)(1,3,6,4)$).

In the same way, by considering the 2-factorization given in $K_9$ seen as the Cayley graph $C_9(1,2,3,4)$ with $F^9$ formed by the 2-factors $F_1^9,F_2^9,F_3^9,F_4^9$ generated by the respective colors 1, 2, 3, 4, namely Hamilton cycles $F_1^9,F_2^9,F_4^9$ but $F_3^9=3K_3$ we get a tight factorization of $\mathbb{P}(4,F^9)$. This is encoded in Table~\ref{VII} in a similar fashion to that of Tables~\ref{IV}--\ref{V}.

If $n=2k$ is even, a similar generalization takes a 2-factorization $F^n=\{F_1^n,$ $F_2^n,\ldots,F_{k-2}^n\}$ of $K_n-\{i(i+k);i=0,\ldots, k-1\}$ and uses the 1-factor $\{i(i+k);i=0,\ldots, k-1\}$ to get a generalized {\it mutant barrel} $\MBr(k-1,F^n)$ in a likewise fashion to that of item (b) in Remark~\ref{a-b} but
modified now via $F^n$ namely with
\begin{enumerate}
\item[{\bf(i')}] $\mathbb{Z}_{k-1}\times\mathbb{Z}_n$ as vertex set and
\item[{\bf(ii')}] edges forming precisely red cycles $((0,i),(1,i),\ldots,(k-1,i),(0,i+\frac{n}{2}),(1,i+\frac{n}{2}),\ldots,(k-1,i+\frac{n}{2}))$ where $i\in\mathbb{Z}_n$ and green subgraphs $\{j\}\times F_j^n$ where $j\in\mathbb{Z}_k$.
\end{enumerate}

 Fig.~\ref{f11} represents a tight factorization of $\mathbb{P}(\MBr(3,F^8))$ where $F^8=\{F_1^8,F_2^8,F_3^8\}$ represented on the upper left of the figure, is such a 2-factorization, with $F_1^8$ and $F_3^8$ as in Fig.~\ref{f9}, and $F_2^8=(0,2,4,6)(1,3,5,7)$ via corresponding thick, thin and dashed, green edge tracing. On the lower left, a representation of the red-green graph $\MBr(3,F^8)$ is found.

We can further extend these notions of barrel and mutant barrel by taking a cycle $G^n=(G_1^n,G_2^n,\ldots,G_\ell^n)$ of copies of the 2-factors
of $F^n$ where $G_i^n\in F^n$ but with no two contiguous $G_i^n$ and $G_{i+1}^n\pmod{n}$ being the same element of $F^n$. Here, $\ell\ge 3$. This defines a barrel $\Br(\ell,G^n)$ or mutant barrel $\MBr(\ell,G^n)$ ($n$ even in this case) and establishes the following.\end{remark}

\begin{theorem}\label{barrels}
The barrels and mutant barrels obtained in Remark~\ref{gener} produce corresponding egc graphs $\mathbb{P}(\Br(\ell,G^n))$ and $\mathbb{P}(\MBr(\ell,G^n))$.
\end{theorem}

\begin{proof} The zigzagging orange-black and hazel-blue cycles between each pair of contiguous green-vertex and red-vertex columns in the graphs of Fig.~\ref{f9}--\ref{f11} are as in Lemma~\ref{SC}.
\end{proof}

\section{Egc girth-g-regular graphs, g larger than 4}\label{12345}

\begin{theorem}\label{arman} We have the following:
\begin{enumerate}
\item[ {\bf(a)}] the $32$-vertex Armanios-Wells graph $\AW$ \cite{Armanios} \cite[p.~266]{BNC} is an egc $(12)^5$-graphs;
\item[{\bf(b)}] the $36$-vertex Sylvester graph $\Syl$ \cite[p.~223]{BNC} is a $8^5$-graph, but is not egc.
\end{enumerate}
\end{theorem}

\begin{proof} {\bf(a)}
The center of Fig.~\ref{13} represents $\AW$ colored as claimed.
The vertices of $\AW$ are denoted $Xi$ ($X\in\{A,B,C,D\}$; $i\in\mathbb{Z}_8$). An edge-color assignment for $\AW$ is generated mod 4 or$\pmod{\mathbb{Z}_4}$ where $\mathbb{Z}_4=\{0,2,4,6\}\subset\mathbb{Z}_8$ is a subgroup and an ideal of $\mathbb{Z}_8$ as follows:

\begin{align}\label{a0c7}\begin{array}{|c|c||c|c|c|c|}\hline
Red&1&(A0,C7)&(A1,C2)&(B0,D1)&(B1,D2)\\
Black&2&(A0,C1)&(A1,C0)&(B0,B3)&(D0,D1)\\
Blue&3&(A0,D6)&(A1,B4)&(B1,C6)&(C1,D3)\\
Hazel&4&(A0,B3)&(A1,D7)&(B0,C5)&(C0,D2)\\
Green&5&(A0,C6)&(A1,C7)&(B0,B5)&(D1,D2)\\\hline
\end{array}\end{align}

\begin{figure}[htp]
\includegraphics[scale=0.73]{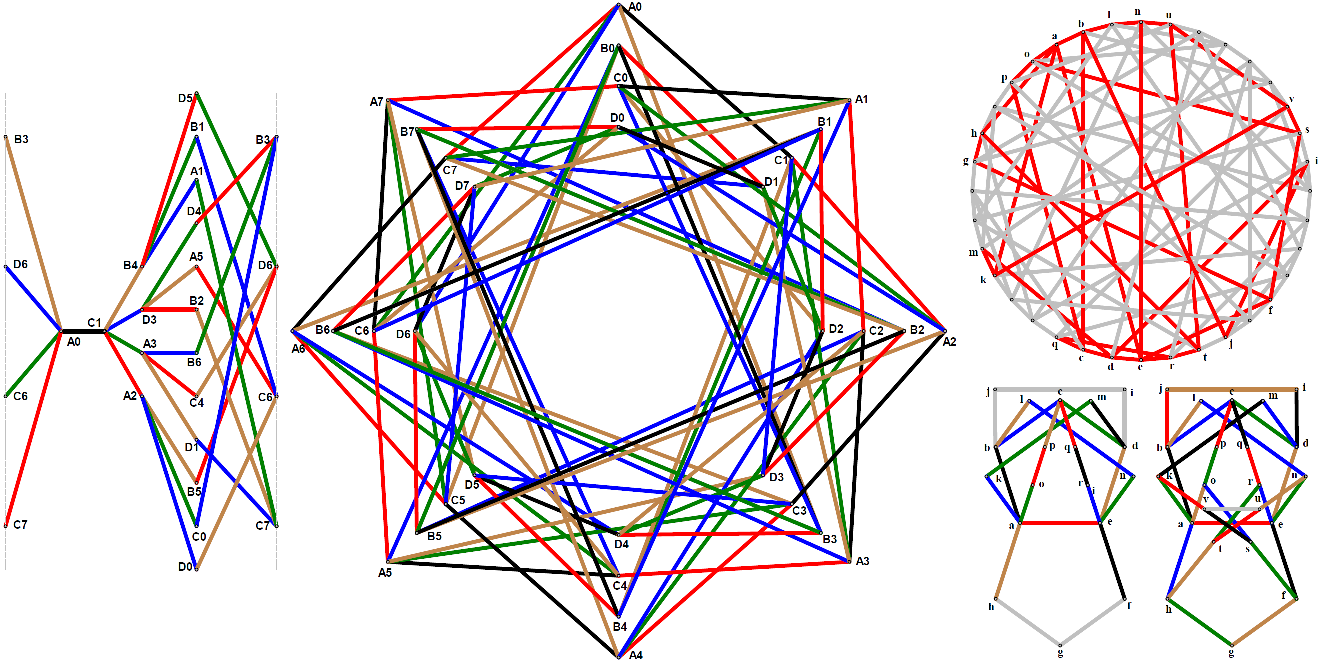}
\caption{The egc Armanios-Wells $(12)^5$-graph and the non-egc Sylvester $8^4$-graph.}
\label{13}
\end{figure}

\noindent where indices are taken$\pmod 8$ so colored-edge orbits$\pmod 4$ are either of the form $$\begin{array}{l}\{(X0,Yi),(X2,Y(i+2)),(X4,Y(i+4),(X6,Y(i+6))\}\mbox{ or of the form:}\\
\{(X1,Yj),(X3,Y(j+2)),(X5,Y(j+4),(X7,Y(j+6))\}\end{array}$$
with $X,Y\in\{A,B,C,D\}$ $i,j\in\mathbb{Z}_8$ and addition taken$\pmod 8$. The left of Fig.~\ref{13} contains the subgraph of $\AW$ spanned by the twelve 5-cycles through the black edge $(A0,C1)$ (where the two dashed lines must be identified), showing the disposition of twelve 5-cycles around an edge of $\AW$. Moreover, the four black edges in the second line of (\ref{a0c7}) represent forty-eight tightly colored 5-cycles, (twelve passing through each black edge, corresponding to the $\frac{5!}{2\times 5}=12$ existing color cycles
$$\begin{array}{c}
(23451),
(23541),
(24351),
(24531),
(25341),
(25431),\\
(24513),
(25413),
(25134),
(24135),
(23514),
(25143))\end{array}$$
and each such cycle yields an orbit of four such 5-cycles.
Since each edge of $\AW$ passes through twelve 5-cycles of $\AW$ and $|E(\AW)|=80$ we count $80\times 12$ 5-cycles in $\AW$ with repetitions. Each 5-cycle in this count is repeated five times, so the number of 5-cycles in $\AW$ is $(80\times 12)/5=16\times 12$. Thus, the number of orbits of tightly-colored 5-cycles is 48 and we obtain a tight coloring of $\AW$.\\
{\bf(b)} The upper right of Fig.~\ref{13} represents $\Syl$ with the following 5-cycles:
$$\begin{array}{lllll}
C_0=(abcde)\!&
C_1=(aefgh)\!&
C_2=(abcgh)\!&
C_3=(cdefg)\!&
C_4=(bcdij)\;
C_5=(hidea)\\
C_6=(hijba)\!&
C_7=(cghid)\!&
C_8=(akmde)\!&
C_9=(ablne)\!&
C_{10}=(abcpo)\\
C_{11}=(cderq)\!&
 C_{12}=(aefso)\!&
 C_{13}=(eahtr)\!&
 C_{14}=(aksvo)\!&
 C_{15}=(enutr).\end{array}$$

The union of $C_0,C_1,C_4,C_8,C_9$ yields the red subgraph.
Coloring tightly $C_0,C_8,C_9,C_{10},$

\noindent $C_{11}$ with color($ab$)=color($ef$) and color($ah$)=color($de$) makes impossible continuing coloring tight\-ly $C_4$ see $\Syl$ as shown in the upper right of Fig.~\ref{13}. Otherwise, in the lower right of Fig.~\ref{13} a forced tight coloring of $C_0,\ldots,C_{12}$ $C_{14}$ and $C_{15}$ is shown in two representations of the subgraph of $\Syl$ induced by these 5-cycles. That leaves $C_{13}$ obstructing a tight-coloring. Thus, $\Syl$ is not egc.
\end{proof}

\begin{figure}[htp]
\includegraphics[scale=0.66]{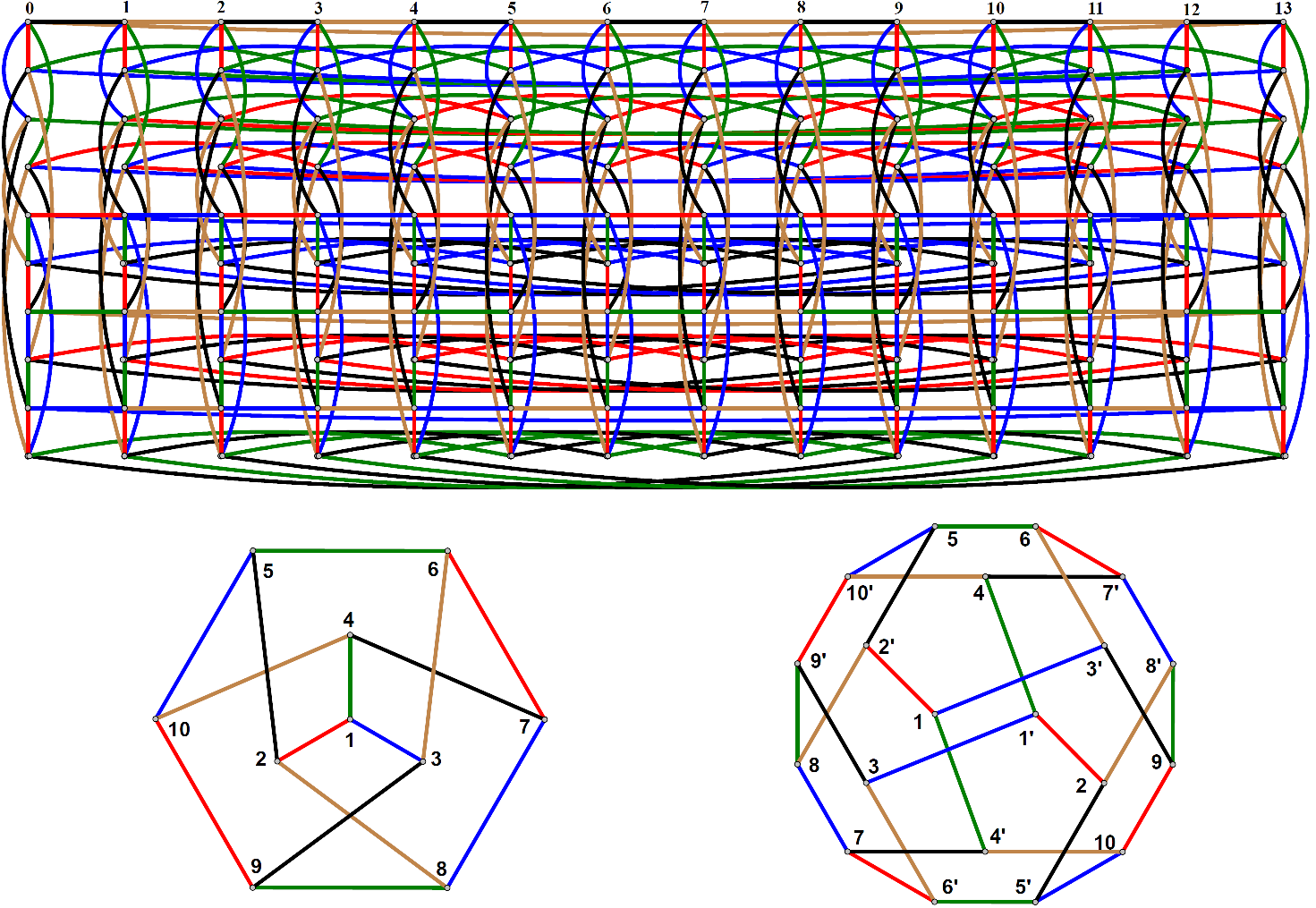}
\caption{Egc $4^30^2$-graph on 140 vertices; Petersen and dodecahedral graphs.}
\label{f12}
\end{figure}

Subsequently, an egc $4^30^2$-graph $\Gamma$ is presented by means of a construction that generalizes the barrel constructions used in Section~\ref{1111}, as follows. See the top of Fig.~\ref{f12}, where fourteen vertical copies of the Petersen graph $\Pet$ are presented in parallel at equal distances from left to right and numbered from 0 to 13 in $\mathbb{Z}_{14}$. The vertices of the $j$-th copy $\Pet^j$ of $\Pet$ are denoted $v_1^j, v_1^j,\ldots,v_{10}^j$ from top to bottom and are joined horizontally by cycles of the Cayley graph of $\mathbb{Z}_{14}$ with generator set $\{1,3,5\}$ namely the cycles

\begin{align}\label{pet}\begin{array}{ll}
(v_i^0,v_i^1,v_i^2,v_i^3,v_i^4,v_i^5,v_i^6,v_i^7,v_i^8,v_i^9,v_i^{10},v_i^{11},v_i^{12},v_i^{13}),&\mbox{ for }i=1,5,7,9;\\
(v_i^0,v_i^3,v_i^6,v_i^9,v_i^{12},v_i^1,v_i^4,v_i^7,v_i^{10},v_i^{13},v_i^2,v_i^5,v_i^8,v_i^{11}),&\mbox{ for }i=2,3,4;\\
(v_i^0,v_i^5,v_i^{10},v_i^1,v_i^6,v_i^{11},v_i^2,v_i^7,v_i^{12},v_i^3,v_i^8,v_i^{13},v_i^4,v_i^9),&\mbox{ for }i=6,8,10.\\
\end{array}\end{align}

\begin{theorem}\label{4^30^2}
There exists an egc $4^30^2$-graph $\Gamma$ of order $140\times k$ for every $0<k\in\mathbb{Z}$ representing all color-cycle permutations, $14k$ times each.
\end{theorem}

\begin{proof}
We assert that an egc $4^30^2$-graph as in the statement contains $14k$ disjoint copies of $\Pet$.
We consider the case represented in the top of Fig.~\ref{f12} for $k=1$ and leave the details of the general case to the reader.
Notice the 6-cycle $(v_5^j,v_6^j,v_7^j,v_8^j,v_9^j,v_{10}^j)$ in $\Pet^j$ with its three pairs of opposite vertices $\{v_5^j,v_8^j\}$ $\{v_6^j,v_9^j\}$ $\{v_7^j,v_{10}^j\}$ joined respectively to the neighbors $v_2^j,v_3^j,v_4^j$ of the top vertex $v_1^j$ for $j\in\mathbb{Z}_{14}$. A representation of the common proper coloring of the graphs $\Pet^j$ is in the lower-left part of Fig.~\ref{f12}, where the vertices $v_i^j$ for $i=1,\ldots,10$ and $j\in\mathbb{Z}_{14}$ are simply denoted $i$. This figure shows the twelve 5-cycles of $\Pet$ as color cycles, with red, black, blue, hazel and green taken respectively as 1, 2, 3, 4 and 5. This gives the one-to-one correspondence, call it $\eta$ from the 5-cycles of $\Pet$ onto their color 5-cycles in the top part of Table~\ref{8}.

There are exactly twelve color 5-cycles; they are the targets $\eta$. They are obtained from the 5!=120 permutations on five objects as the twelve orbits of the dihedral group $D_{10}$ generated both by translations$\pmod{5}$ and by reflections of the 5-tuples on $\{1,2,3,4,5\}$.
The edges of $\Gamma$ not in $\cup_{j=0}^{13}Pet^j$ occur between different copies $\Pet^j$ of $Pet$; these are colored as shown in the bottom part of Table~\ref{8}.
This insures the statement  for $k=1$ since the twelve vertical copies $\Pet^j$ of $\Pet$ are the only source of the color cycles. The extension of this for any $0<k\in\mathbb{Z}$ is immediate.
\end{proof}

The dodecahedral graph $\Dod$ with vertex set $\{u_i,w_i|i=1,2,\ldots,10\}$ and edge set formed by an edge pair $\{(u_i,w_{i'}),(u_{i'},w_i)\}$ for each $(v_i,v_{i'})\in(E(\Pet)\setminus\{(v_5,v_6),(v_7,v_8),(v_9,v_{10})\})$ and an edge pair  $\{(u_i,u_{i+1}),(w_i,w_{i+1})\}$ for each $i\in\{5,7,9\}$ is represented in the lower-right of Fig.~\ref{f12}, where $u_i$ and $w_i$ (that we will refer to as antipodal vertices) are respectively indicated by $i$ and $i'$ for $i=1,\ldots,10$. $\Dod$  is a 2-covering graph of $\Pet$ via the graph map $\phi:\Dod\rightarrow \Pet$ such that $\phi^{-1}(\{v_i\})=\{u_i,w_i\}$.

\begin{table}
$$\begin{array}{|lll|llll|}\hline
(1,2,5,10,4)&\rightarrow&(1,2,3,4,5)&&(3,6,7,8,9)&\rightarrow&(4,1,3,5,2)\\
(4,1,2,8,7)&\rightarrow&(5,1,4,3,2)&&(10,9,3,6,5)&\rightarrow&(1,2,4,5,3)\\
(1,2,5,6,3)&\rightarrow&(1,2,5,4,3)&&(4,7,8,9,10)&\rightarrow&(2,3,5,1,4)\\
(2,5,10,9,8)&\rightarrow&(2,3,1,5,4)&&(1,3,6,7,4)&\rightarrow&(3,4,1,2,5)\\
(5,10,4,7,6)&\rightarrow&(3,4,2,1,5)&&(2,8,9,3,1)&\rightarrow&(4,5,2,3,1)\\
(10,4,1,3,9)&\rightarrow&(4,5,3,2,1)&&(5,6,7,8,2)&\rightarrow&(5,1,3,4,2)\\\hline
\end{array}$$
\label{VIII}
$$\begin{array}{|ccc|}\hline
(v_1^j, v_1^{j+1})\mbox{ has color }2,\mbox{ if }j\equiv 0\pmod{2}&\mbox{and} &4,\mbox{ if }j\equiv 1\pmod{2}\\
(v_5^j, v_5^{j+1})\mbox{ has color }1,\mbox{ if }j\equiv 0\pmod{2}&\mbox{and} &3,\mbox{ if }j\equiv 1\pmod{2}\\
(v_7^j, v_7^{j+1})\mbox{ has color }5,\mbox{ if }j\equiv 0\pmod{2}&\mbox{and} &4,\mbox{ if }j\equiv 1\pmod{2}\\
(v_9^j, v_9^{j+1})\mbox{ has color }3,\mbox{ if }j\equiv 0\pmod{2}&\mbox{and} &4,\mbox{ if }j\equiv 1\pmod{2}\\
(v_2^j, v_2^{j+3})\mbox{ has color }5,\mbox{ if }j\equiv 0\pmod{2}&\mbox{and} &3,\mbox{ if }j\equiv 1\pmod{2}\\
(v_3^j, v_3^{j+3})\mbox{ has color }1,\mbox{ if }j\equiv 0\pmod{2}&\mbox{and} &5,\mbox{ if }j\equiv 1\pmod{2}\\
(v_4^j, v_4^{j+3})\mbox{ has color }1,\mbox{ if }j\equiv 0\pmod{2}&\mbox{and} &3,\mbox{ if }j\equiv 1\pmod{2}\\
(v_6^j, v_6^{j+5})\mbox{ has color }3,\mbox{ if }j\equiv 0\pmod{2}&\mbox{and} &2,\mbox{ if }j\equiv 1\pmod{2}\\
(v_8^j, v_8^{j+5})\mbox{ has color }1,\mbox{ if }j\equiv 0\pmod{2}&\mbox{and} &2,\mbox{ if }j\equiv 1\pmod{2}\\
(v_{10}^j, v_{10}^{j+5})\mbox{ has color }5,\mbox{ if }j\equiv 0\pmod{2}&\mbox{and} &2,\mbox{ if }j\equiv 1\pmod{2}\\\hline
\end{array}$$
\caption{Color assignment of the $4^30^2$-graph $\Gamma$ in Theorem~\ref{4^30^2}}.
\label{8}
\end{table}

\begin{theorem}\label{2^30^2}
There exists an egc $2^30^2$-graph $\Gamma$ of order $140\times k$ for every $0<k\in\mathbb{Z}$.
\end{theorem}

\begin{proof} We consider the case $k=1$ and leave the details of the general case to the reader.
 We take seven vertical copies of $\Dod$ presented in parallel at equal distances from left to right and numbered from 0 to 6 in $\mathbb{Z}_7$. The vertices of the $j$-th copy
 $\Dod^j$ of $\Dod$ are denoted $u_1^j$ and $w_i^j$ for $i=1,\ldots,10$ and are joined by the additional cycles

\begin{align}\label{dod}\begin{array}{ll}
(u_i^0,u_i^1,u_i^2,u_i^3,u_i^4,u_i^5,u_i^6,w_i^0,w_i^1,w_i^2,w_i^3,w_i^4,w_i^5,w_i^6),&\mbox{ for }i=1,5,7,9;\\
(u_i^0,u_i^3,u_i^6,u_i^2,u_i^5,u_i^1,u_i^4,w_i^0,w_i^3,w_i^6,w_i^2,w_i^5,w_i^1,w_i^6),&\mbox{ for }i=6,8,10;\\
(u_i^0,u_i^5,u_i^3,u_i^1,u_i^6,u_i^4,u_i^2,w_i^0,w_i^5,w_i^3,w_i^1,w_i^6,w_i^4,w_i^2),&\mbox{ for }i=2,3,4.\\
\end{array}\end{align}
so that each such additional cycle passes through two antipodal vertices of each copy $\Dod^j$. Notice the change of the order of the indices $i\in\{1,\ldots,10\}$ in the assignment of the additional cycles in display~(\ref{dod}) with respect to the one in display~(\ref{pet}). This is done to avoid the formation of 5-cycles not entirely contained in the copies $\Dod^j$ ($j\in\mathbb{Z}_7$).
Since $\Dod$ has girth 5 and signature $2^3=222$ the graph $\Gamma$ given by the union of the seven copies $\Dod^j$ and the just presented additional cycles is a $2^30^2$-graph. By coloring the edges of the additional cycles of $\Gamma$ via the same color pattern as in Theorem~\ref{4^30^2}, it is seen that $\Gamma$ is egc.
\end{proof}

The truncated-icosahedral graph $TI$ is the graph of the truncated icosahedron. This is obtained from the icosahedral graph, i.e. the line graph $\Ico=L(\Dod)$ of $\Dod$ by replacing each vertex $v$ of $\Ico$ by a copy $C_5^{TI}(v)$ of its open neighborhood $N_{Ico}(v)$ considering all such copies $C_5^{TI}(v)$ pairwise disjoint, and replacing each edge $(u,v)$ of $\Ico$ by an edge from the vertex corresponding to $u$ in $C_5^{TI}(v)$ to the vertex corresponding to $v$ in $C_5^{TI}(u)$.
Note $TI$ has sixty vertices, ninety edges, twelve 5-cycles, twenty 6-cycles and signature $1^20=110$.

\begin{theorem}\label{10^4} There exists an egc $10^4$-graph on $840.k$ vertices, for each integer $k>0$.\end{theorem}

\begin{proof}
By means of a barrel-type construction as in Fig.~\ref{f12}, one can combine $14k$ copies of $TI$ and the Cayley graph of $\mathbb{Z}_{14}$ with generator set $\{1,3,5\}$ to get an egc graph as claimed.
\end{proof}

\begin{remark}\label{W} The point graph of the generalized hexagon
$GH(1,5)$ \cite[p.~204]{BNC}, the point graph of the Van Lint--Schrijver partial geometry \cite[p.~307]{BNC} and the odd graph \cite[p.~259]{BNC} on eleven
points are distance-regular with intersection arrays $\{6,5,5;1,1,6\}$ $\{6,5,5,4;1,1,2,6\}$
and $\{6,5,5,4,4;1,1,2,2,3\}$ respectively. If their chromatic number were 6,  they would be  $(125)^6$-, $(25)^6$- and $(25)^6$-graphs, respectively, but it is known that none of them is egc.
\end{remark}

\section{Hamilton cycles and hamiltonian decomposability}\label{hd}

Some feasible applications of egc graphs occur when the unions of pairs of composing 1-factors are Hamilton cycles, possibly attaining hamiltonian decomposability in the even-degree case. This offers a potential benefit to the applications drawn in Section~\ref{intro}, if an optimization/decision-making problem requires alternate inspections covering all nodes of the involved system,  when the alternacy of two colors is required.

In Fig.~\ref{fig1}, the cases (h--j) and their triangle-replaced graphs (m-o) as well as the case (u), and the 3-colored dodecahedral graph $\Dod$ that has the case (u) as its triangle replaced graph, have the unions of any two of their 1-factors forming a Hamilton cycle, while the cases (k), (p) and (v) have those unions as disjoint pairs of two cycles of equal length. In particular, the 3-cube graph $Q_3$ that admits just two tight factorizations, has one of them creating Hamilton cycle (case (l) via green and either red or blue edges, but not red and blue edges). The triangle-replaced graph, $\nabla(Q_3)$ has corresponding tight factorizations in cases (p--q) with similar differing properties as those of cases (k--l). Preceding Theorem~\ref{stat}, similar comments are made for $\nabla(\Gamma')$ where $\Gamma'$ is $\Dod$ or the Coxeter graph $\Cox$. Recall the union of two 1-factors of $\Dod$ is hamiltonian while the union of two 1-factors of $\Cox$ is not.

\begin{table}[htp]
$$\begin{array}{||c||c|c|c|c|c|c||c||c|c|c||c||c|c|c||}\hline
s=1&5&7&9&11&13&15&s=3&5&7&9&s=5&5&7&9\\\hline
 \!r=6&113&&&&&&\!r=6&131&&&  &&&\\
 \!r=8&211&114&&&&&\!r=8&114&211&&   &&&\\
 \!r=10&111&111&511&&&&\!r=10&111&511&111&               \!   r=   10&555&151&151\\
\!r=12&213&312&112&116&&&\!r=12&132&233&336&         \!r=    12&611&611&211\\
\!r=14&111&111&111&111&117&&\!r=14&141&111&111&        \!r=  14&711&111&117\\
\!r=16&211&114&114&112&211&118&\!r=16&411&112&112&  \!r= 16&811&211&211\\
\!r=18&113&311&111&311&311&111&\!r=18&131&131&333&  \!r= 18&911&111&111\\
\!r=20&211&112&215&215&112&112&\!r=20&211&512&112& \!r=  20&a51&251&251\\
\!r=22&111&111&111&111&111&111& \!r=22&111&111&111&\!r=22&b11&111&111\\
\!r=24&213&314&411&611&611&114&\!r=24&133&231&336&\!r=24&c11&211&211\\
\!r=26&111&111&111&111&111&111&\!r=26&111&111&111&\!r=26&d11&111&111\\
\!r=28&211&112&211&211&217&711&\!r=28&112&211&112&\!r=28&e11&211&217\\
\!r=30&113&311&115&115&111&111&\!r=30&132&135&333&\!r=30&f11&151&111\\\hline
\end{array}$$
\caption{Various cases of Theorem~\ref{re} item 3($e$).}
\label{IX}
\end{table}

In Fig.~\ref{f2}(d), the three color partitions of $Q_4$ namely (12)(34), (13)(24) and (14)(23), yield 2-factorizations with 2-factors formed each by two cycles of equal length $\frac{1}{2}|V(Q_4)|=8$. We denote this fact by writing $Q_4(2,2,2)$. In a likewise fashion, we can denote toroidal items in  Fig.~\ref{f2} as follows:
(e) $\{4,4\}_{12,2}^4(1,4,3)$ formed by 2-factorizations with 2-factors of one, three and four cycles of equal lengths 24, 6 and 8, respectively. Similarly:
(f) $\{4,4\}_{10,2}^4(1,2,1)$;
(g) $\{4,4\}_{6,3}^3(3,3,3)$;
(h) $\{4,4\}_{20,1}^0(2,1,1)$;
(i) $\{4,4\}_{28,1}^0(1,1,2)$; and
(j) $\{4,4\}_{22,1}^5(1,2,1)$.

Table~\ref{IX} lists various cases of Theorem~\ref{re} item 3($e$), indicating without parentheses or commas the triples $abc$ corresponding to the numbers $a$ $b$ and $c$ of cycles (of equal length in each case) of the respective 2-factors
$(12)$ $(13)$ and $(14)$.

\begin{remark}\label{differ} For the toroidal cases in Theorem~\ref{re} item 3 depicted as in Fig.~\ref{f2}(f,g,h,j), assume that the 2-factors $(12)$ and $(14)$ complete 2-factorizations composed by {\it 1-zigzagging} cycles of equal length (i.e., composed by alternating horizontal and vertical edges) and that the 2-factors $(13)$ and $(24)$ are composed by vertical and horizontal edges, respectively. This way, while vertical edges form $\gcd(r,s)$ cycles of equal length, horizontal edges form $t$ cycles of not necessarily the same length, so the notation in the previous paragraph cannot be carried out for example for item 3($e$) because $\gcd(r,s)\ne t$. So we modify that notation for such cases by simply writing $\{4,4\}_{r,t}^s(a,b,c)$ that we call the {\it star notation} \cite{DS}.\end{remark}

\begin{theorem}\label{wf} In the star notation of Remark~\ref{differ}, each applicable toroidal case $\Gamma=\{4,4\}_{r,t}^s$ as in Theorem~\ref{re} is expressible as:
$\{4,4\}_{r,t}^s(\frac{1}{2}\gcd(r,|t-s|),\gcd(r,s),\frac{1}{2}\gcd(r,t+s)).$\end{theorem}

\begin{proof}
We prove the statement for the toroidal cases of Theorem~\ref{re} with two colors on horizontal cycles and the other two on vertical cycles, for the factorization $\{F_{12},F_{34}\}$ and leave the rest to the reader. Consider the straight upper-right-to-lower-left line $L_1$ from the upper-right vertex $(0,0)$ in the cutout $\Phi$ (Remark~\ref{rectangle}) of $\mathbb{Z}_r\times\mathbb{Z}_t$ passing through $(0,r-t+s)$ in the lower border of $\Phi$
 and formed by the diagonals of $\frac{rt}{2\gcd(r,|t-s|)}$ squares representing 4-cycles of $\Gamma$. $L_1$ determines
 two 1-zigzagging cycles $C_1^0,C_1^1$ through $(0,0)$ in $F_{12},F_{34}$ respectively, touching $L_1$ on alternate vertices of $\Gamma$.
In the end, we get parallel lines $L_1,\ldots,L_z$ where $z=\frac{1}{2}\gcd(r,|t-s|)$ such that each $L_i$ ($i=1,\ldots,z$) determines two 1-zigzagging cycles $C_i^0,C_i^1$ in $F_{12},F_{34}$ respectively, touching $L_i$ at alternate vertices of $\Gamma$. An example is shown in Fig.~\ref{f2}(c) for $r=12,t=5,s=9$ with $z=2$ where $L_1$ is given in black thin trace and $L_2$ is given in gray thin trace, (not considering here the intermittent diagonals).
\end{proof}

\begin{corollary}
If $\frac{1}{2}\gcd(r,|t-s|)=\gcd(r,s)=\frac{1}{2}\gcd(r,t+s)=1$ then the 2-factors $(12)$ $(34)$ $(13)$ $(14)$ and $(23)$ are composed by a Hamilton cycle each (a total of six Hamilton cycles), comprising the 2-factorizations $(12)(34)$ and $(14)(23)$.
\end{corollary}

\begin{proof}
This is due to Remark~\ref{differ} and to the quadruple equality in the statement.
\end{proof}

Additional examples are provided in display~(\ref{X}) for fixed $t=2$ as in Theorem~\ref{re} item 3($b$). The reader is invited to do similarly for Theorem~\ref{re}, items 3($a$) and 3($c$).

\begin{align}\label{X}
\begin{array}{||c||c|c||c||c|c||c||c|c||c||}
\hline r&s=4&s=6&r&s=4&s=6&r&s=4&s=6\\\hline
8&141&&14&121&121&20&141&122\\
10&121&&16&141&124&22&121&121\\
12&143&162&18&123&131&24&143&164\\\hline
\end{array}
\end{align}

\begin{corollary}\label{hamil}
In all cases of Theorem~\ref{re} item 2(b), there are exactly two hamiltonian 2-factorizations. Moreover,
the toroidal graphs $\Gamma$ in Theorem~\ref{re} are

\begin{enumerate}
\item  $\Gamma(222)$ for items 1(c)--2(a);
\item  $\Gamma(211)$ for item 2(b) just for $s\equiv 1\pmod{4}$;
\item $\Gamma(112)$ for item 2(b) just for $s\equiv 3\pmod{4}$.
\end{enumerate}
\end{corollary}

\begin{proof}
The toroidal graphs in Theorem~\ref{re} items 1($c$) and 2($b$) behave differently from those in Theorem~\ref{wf} in that the 2-factors in question are 1-zigzagging in only one of the three 2-factorizations, while the other two 1-factorizations are 2-zigzagging, namely:

\begin{enumerate}\item[{\bf(a)}]
in Theorem~\ref{re} items 1($c$) and 2($b$), just for $s\equiv 1\pmod{4}$ the 1-factorization (12)(34) is 1-zigzagging and the 1-factorizations (13)(24) and (14)(23) are 2-zigzagging;
\item[{\bf(b)}] in Theorem~\ref{re} item 2($b$) just for $s\equiv 3\pmod{4}$ the 2-factorizations (12)(34)--(13)(24) are 2-zigzagging and the 2-factorization (14)(23) is 1-zigzagging.\end{enumerate}
\end{proof}

\section{Applications to 3-dimensional geometry}\label{3d}

\begin{figure}[htp]
\includegraphics[scale=0.376]{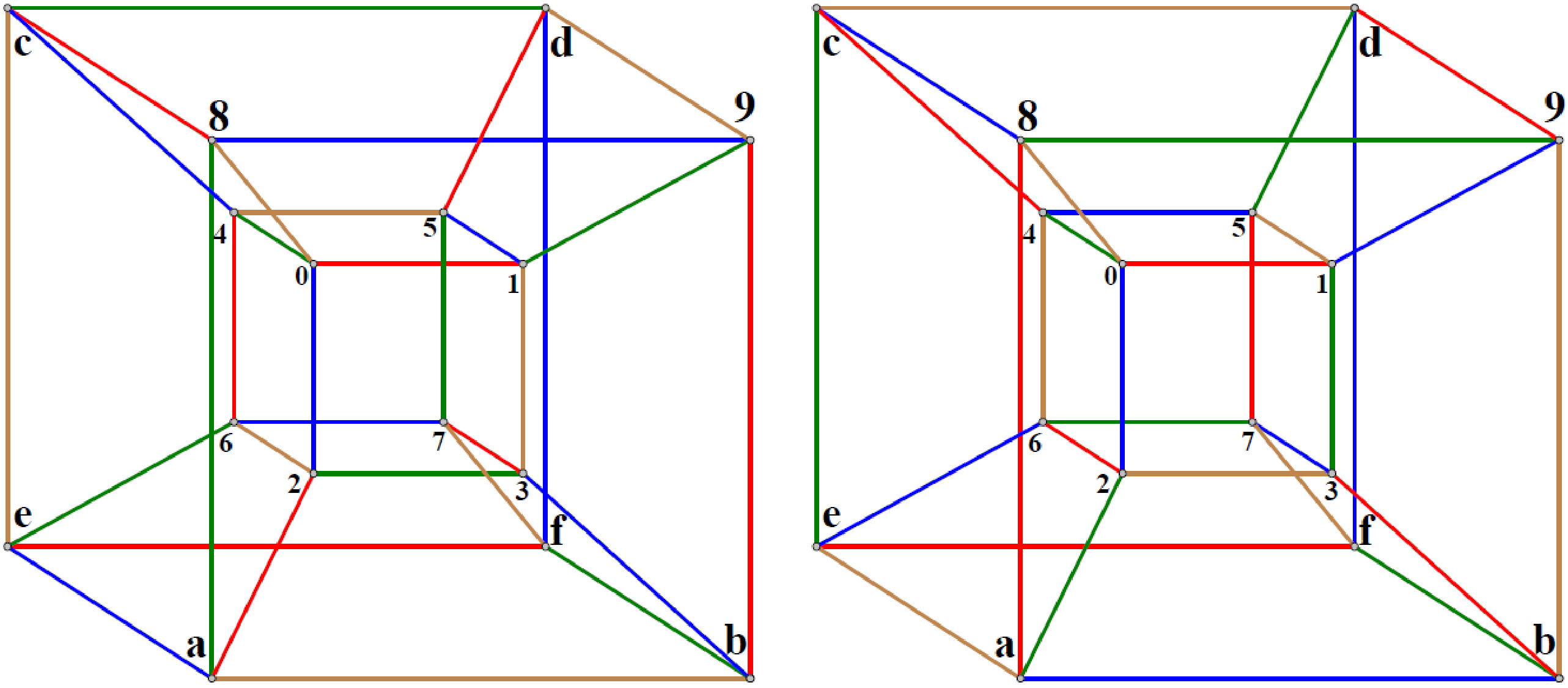}
\caption{Two egc edge-colored 4-cube graphs.}
\label{14}
\end{figure}

Fig.~\ref{14} redraws the two toroidal copies of $Q_4$ in the lower center and right of Fig.~\ref{f3} (arising in Subsection~\ref{Q4} from the central and right latin squares, respectively, in display~(\ref{(1)})) as edge-colored 4-cubes with common vertex set $\{0,\ldots,7\}\cup\{i+8=i'|i=0,\ldots,7\}$ expressed in lowercase hexadecimal notation, in order to extract in Subsections~\ref{sub}, \ref{use} and~\ref{sigue} piecewise linear (PL) (as in \cite{Rourke}) realizations of two enantiomorphic (i.e., mirror images of each other) compounds of four M\"obius strips (as in \cite{Fuchs,HV}). In the sequel, such pair of compounds is shown in Subsection~\ref{poly} to be equivalent to corresponding enantiomorphic Holden-Odom-Coxeter polylinks of four locked hollow equilateral triangles each \cite{Coxeter,Holden,Holden2}, from a group-theoretical point of view, having determined the automorphism group of the compounds in Subsection~\ref{au}.

\begin{figure}[htp]
\includegraphics[scale=0.85]{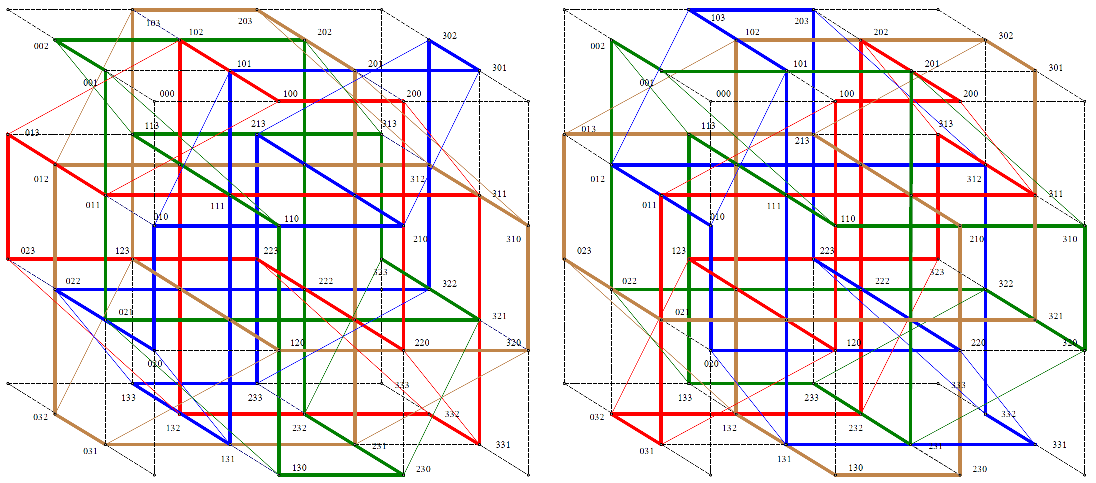}
\caption{Two enantiomorphic quadruples of edge-disjoint PL knots $3_1$ in $[0.3]^3$}
\label{18}
\end{figure}

\begin{figure}[htp]
\includegraphics[scale=0.85]{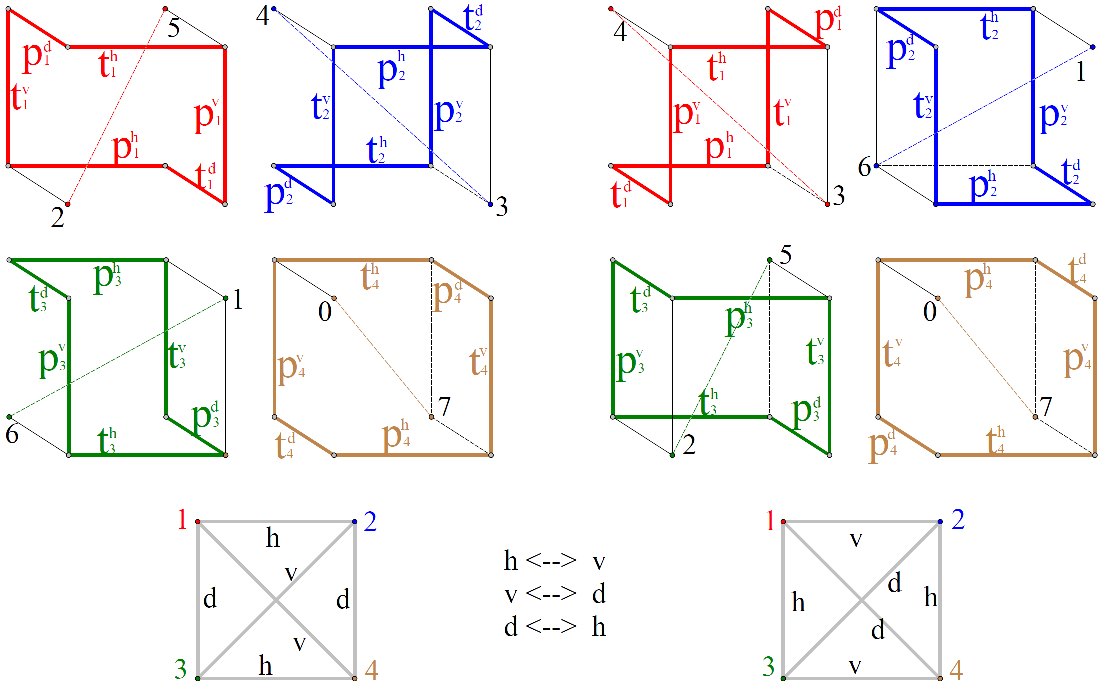}
\caption{Edges of rotation.}
\label{15}
\end{figure}

\begin{figure}[htp]
\includegraphics[scale=1.04]{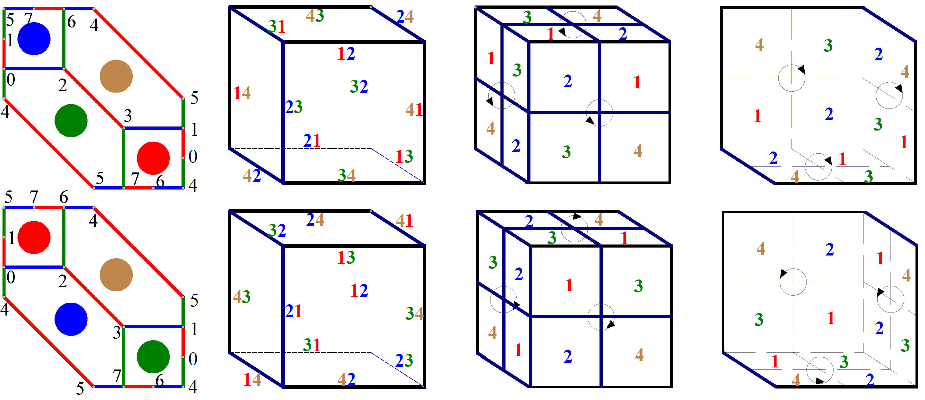}
\caption{Assignment of pairs of colors to the edges of 3-cubes}
\label{16}
\end{figure}

\subsection{Compound of four PL M\"obius strips}\label{sub}

\begin{theorem}\label{thm1} There exists a pair of enantiomorphic piecewise linear M\"obius strip embedded in the hollow cube $[0,3]^3\setminus[1,2]^3\subset\mathbb{R}^3$ with piecewise linear closed curves as their boundary formed by segments parallel to the coordinate directions whose endvertices are points of $\mathbb{Z}^3$.
\end{theorem}

\begin{figure}[htp]
\hspace*{15mm}
\includegraphics[scale=1.5]{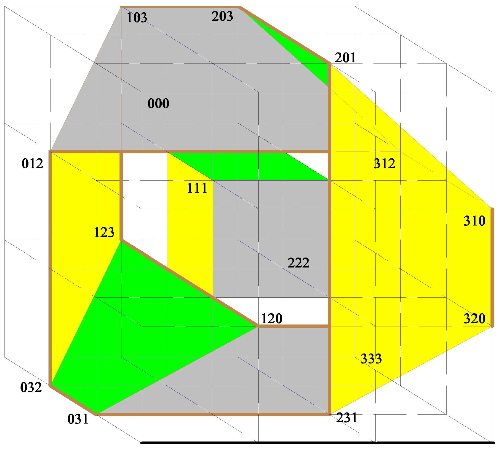}
\caption{PL M\"obius strip whose PL boundary has color 4 = hazel.}
\label{17}
\end{figure} 

\begin{proof}
Fig.~\ref{18} represents two copies of the union $[0,3]^3\subset\mathbb{R}^3$ of 27 unit 3-cubes. In each such copy, Fig~\ref{18} highlights four disjoint {\it PL trefoil knots}, i.e. piecewise linear knots $3_1$  \cite[pp. 51--60]{Rolfsen} in thick trace (in contrast to the remaining edges of the 3-cubes, shown in dashed trace). Each such knot has its composing unit-length edges bearing a common color $i$, where $i=1$ for red, $i=2$ for blue, $i=3$ for green and $i=4$ for hazel. Accordingly, we denote $C_i$ and $C'_i$ for the PL trefoil knots in dark-traced color $i$ in the left and right, respectively, of Fig.~\ref{18}, for $i=1,2,3,4$. Note that $\cup_{i=1}^4 C'_i$ is an enantiomorph of $\cup_{i=1}^4 C_i$ and vice versa. 

Each of the knots $C_i$ or $C'_i$ ($i=1,2,3,4$) is the boundary of a corresponding PL M\"obius band $M_i$ or $M'_i$, respectively, contained in the hollow cube $[0,3]^3\setminus[1,2]^3$ formed by the 26 unit 3-cubes of $[0,3]^3$ other than $[1,2]^3$. In fact, each $C_i$ or $C'_i$ occurs in just 12 of the 26 unit cubes. 
Note that $\cup_{i=1}^4 M'_i$ is an enantiomorph of $\cup_{i=1}^4 M_i$ and vice versa.

Illustrated for $M_i$ where $i=4$, Fig.~\ref{17} exemplifies the fact that each $M_i$ or $M'_i$ is the union of six planar quadrilaterals, namely three (tilted) parallelograms and three (isosceles) trapezoids, in alternating contiguity, for $i=1,2,3,4$. In Fig.~\ref{17}, colors gray, green and yellow highlight the visible parts of the planar faces of $M_4$ and of the central cube $[1,2]^3$. 

Still in Fig.~\ref{18}, the sides of the six planar quadrilaterals of each $M_i$ or $M'_i$ ($i=1,2,3,4$) not in its boundary $C_i$ of $C'_i$ are in color-$i$ thin trace so as to help visualize the parallelograms and trapezoids.

More specifically, the $i$-colored edges of the unit 3-cubes of $[0,3]^3$ (illustrated in Fig.~\ref{17} for $i=4$) form two parallel sides in each of the six quadrilaterals (with the occult parts in Fig.~\ref{17} in dashed trace, though visible on the left of Fig.~\ref{18}). In the case of a trapezoid (resp., parallelogram) the lengths of those sides are 3 units internally and 1 unit externally (resp., 2 units internally and 2 units externally).  

The sides of the hazel PL trefoil knot $C_4$ in Fig.~\ref{17} have their endvertex coordinates detailed in display (\ref{(3)}), with colors of quadrilaterals cited between brackets starting clockwise at the left upper corner in Fig.~\ref{17}, whether they are shown in full or in part, as indicated).
In display (\ref{(3)}), a segment denoted $[a_1a_2a_3,b_1b_2b_3]$ stands for $[(a_1,a_2,a_3),(b_1,b_2,b_3)]$, with $a_i,b_i\in\{0,1,2,3\}$, ($i=1,2,3$); each of the twelve segment brackets is appended with its length as a subindex and an element of $\{h,v,d\}$ as a superindex, where $h,v,$ and $d$ stand for horizontal, vertical and in-depth directions, respectively. The PL boundary curve $C_4$ is the hazel-colored PL trefoil knot in Fig.~\ref{17}. 
\begin{align}\label{(3)}\begin{array}{|cc|c|c|c|c|c|}
&gray(part)&green(part)&yellow(full)&gray(part)&green(full)&yellow(part)\\\hline
&[103,203]_1^h&[203,201]_2^v&[201,231]_3^d&[231,031]_2^h&[031,032]_1^v&[032,012]_2^d\\
&[012,312]_3^h&[312,310]_2^v&[310,320]_1^d&[320,120]_2^h&[120,123]_3^v&[123,103]_2^d\\\hline
\end{array}\end{align}
A similar display can be obtained for the other PL-trefoil knots, shown for $C_2,C_3,C_4$ respectively in displays (\ref{(4)}), (\ref{(5)}), (\ref{(6)}), below. 
\end{proof}

\begin{theorem}\label{thm2} There are four M\"obius strips $M_i$ in $[0,3]^3\setminus[1,2]^3$ ($i=1,2,3,4\}$)
with PL closed curves $C_i$ as their corresponding boundaries; these are formed by segments parallel to the coordinate directions with endpoints in $\mathbb{Z}^3$; in fact, they are PL knots $3_1$ with intersections in $\mathbb{Z}^3$. 
Moreover, the compound $M=\cup_{i=1}^4M_i$ has an enantiomorphic compound $M'=\cup_{i=1}^4M'_i$ in $[0,3]^3\setminus[1,2]^3$, formed by other four M\"obius strips $M'_i$ with similar properties to those of the $M_i$.
\end{theorem}

\begin{proof} To prove the statement, consider again the left of Fig.~\ref{18}, where the hazel $C_4$ is drawn (and summarized in display (\ref{(3)})), as well as the red PL trefoil knot $C_1$ (bounding M\"obius strip $M_1$), formed by the segments 
\begin{align}\label{(4)}\begin{array}{ccccccc}
&[100,200]_1^h&[200,220]_v^2&[220,223]_3^d&[223,023]_2^h&[023,013]_1^v&[013,011]_2^d\\
&[011,311]_3^h&[311,331]_v^2&[331,332]_1^d&[332,331]_2^h&[331,102]_3^v&[102,100]_2^d\\
\end{array}\end{align}
and the blue PL trefoil knot $C_2$ (bounding M\"obius strip $M_2$), formed by the segments
\begin{align}\label{(5)}\begin{array}{ccccccc}
&[133,233]_1^h&[233,213]_2^v&[213,210]_3^d&[210,010]_2^h&[010,020]_1^v&[020,022]_2^d\\
&[022,322]_3^h&[322,302]_2^v&[302,301]_1^d&[301,101]_2^h&[101,131]_3^v&[131,133]_2^d\\
\end{array}\end{align}
and the green PL trefoil knot $C_3$ (bounding M\"obius strip $M_3$), formed by the segments
\begin{align}\label{(6)}\begin{array}{ccccccc}
&[230,130]_1^h&[130,110]_2^v&[110,113]_3^d&[113,313]_2^h&[313,323]_1^v&[323,321]_2^d\\
&[321,021]_3^h&[021,001]_2^v&[001,00 2]_1^d&[002,202]_2^h&[202,232]_3^v&[232,230]_2^d\\
\end{array}\end{align}
The six planar faces of $M_4$ intersect $C_4$ in six corresponding pairs of segments, shown in display (\ref{(3)}) vertically: first $[103,203]_1^h$ above and $[012,312]_3^h$ below, which delimit an isosceles trapezoid with the segments $[103,012]$ and $[203,312]$; then, $[203,201]_2^v$ 
above, and $[312,310]_2^v$ below, which delimit a parallelogram with the segments $[203,312]$ and $[201,310]$, etc.
Similar observations can be made with respect to displays~(\ref{(4)}), (\ref{(5)}) and~(\ref{(6)}).
Similar conclusions can be obtained for $C'_i$ ($i=1,2,3,4$).
\end{proof}

\subsection{Use of double egc girth-4-regular 4-cube}\label{use}

We denote the planar faces in displays~(\ref{(3)}), (\ref{(4)}), (\ref{(5)}) and~(\ref{(6)}) by means of the symbols $t_i^h,p_i^d,t_i^v,p_i^h,t_i^d,p_i^v$, respectively, ($i\in\{1,2,3,4\}=\{$red, blue, green, hazel$\}$), as in the four left  3-cubes in the schematic Fig.~\ref{15}.

In fact, Fig.~\ref{15} contains two sets of four copies of $Q_3$ each, the two sets differentiated by a vertical separator line, the leftmost set having its four copies of $Q_4$ already mentioned in the precious paragraph. 
These two sets of four copies of $Q_4$ correspond bijectively to the two sets of four trefoil knots in the copies of $[0,3]^3$ in Fig.~\ref{18} so that each copy of $Q_3$ bears the monochromatic color of a corresponding trefoil knot. Then,
each such copy of  $Q_3$ is denoted $Q_3^i$ where $i\in\{1,2,3,4\}=\{$red, blue, green, hazel$\}$ is its monochromatic color in Fig.~\ref{15}, that corresponds bijectively to the monochromatic color $i$ of its associated PL trefoil knot in Fig.~\ref{18}. (Recall that Fig.~\ref{18} contains two copies of each $Q_3^i$ with color $i=1$ for red, $i=2$ for blue, $i=3$ for green and $i=4$ for hazel).  

Two 6-cycles $\xi_i$ and $\zeta_i$ are highlighted in thick trace on the left and right representations of $Q_3^i$ in Fig.~\ref{15}, respectively. Each copy of $Q_3^i$ in Fig.~\ref{15} has its 6-cycle $\xi_i$ or $\zeta_i$ with its highlighted edges marked $x_i^j$, where 
\begin{enumerate}
\item $x=t$ for a trapezoid and $x=p$ for a parallelogram (both cases being quadrilaterals that are the faces of PL M\"obius strips as in Fig.~\ref{18}), 
\item $i\in\{1,2,3,4\}$ for associated color and
\item $j\in\{h,v,d\}$ for edge direction ($h$ for horizontal, $v$ for vertical and $d$ for  in-depth). 
\end{enumerate}
Each difference set $Q_3^i\setminus\xi_i$ or $Q_3^i\setminus\zeta_i$ appears in Fig.~\ref{15} as a pair of antipodal vertices in its copy of $Q_3^i$ joined by a color-$i$ dashed axis. 

Below the left and right sets of four copies of $Q_3$ in Fig.~\ref{15} there are two corresponding auxiliary copies of $K_4$ whose vertices 1,2,3 and 4 correspond to, and are colored with, the red, blue, green and hazel of such copies of $Q_3$, and where the edges bear labels $h$ for horizontal, $v$ for vertical and $d$ for in-depth to indicate the directions of the only two shared edges of each pair of $\xi_i$ or $\zeta_i$. Between the two copies of $K_4$ the correspondence of colors between the $\xi_i$ and the $\zeta_i$ is indicated. 

We will establish now two bijections from the sets of four copies of $Q_3$ in Fig.~\ref{15} (representing $M,M'$ in Theorem~\ref{thm1}) onto the 1-factorizations $F_1,F_2$ of $Q_4$ in Subsection~\ref{Q4} illustrated in Fig.~\ref{f3} and Fig.~\ref{14}, with the subgraphs of $Q_4$ induced by the vertex subsets $\{0,1,\ldots,7\}$ and $\{8,,9.a,b,c,d,e,f\}$ referred as {\it inner} and {\it outer}, respectively, copies of $Q_3$ in $Q_4$.   

\begin{theorem}\label{o1} 
There exists a bijection $\eta$ from the edges denoted $t_i^j$  (resp., $p_i^j$) in the 6-cycles $\xi_i$ to the edges with color $i$ and direction $j$ in the inner (resp., outer) copy of $Q_3$ in the copy of $Q_4$ colored according to the 1-factorization $F_1$. 
\end{theorem}

\begin{proof} 
Display (\ref{izq}) demonstrates in detail the claimed bijection $\eta$
where the rows are headed by the symbols $0_t,0_p,3_t,3_p,5_t,5_p,6_t,6_p$ (corresponding to the alternate vertices
0,3,5,6 of $Q_3$ with subindices $x=t$ and $x=p$ for trapezoids and parallelograms, respectively, in Fig.~\ref{17}), followed by three symbols $x_i^j$ associated to the corresponding edges of the 6-cycles of $Q_3$. The rest of each row is headed by the vertices $0,8,3,b,5,d,6,e$, alternatively in the inner and outer 3-cube of the left 4-cube of Fig.~\ref{14}, each such vertex followed by three edges associated bijectively with the three symbols $x_i^j$ to their left, where $(0,8)$, $(3,b)$, $(5,d)$, $(6,e)$ are vertex pairs corresponding to vertices 0,3,5,6 of $Q_3$, respectively, and each edge in the inner and outer 3-cubes of $Q_4$ has its color in $\{1,2,3,4\}$ as a subindex.
\begin{eqnarray}\label{izq}\begin{array}{||c|cccc||c|cccc||}\hline\hline
0_t &t_1^h &t_2^v &t_3^d   &        &0&(0,1)_1&(0,2)_2&(0,4)_3&\\
0_p&p_1^d&p_2^h&p_3^v  &        &8&(8,c)_1&(8,9)_2&(8,a)_3&\\\hline
3_t&t_1^d &t_3^h  &          &t_4^v &3&(3,7)_1&(3,2)_3&           &(3,1)_4\\
3_p&p_1^v&p_3^d&          &p_4^h&b&(b,9)_1&(b,f)_3&            &(b,a)_4\\\hline
5_t&           &t_2^d&t_3^v  &t_4^h&5 &           &(5,1)_2&(5,7)_3&(5,4)_4\\
5_p&          &p_2^v&p_3^h&p_4^d&d&           &(d,f)_2  &(d,e)_3&(d,9)_4 \\\hline 
6_t&t_1^v&t_2^h&&t_4^d&6&(6,4)_1&(6,7)_2&&(6,2)_4\\
6_p&p_1^h&p_2^d&&p_4^v&e&(e,f)_1&(e,a)_2&&(e,c)_4\\\hline\hline
\end{array}\end{eqnarray} 
 \end{proof}
 
 \begin{theorem}\label{o2}
 There exists a bijection $\eta'$ from the edges denoted $t_i^j$  (resp., $p_i^j$) in the 6-cycles $\zeta_i$ to the edges with color $i$ and direction $j$ in the inner (resp., outer) copy of $Q_3$ in the copy of $Q_4$ colored according to the 1-factorization $F_2$.  
\end{theorem}

\begin{proof} Similar to the proof of Theorem~\ref{o1} but replacing $\eta$ by $\eta'$ and display (\ref{izq}) by display (\ref{der}).
\begin{eqnarray}\label{der}\begin{array}{||c|cccc||c|cccc||}\hline\hline
0_t &t_1^h &t_2^v &t_3^d   &        &0&(0,1)_1&(0,2)_2&(0,4)_3&\\
0_p&p_1^v&p_2^d&p_3^h  &        &8&(8,a)_1&(8,c)_2&(8,9)_3&\\\hline
3_t&          &t_2^d &t_3^v  &t_4^h &3&           &(3,7)_2&(3,1)_3&(3,2)_4\\
3_p&         &p_2^h&p_3^d &p_4^v&b&           &(b,a)_2&(b,f)_3&(b,9)_4\\\hline
5_t&t_1^v&t_2^h  &&t_4^d&5  &(5,7)_1&(5,4)_2&&(5,1)_4\\
5_p&p_1^d&p_2^v&&p_4^h&d&(d,9)_1&(d,f)_2&&(d,c)_4 \\\hline 
6_t&t_1^d&&t_3^h&t_4^d&6&(6,2)_1&&(6,7)_3&(6,4)_4\\
6_p&p_1^h&&p_2^d&p_4^v&e&(e,f)_1&&(e,c)_3&(e,a)_4\\\hline\hline
\end{array}\end{eqnarray}
\end{proof}

\begin{remark}
While the bijection in display (\ref{izq}) corresponds to the compound of four M\"obius strips $\cup_{i=1}^4 M_i$ constructed in Subsection~\ref{sub}, the bijection in (\ref{der}) corresponds to its enantiomorphic compound $\cup_{i=1}^4 M'_i$. 
\end{remark}

\subsection{Petrie polygons}\label{sigue}

\begin{remark}
The four 6-cycles $\xi_i$ (resp., $\zeta_i$) on the left (resp., right) of Fig.~\ref{15} are the Petrie polygons (see \cite{self}) of the 3-cube $Q_3$ viewed as a $\{4,3\}$-map, namely the regular map $\{6,3\}_{2,0}$ in \cite[Fig. 5]{self} representing an embedding of the graph $Q_3$ into the torus, which we have slightly modified on the left of Fig.~\ref{16} representing $\cup_{i=1}^4\xi_i$ on top (resp., $\cup_{i=1}^4\zeta_i$ on the bottom), where solid colored circles indicate which Petrie polygons of Fig.~\ref{15} are represented, and vertices are denoted via a single digit $0,1,2,3,4,5,6,7$ standing for vertex $000,001,010,011,100,101,110,111$, respectively, in Fig.~\ref{18} or~\ref{17}. In these representations of $\cup_{i=1}^4\xi_i$ and  $\cup_{i=1}^4\zeta_i$ the edges are colored according to the coordinate directions along which they run: 1 for red, 2 for blue, 3 for green and 4 for hazel.
Observe that $\xi_i$ and $\zeta_i$ share exactly two antipodal edges along direction $i$, for $i=1,2,3$, while $\xi_4$ coincides with $\zeta_4$. 
\end{remark}

In the rest of Fig.~\ref{16}, to the right of the representations of $\cup_{i=1}^4\xi_i$ and  $\cup_{i=1}^4\zeta_i$ 
the two leftmost 3-cubes
have each of their edges assigned a pair of colors in $\{1,2,3,4\}$. The first (resp., second) color is obtained from
the corresponding row of 3-cubes in Fig.~\ref{15} as the color $i$ of the only edge $t_i^j$ (resp., $p_i^j$) in that edge position among the four cases in the row, where color numbers in Italics are employed for $i=4$ (hazel) and in Roman for $i\in\{1,2,3\}$, (to distinguish red, blue and green).
 This allows an assignment from the color set $\{1,2,3,4\}$ onto the set of 24 face quadrants of the 3-cube, shown in the right-center and right of Fig.~\ref{16} in each of the two cases, $\cup_{i=1}^4\xi_i$ and  $\cup_{i=1}^4\zeta_i$. The faces in these 3-cubes are given an orientation each. The color pairs in Fig.~\ref{15} can be recovered from the quadrant colors in Fig.~\ref{16} by reading them along an edge according to the corresponding face orientation.

Let us look at the segmental intersections of trapezoids and parallelograms in $M=M_1\cup M_2\cup M_3\cup M_4$ from displays~(\ref{(3)}), (\ref{(4)}), (\ref{(5)}) and~(\ref{(6)}).

The top-front set
$TF_i^h=([000,300]\times[000,010]\times[000,001])\cap M_i$ of $[0,3]^3$ in Fig.~\ref{18} is trapezoid $t_1^h$ for $i=1$ and parallelogram $p_2^h$ for $i=2$. In the upper leftmost cube in Fig.~\ref{16}, let us call it $Q$, we indicate the corresponding segmental intersection $t_1^h\cap t_2^h$ via the edge-labelling pair 12, with 1 in red and 2 in blue, namely the colors used to represent $C_1$ and $C_2$, respectively. This pair 12 labels the top-front horizontal edge of $Q$, corresponding to the position of  $TF_1^h\cap TF_2^h=t_1^h\cap p_2^h$ in Fig.~\ref{18}.

In all edge-labelling pairs in $Q$ (Fig.~\ref{16}), the first number is associated to a trapezoid and the second one to a parallelogram, each number printed in its associated color.

We subdivide the six faces of $Q$ into four quarters each, and label such quarters with the numbers 1 to 4, setting external counterclockwise (or internal clockwise) orientations to the faces, so that the two  numbers corresponding to an edge of any given face yield, via the defined orientation, the labelling pair as defined above. This is represented in the upper center and upper right 3-cubes in Fig.~\ref{16}.

We can identify the cube $Q$ with the union of the pairwise distinct intersections $M_i\cap M_j$, ($1\leq i<j\leq 4$). This union becomes formed by the segments:

\begin{align}\label{(7)}\begin{array}{cc}
&^{[(0.5,0.5,0.5),(2.5,0.5,0.5)]_1^2\,[(0.5,2.5,0.5),(2.5,2.5,0.5)]_2^4\,[(0.5,0.5,2.5),(2.5,0.5,2.5)]_2^1\,[(0.5,2.5,2.5),(2.5,2.5,2.5)]_4^2}\\
&^{[(0.5,0.5,0.5),(0.5,2.5,0.5)]_2^3\,[(2.5,0.5,0.5),(2.5,2.5,0.5)]_3^4\,[(0.5,0.5,2.5),(0.5,2.5,2.5)]_3^2\,[(2.5,0.5,2.5),(2.5,2.5,2.5)]_4^3}\\
&^{[(0.5,0.5,0.5),(0.5,0.5,2.5)]_3^1\,[(2.5,0.5,0.5),(2.5,0.5,2.5)]_1^4\,[(0.5,2.5,0.5),(0.5,2.5,2.5)]_1^3\,[(2.5,2.5,0.5),(2.5,2.5,2.5)]_4^1}
\end{array}\end{align}

\noindent where each segment has its trapezoid color number as a subindex and its parallelogram color number as a superindex.
The first, second and third lines in~(\ref{(7)}) display those segments of $Q$ parallel to the first, second and third coordinates, respectively.

\subsection{Automorphism group of the compound}\label{au}

\begin{theorem}\label{antes}
The automorphism group $G=Aut(M)$ of
$M=\bigcup_{i=1}^4 M_i$ is isomorphic to the group of 24 rotations of the 3-cube, that is the symmetry group $Sym_4$ of four elements. Moreover, the 24 reflections of the cube takes $M$ bijectively into an enantiomorph of $M$ also obtainable from $M'$ by rotations.  
\end{theorem}

\begin{proof} 
The symmetries of $M$ coincide with the 24 rotations of the 3-cube $[1,2]^3$, namely the identity plus 9 rotations around the lines joining the centers of opposite faces plus 8 rotations around the lines joining opposite vertices plus 6 rotations around the lines joining the centers of opposite edge. On the other hand. the 24 reflections of $M$ through the center $(1.5,1.5,1.5)$ of $[1,2]^3$ takes $M$ into na enantiomorph of $M$ isomorphic to $M'$.  
\end{proof}

\begin{figure}[htp]
\hspace*{8mm}
\includegraphics[scale=0.7]{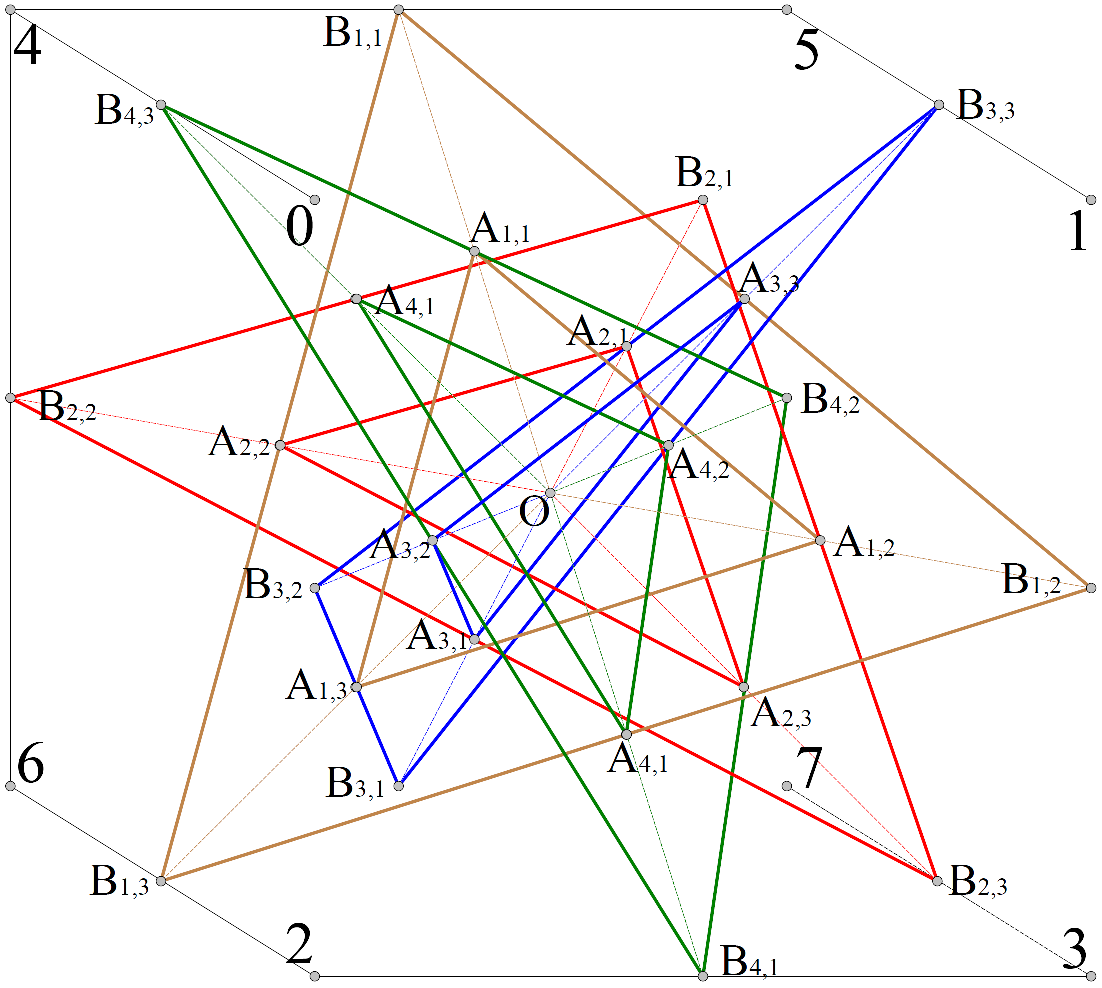}
\caption{Odom-Coxeter polylink.}
\label{19}
\end{figure}

\subsection{Polylink of four hollow triangles}\label{poly}

A sculpture by G. P. Odom Jr., see Fig.~\ref{19}, was analyzed by H. S. M. Coxeter \cite{Coxeter}, for its geometric and symmetric properties.
According to \cite[~p. 270]{Doris}, Odom and Coxeter were unaware of the earlier discovery \cite{Holden} of this  by A. Holden, who called it a {\it regular polylink  of four locked hollow triangles} \cite{Holden2}.

We relate the top M\"obius-strip compound above to this structure, noting that
the centers of the maximum linear parts of the PL trefoil knots $C_i$ ($i=1,2,3,4$) are the vertices of four corresponding equilateral triangles, namely (in the order of the triangle colors):
\begin{align}\label{10}\begin{array}{cccc|cccc}
1&	(1.5,0,3),&(3,1.5,0),&(0,3,1.5);&
2:& 	(1.5,0,0),&(0,1.5,3),&(3,3,1.5);\\
3:& 	(1.5,3,3),&(0,1.5,0),&(3,0,1.5);&
4:& 	(1.5,3,0),&(3,1.5,3),&(0,0,1.5).
\end{array}\end{align}

\noindent Each of these triangles $T_i$ ($i=1,2,3,4$), gives place to a hollow triangle (i.e., a planar region bounded by two homothetic and concentric equilateral triangles \cite{Coxeter}) by removing from $T_i$ the equilateral triangle $T'_i$ whose vertices are the midpoints of the segments between the vertices of $T_i$ (display~(\ref{10})) and $O=(1.5,1.5,1.5)$. Characterized by colors, these midpoints are

\begin{align}\label{11}\begin{array}{cccc}
1:&	(1.50,0.75,2.25),&(2,25,1.50,0.75),&(0.75,2.25,1.50);\\
2:&	(1.50,0.75,0.75),&(0.75,1.50,2.25),&(2.25,2.25,1.50);\\
3:&	(1.50,2.25,2.25),&(0.75,1.50,0.75),&(2.25,0.75,1.50);\\
4:&	(1.50,2.25,0.75),&(2.25,1.50,2.25),&(0.75,0.75,1.50).
\end{array}\end{align}

The centers in display~(\ref{10}) are the vertices of an Archimedean cubocta\-he\-dron. Consider the midpoints of the sides of the triangles $T_i$, namely:

\begin{align}\label{12}\begin{array}{cccc}
1:&(2.25,0.75,1.50),&(1.50,2.25,0.75),&(0.75,1.50,2.25);\\
2:&(0.75,0.75,1.50),&(1.50,2.25,2.25),&(2.25,1.50,0.75);\\
3:&(0.75,2.25,1.50),&(1.50,0.75,0.75),&(2.25,1.50,2.25);\\
4:&(2.25,2.25,1.50),&(1.50,0.75,2.25);&(0.75,1.50,0.75).
\end{array}\end{align}

By expressing the 3-tuples in (\ref{11}) via a $4\times 3$-matrix $\{A_{i,j};i=1,2,3;j=1,2,3,4\}$, we have the following correspondence from (\ref{11}) to (\ref{12}), in terms of the notation in Fig.~\ref{19}:

\begin{align}\label{(13)}
\begin{pmatrix}
A_{1,1}&A_{1,2}&A_{1,3}\\
A_{2,1}&A_{2,2}&A_{2,3}\\
A_{3,1}&A_{3,2}&A_{3,3}\\
A_{4,1}&A_{4,2}&A_{4,3}
\end{pmatrix}
\rightarrow
\begin{pmatrix}
B_{1,12}&B_{1,23}&B_{1,31}\\
B_{2,12}&B_{2,23}&B_{2,31}\\
B_{3,12}&B_{3,23}&B_{3,31}\\
B_{4,12}&B_{4,23}&B_{4,31}
\end{pmatrix}
=
\begin{pmatrix}
A_{3,3}&A_{4,1}&A_{2,2}\\
A_{4,3}&A_{3,1}&A_{1,2}\\
A_{1,3}&A_{2,1}&A_{4,2}\\
A_{2,3}&A_{1,1}&A_{3,2}
\end{pmatrix}
\end{align}
meaning
each midpoint $B_{i,jk}$ of a $T_i$ between $B_{i,j}$ and $B_{i,k}$ equals a vertex $A_{i',j'}$ of some $T'_{i'}$, and vice-versa, as in Fig.~\ref{19}.
Take each $T_i$ as the 6-cycle of its vertices and side midpoints:

\begin{align}\label{(14)}\begin{array}{ccccccc}
T_1=&(B_{1,1}&A_{3,3}&B_{1,2}&A_{4,1}&B_{1,3}&A_{2,2});\\
T_2=&(B_{2,1}&A_{1,2}&B_{2,3}&A_{3,1}&B_{2,2}&A_{4,1});\\
T_3=&(B_{3,1}&A_{1,3}&B_{3,2}&A_{2,1}&B_{3,3}&A_{4,2});\\
T_4=&(B_{4,1}&A_{1,1}&B_{4,2}&A_{2,3}&B_{4,1}&A_{3,2}).
\end{array}\end{align}

\begin{corollary} The automorphism group of the union $\bigcup_{i=1}^4T_i$ is $G=Aut(\bigcup_{i=1}^4M_i)=Aut(\bigcup_{i=1}^4T_i)$. In addition, $G_i=Aut(M_i)=Aut(T_i)$, for each $i=1,2,3,4$, is isomorphic to the dihedral group of six elements. Moreover, there are only two polylinks of four hollow triangles in $[0,3]^3$, including $\bigcup_{i=1}^4T_i$, and these two polylinks are enantiomorphic.\end{corollary}

\begin{proof} By expressing, as in Fig.~\ref{19}, the vertices of $[0,3]^3$ by:
\begin{align}\label{(15)}\begin{array}{cccccccc}
\mathbf{0}=000,&\mathbf{1}=300,&\mathbf{2}=030,&\mathbf{3}=330,\mathbf{4}=003,&\mathbf{5}=303,&\mathbf{6}=033,&\mathbf{7}=033,
\end{array}\end{align}
we notice that the $\pm120^\circ$ angle rotations of $[0,3]^3$ around the axis line determined by the two points of color $i$ in $Q$, as indicated in Fig.~\ref{15},
 correspond respectively to the permutations:
\begin{align}\label{(16)}\begin{array}{clll}
 Color\;1:&R_1=(\mathbf{124})(\mathbf{365})&\mbox{ and }&R_1^{-1}=(\mathbf{142})(\mathbf{563});\\
 Color\;2:&R_2=(\mathbf{036})(\mathbf{174})&\mbox{ and }&R_2^{-1}=(\mathbf{063})(\mathbf{147});\\
 Color\;3:&R_3=(\mathbf{065})(\mathbf{271)}&\mbox{ and }&R_3^{-1}=(\mathbf{056})(\mathbf{217});\\
 Color\:4:&R_4=(\mathbf{247})(\mathbf{053})&\mbox{ and }&R_4^{-1}=(\mathbf{274})(\mathbf{035}),
 \end{array}\end{align}
 where color 1 is red, color 2 is blue, color 3 is green and color 4 is hazel; notice also that the axes in $Q$ corresponding to these colors are:
 $$\begin{array}{ccc}&[(0.5,0.5,0.5),(2.5,2.5,2.5)],& [(0.5,2.5,0.5),(2.5,0.5,2.5)],\\
& [(0.5,0.5,2.5),(2.5,2.5,0.5)],& [(2.5,0.5,0.5),(0.5,2.5,2.5)].\end{array}$$
The rest of the statement arises from Theorem~\ref{antes}. The fact that there are two polylinks of four hollow triangles in $[0,3]^3$ that are enantiomorphic is inherited from Theorem~\ref{antes}.
\end{proof}


\begin{thebibliography}{99}

\bibitem{s1} S. Ali, M. K. Jamil, M. Azeem, M. A. Zahid, T. A. Ismail, {\it Double resolving sets and the exchange property with applications in network optimization and cybersecurity}, Materials Chemistry and Physics, (2025) DOI: 10.1016/j.matchemphys.2025.131289

\bibitem{s2} M. Azeem, {\it Cycle-super magic labeling of polyomino linear and zig-zag chains}, Journal of Operations Intelligence, {\bf 1.1}.(2023): 67--81. https://doi.org/10.31181/jopi1120235.

\bibitem{Armanios} C. Armanios, {\it A new 5--valent distance transitive graph}, Ars Combin. {\bf 19A} (1985), 77--85.

\bibitem{BNC} A. Brouwer, A, Neumaier,  A. Cohen, Distance-Regular Graphs, Springer-Verlag, 1998.

\bibitem{CD}
C.  J. Colbourn,  J. H. Dinitz, Handbook of Combinatorial Designs (2nd ed.), Boca Raton: Chapman \& Hall/ CRC, (2007).

\bibitem{CS} J. H. Conway,  N. J. A. Sloane, Sphere Packings, Lattices and Groups, (3rd ed.), Berlin, New York: Springer-Verlag, (1999).

\bibitem{Pal} S. L. R. Costa, R. Muniz, E. Agustini, R. Palazzo, {\it Graphs, Tessellations and Perfect Codes in Flat Tori}, IEEE Transactions on Information Theory, {\bf 50-10}(2004), 2363-2377.


\bibitem{self} H. S. M. Coxeter, {\it Self-dual configurations and regular graphs}, Bull. Amer. Math. Soc., {\bf 56(5)} (1950), 413--455.

\bibitem{Coxeter} H. S. M. Coxeter, {\it Symmetrical Combinations of Three or Four Hollow Triangles}, The Mathematical Intelligencer, {\bf 16}(3) (1994), 25--30.

\bibitem{Thomsen} H. S. M. Coxeter,  {\it Self-Dual Configurations and Regular Graphs}, Bull. Amer. Math. Soc. {\bf 56} (1950), 413--455.

\bibitem{Crom} P. R. Cromwell, Polyhedra, Cambridge Univ. Press, Cambridge UK (1997).

\bibitem{DS} I. J. Dejter,  O. Serra, {\it Efficient dominating sets in Cayley graphs},
Discrete Applied Mathematics {\bf 129}(2--3) (2002), 319--328.

\bibitem{Eiben} E. Eiben, R. Jajcay, P. ’\v{S}parl, {\it Symmetry properties of generalized graph truncations}, J. Combin. Theory, Ser. B {\bf 137} (2019) 29--131.

\bibitem{Fuchs} D. Fuchs,  S. Tabachnikov, Mathematical Omnibus: Thirty Lectures on Classic Mathematics, 2007, 199--206, at
  http://www.math.psu.edu/tabachni/Books/taba.pdf

\bibitem{s3} G. Haidar, S. Ali, W. Khan, {\it Fault-tolerant topological indices of graph joins with applications in multi-swarm drone networks}, Computer Networks, {\bf 270}(2025) DOI: 10.1016/j.comnet.2025.111517

\bibitem{HV}  D. Hilbert,  S. Cohn-Vossen, Geometry and the Imagination (2nd ed.), Chelsea, 1952.

\bibitem{Holden} A. Holden, Shapes, Spaces and Symmetry, Columbia Univ. Press, 1971.

\bibitem{Holden2} A. Holden, {\it Regular Polylinks}, Structural Topology, {\bf 4} (1980),  41--45.

\bibitem{Imrich} W. Imrich, S. Klav\v{z}ar,  D. F. Rall, Graphs and their cartesian products, A. K. Peters, (2008).

\bibitem{s4} A. Khan, S. Ali, S. Hayat, M. Azeem, Y. Zhong, M. A. Zahid, M. J.F. Alenazi, {\it Fault-tolerance and unique identification of vertices and edges in a graph: The fault-tolerant mixed metric dimension}, Journal of Parallel and Distributed Computing, {\bf 197} (2025), https://doi.org/10.1016/j.jpdc.2024.105024.

\bibitem{census} P. Poto\v{c}nik, P. Spiga, G. Verret, {\it A census of cubic vertex-transitive graphs}

\url{https://staff.matapp.unimib.it/~spiga/census.html}.

\bibitem{PV} P. Poto\v{c}nik,  J. Vidali, {\it Girth-regular graphs}, Ars Mathematica Contemporanea, {\bf 17} (2019), 349--368.

\bibitem{PW} P. Poto\v{c}nik,  S. Wilson, {\it Tetravalent edge-transitive graphs of girth at most 4}, Jour. Combin. Theory, Ser. B, {\bf 97} (2007), 217--236.

\bibitem{ARS} P. Poto\v{c}nik,  S. Wilson, {\it Linking-ring structures and tetravalent semisymmetric graphs}, Ars. Math. Contemporanea, {\bf 7} (2014), 341--352.

\bibitem{Rolfsen} D. Rolfsen, Knots and Links, Publish or Perish  Press, Wilmington DE, 1976

 \bibitem{Rourke} C. P Rourke,  B. J. Sanderson,
 Introduction to Piecewise-Linear Topology, Springer, New York, NY, 1972.

\bibitem{Doris} D. Schattschneider,
{\it Coxeter and the Artists:  2-Way Inspiration}, in The Coxeter Legacy, Reflections and Projections, (ed. C. Davis et al.), Amer. Math. Soc., 2006.

\bibitem{s5} Y. Tu, S. Ali, M. Azeem, M. Arshad, G. Haidar, H. K. Thabet, {\it Optimizing emergency response services in urban areas through the fault-tolerant metric dimension of hexagonal nanosheet}, Scientific Reports, {\bf 15(1)} (2025) DOI:10.1038/s41598-025-16684-0.

\bibitem{WDW} W. D. Wallis, One-Factorizations (Math. and its Appl.), Springer, NY, 1997.

\bibitem{umkb} S. Wilson, {\it Uniform maps on the Klein bottle}, Jour. for Geom. and Graphics, {\bf 10}(2) (2006), 161--171.

\bibitem{WP}  S. Wilson,  P. Poto\v{c}nik, {\it Recipes for edge-transitive tetravalent graphs}, Art Discrete Appl. Math., {\bf 3} (2020), 33 pp.

\bibitem{s6} S. N. A. Zamri, S. Ali, M. Azeem, H. A. Neamah, B. Almohsen, {\it The Mixed Partition Dimension: A New Resolvability Parameter in Graph Theory}. IEEE Access (2025). DOI: 10.1109/ACCESS.2025.3534819.

\end{thebibliography}
\end{document}